\documentclass[11pt]{article}

\usepackage{amsthm}
\usepackage{hyperref}
\usepackage{amssymb}
\usepackage{latexsym}
\usepackage{amsmath}
\usepackage{amsfonts}
\usepackage{version}
\usepackage[dvips]{graphicx}
\usepackage[LGR,T1]{fontenc}%
\usepackage[francais,english]{babel}
\usepackage{url}
\usepackage{enumerate}
\usepackage{hyperref}
\usepackage{color}
\usepackage{bm}
\usepackage{empheq}
\usepackage{framed}
\usepackage[colorinlistoftodos,prependcaption,textsize=tiny]{todonotes} 
\usepackage{float}
\usepackage{subcaption} 
\usepackage[linesnumbered,ruled,vlined,titlenotnumbered]{algorithm2e}

\newtheorem{theorem}{Theorem}
\newtheorem{definition}[theorem]{Definition}

\newtheorem{proposition}[theorem]{Proposition}
\newtheorem{remark}[theorem]{Remark}

\def \RR{\mathbb{R}}
\def \E{\mathbb{E}}

\def\Nc{{\cal N}}

\def\Law{\mathcal L}
\def\be{\begin{eqnarray}}
\def\ee{\end{eqnarray}}
\def\b*{\begin{eqnarray*}}
\def\e*{\end{eqnarray*}}
\def\reff#1{{\rm(\ref{#1})}}
\newcommand{\mbf}{\mathbf{m}}
\newcommand{\ebf}{\mathbf{e}}

\numberwithin{question}{section}

\def\ind{\mathrm{ind}}
\def\grad{\nabla}

\title{Large Banking Systems with Default and Recovery: \\ A Mean Field Game Model}

\vspace{1cm} 

\author{ 
ROMUALD ELIE 
\thanks{~Universit\'e Paris-Est Marne-la-Vall\'ee, France (E-mail:    {\it romuald.elie@univ-mlv.fr}).}
\and 
TOMOYUKI ICHIBA 
\thanks{~Department of Statistics and Applied Probability, South Hall, University of California, Santa Barbara, CA 93106, USA (E-mail:    {\it ichiba@pstat.ucsb.edu}). Research supported in part by the National Science Foundation under grants NSF-DMS-13-13373 and DMS-1615229
          } 
\and 
MATHIEU LAURIERE 
\thanks{~Department of Operations Research and Financial Engineering, Sherrerd Hall, Princeton University, Princeton, NJ 08540, USA (E-mail:    {\it lauriere@princeton.edu}).}
}

\date{}

\begin{document}

\maketitle


\begin{abstract}
We consider a mean-field model for large banking systems, which takes into account default and recovery of the institutions. Building on models used for groups of interacting neurons, we first study a McKean-Vlasov dynamics and its evolutionary Fokker-Planck equation in which the mean-field interactions occur through a mean-reverting term and through a hitting time corresponding to a default level. The latter feature reflects the impact of a financial institution's default on the global distribution of reserves in the banking system. The systemic risk problem of financial institutions is understood as a blow-up phenomenon of the Fokker-Planck equation. Then, we incorporate in the model an optimization component by letting the institutions control part of their dynamics in order to minimize their expected risk. Phrasing this optimization problem as a mean-field game, we provide an explicit solution in a special case and, in the general case, we report numerical experiments based on a finite difference scheme.
\end{abstract}


\section{Introduction}

Financial institutions form a highly connected network through monetary flow and complex dependencies. Each institution is trying to maximize its expected return objective over time, while the aggregation of all investment strategies generates feedback loops and results in some overall patterns of the financial market. The institutions are all competing against each other as players in a financial interacting game. When the number of players becomes large as we can observe in the current fully connected worldwide banking system, individual interactions become intractable while the global patterns become more apparent. Our goal in this paper is to capture the origins of these patterns in a simple mathematical and numerical setup, described by a mean-field game banking system, taking into account births and defaults of  financial institutions.

The set up of the birth and default dynamics considered here is inspired by the neuron firing models \cite{MR2853216,MR3322871,MR3349003,MR3411723} as well as mathematical physics models \cite{nadtochiy2018mean,MR3910001}; see also \cite{hambly2018mckean,hambly2018spde,kaushansky2018semi,ledger2018mercy}. These models aim at describing how the electric potential of a network of neurons evolves over time. A salient feature occurs whenever the potential of one single neuron reaches a given threshold level, leading to a so-called firing event where all the energy accumulated is transmitted to neighboring neurons in the network. In these models, it has been observed that a few parameters characterize some sort of phase transition period, according to which steady states or blow-up phenomena may occur or not. 
A blow-up typically corresponds to the situation in which the firing of one neuron immediately triggers the firing of some other neurons and so on, resulting in a cascade of firings. In the macroscopic limit, this translates into a jump in the instantaneous rate of firings. Handling in a mathematically rigorous way such phenomenon is an active and very challenging research area, in which recent progress has been made for instance in~\cite{delarue2019stefan} in connection with the aforementioned line of work.

Drawing an analogy between neurons firing and banks defaulting, the blow-up phenomenon appears as a natural tool for describing systemic patterns of a financial crisis: complicated interactions between poorly regulated banks driven by selfish objectives can sometimes lead to cascade of defaults. 
This line of modeling has already been pointed out e.g. in~\cite{talkDelarue2015,MR3910001} to describe the evolution of a financial system. Building on this type of model, our main contribution in this paper is to incorporate a proper optimization component for each bank, by letting each bank influence the dynamics of its state so as to minimize a chosen idiosyncratic risk criterion, in the spirit of \cite{MR3325083}. \\

More specifically, we focus here on the mean-field limit of the following toy model: The banking system consists of  $\,N\,$ banks, where each bank~$\,i\,$ owns liquid assets as monetary reserve evolving continuously in time. Let denote by $\,(X_{t}^{i})_{t\ge 0}$ the non-negative level of cash reserve (in liquid assets) of bank $\,i\,$, for $\,i \, =\, 1, \ldots , N\,$ and introduce $\, (\overline{X}_{t})_{t\ge 0} \,$ the average level of cash reserves over the full banking system, i.e., $\, \overline{X}_{t} \, :=\, (X_{t}^{1} + \cdots + X_{t}^{N}) \, / \, N\,$ at time $\,t \ge 0\,$. 

We assume for simplicity that the diffusion dynamics of each cash reserve $\,X^{i}\,$ is of mean field type, i.e. it depends on the global banking system through the empirical distribution of the cash reserve levels, and more specifically through their empirical average $\, \overline{X}\,$. More precisely, the diffusive behavior of $\,X^{i}\,$ at time $t$ is locally specified by a drift function $\, b : \mathbb R^{2}_{+} \to \mathbb R\,$ of its value $\, X_{t}^{i}\,$ together with the running average $\, \overline{X}_{t}\,$ of the system. The sample path $\, t\to X_{t}^{i} \,$ is assumed to be right continuous with left limits.

The bank defaults are modeled in the following way: if the amount $\, X_{t}^{i_{0}} \,$ of liquid assets of  bank  $\,i_{0} \,$ reaches the threshold level $0$ at time $\,t_{0}\,$, this bank $\,i_{0}\,$ is defaulted from the system. A systemic effect induces a financial shock to each institution of the system, so that the cash reserve $X^j$ of every bank $\,j (\neq i_{0}) \,$ suffers at time ${t_{0}-}$ a downward jump of size $\, \overline{X}_{t_{0}-} / N \,$. Instantaneously and for ease of modeling, a new institution is also created in the market with cash reserve at the average level $\overline{X}_{t_{0}-}$, so that the number of banks in the financial system remains constant. 
 For notational simplicity, this new bank keeps the same number $\,i_{0} \,$, i.e., $\, X_{t_{0}-}^{i_{0}} \, =\, 0 \, , X_{t_{0}+}^{i_{0}} \, :=\, \overline{X}_{t_{0}}\, $. With initial configuration $\, x_{0} \, :=\, (X_{0}^{1}, \ldots , X_{0}^{N}) \in (0, \infty)^{N}\,$, the resulting dynamics may be depicted by the following system of stochastic differential equations 
\begin{equation} \label{eq: PS} 
\begin{cases}
\displaystyle
	X_{t}^{i} \, &=\, \displaystyle 
	X_{0}^{i} + \int^{t}_{0} b (X_{s}^{i} , \overline{X}_{s}) {\mathrm d} s + \sigma W_{t}^{i} + \int^{t}_{0} \overline{X}_{s-} \Big( {\mathrm d} M_{s}^{i} - \frac{1}{\, N\, } \sum_{j \neq i} {\mathrm d} M_{s}^{j} \Big) \, ; \quad t \ge 0 \, ,  
	\\
\displaystyle
	M_{t}^{i} \, &:=\, \displaystyle
	\sum_{k=1 }^{\infty} {\bf 1}_{\{\tau_{k}^{i} \le t\}} \, , \quad \tau_{k}^{i} \, :=\, \inf \Big \{ s > \tau_{k-1}^{i} : \, X_{s-}^{i} - \frac{\overline{X}_{s-}}{N} \sum_{j \neq i } (M_{s}^{j} - M_{s-}^{j}) \, \le \,  0 \Big\} \, ; \quad k \in \mathbb N \, , 
\end{cases}
\end{equation}
for $\,i \, =\, 1, \ldots , N\,$, where $\, W \, :=\, (W^{1}, \ldots , W^{N})\,$, $\, t\ge 0\,$ is a standard $N$-dimensional Brownian motion, $\, \overline{X}\,$ is the average of $\,X \, :=\, (X^{1}, \ldots , X^{N})\,$, $\, M_{t}^{i}\,$ is the cumulative number of defaults of bank $i$ by time $\,t \ge 0 \,$, and $\,\tau_{k}^{i}\,$ is its $\,k\,$-th default time with $\, \tau_{0}^{i} \, =\, 0\,$. Although, in our toy model, the drift coefficient $\, b\,$ depends on both the state $X$ of a bank together with the average $\overline{X}$ whereas the diffusion coefficient is a positive constant $\sigma$, one can in general consider more complicated dependence on the coefficients, such as time-dependent drift and diffusion coefficients. The system~\eqref{eq: PS} of $\,N\,$ banks includes mean-field interactions through the drift function $\,b (\cdot) \,$ as well as via the cumulative number of defaults $M$. The mathematical analysis of such system raises some delicate and technical issues, due to mean-field interactions together with multiple defaults.\\

We then turn our attention to the more realistic situation in which banks are not passively following the dynamics~\eqref{eq: PS} but are able  to partly control their drifts and seek to minimize the sum of expected running costs (possibly together with a terminal default cost), occurring at default time. Compared with the model~\eqref{eq: PS}, the drift of bank $i$ incorporates a linear dependence on the control $\xi^i_t$ used by this bank. The running cost rate $\,f\,$ at time $\,t \ge 0 \,$  of bank $i\,$ depends on the control $\,\xi^{i}_{t}\,$, the monetary reserve $\,X^{i}_{t}\,$ as well as the empirical distribution $\, m_{N, t}(\cdot) \, :=\, \sum_{i=1}^{N} \delta_{X_{t}^{i}} (\cdot)\, /\, N\, $ of the system. This allows in particular to penalize strong deviations from the average of the banking system. Here, we borrow the running cost functional from a model introduced in~\cite{MR3325083} for systemic risk: The running cost functional takes the form of a quadratic function. Moreover, the running cost is discounted at rate $\,r\,$. Thus, over the time interval $\,[0, T]\,$, each bank $\,i \,$ seeks to minimize 
\[
	\mathbb E \left[ \int^{\tau^{i} \wedge T}_{0} e^{-r s} f(X^{i}_{s}, m_{N, s}, \xi^{i}_{s} ) {\mathrm d} s  \right] \, ,
\]
where $X$ is controlled by $\xi$ via its drift. In this context, we look for a Nash equilibrium, in the sense that when bank $\,i\,$ performs the optimization, the controls used by the other banks $\,j \neq i\,$ are fixed. This modeling of systemic risk hence identifies to an $\,N\,$-player stochastic game, as the other players policy and default rate modify the banking system dynamics, under which each bank optimizes its policy and resulting monetary reserve. 
As the size of the system becomes large, i.e., $\,N \to \infty \,$, the difficulty and complexity of the mathematical analysis increase rapidly. In order to obtain a tractable approximation of a large banking system, we rely on the mean field game (MFG) paradigm.  Mean field games were introduced by Lasry and Lions~\cite{PLLCDF,MR2295621,MR2269875,MR2271747} and Caines, Huang and Malham\'e~\cite{HuangCainesMalhame-2003-individual-mass-wireless,MR2352434-HuangCainesMalhame-2007-LQG,MR2346927-HuangCainesMalhame-2006-closedLoop}. The interested reader is referred to the recently published books \cite{MR3134900,MR3752669,MR3753660} and the references therein. \\

The rest of the paper is organized as follows. In section~\ref{sec: 2}, we first show that the $N$-agent system~\eqref{eq: PS} is well defined using a notion of physical solution borrowed from~\cite{MR3322871}. We then informally derive the limiting mean-field system as $\,N \to \infty\,$, as well as the corresponding nonlinear Fokker-Plank equation.  In the limiting system, we show in Theorem~\ref{thm: Blowup} that blow-up may occur when the initial distribution is concentrated near the origin. This result shades in particular a new light on the blow-up phenomenon described in \cite{MR2853216}. Importantly, we observe that the Fokker-Plank system has an explicit stationary solution, as derived in  Theorem~\ref{thm:characterize-statio-p-e0}.  We conclude section~\ref{sec: 2} with numerical results for the dynamics of the system. In section~\ref{sec: 3}, we incorporate the optimization component. We formulate the mean-field game in section~\ref{sec: 3.1}. Its solution is characterized by a Hamilton-Jacobi-Bellman (HJB) equation coupled with a Fokker Plank (FP) equation with well suited boundary conditions at the boundary of the domain. In section~\ref{sec: 3.3}, we derive an explicit solution for a stationary mean field game from the PDE system, where Theorem~\ref{thm:characterize-statio-p-e0} is used to determine the probability distribution. This stationary solution serves as a baseline for our numerical study. In section~\ref{sec: 3.4}, we provide a discrete numerical scheme for the HJB-FP system, building on~\cite{MR2679575}. Finally, numerical results in a non-stationary regime are presented in section~\ref{sec: 3.5}.


\section{Fokker Planck equation for the particle system}
\label{sec: 2}

\subsection{Construction of Physical Solution} 

\subsubsection{Modeling bank interactions and defaults}

Throughout this section, we assume for simplicity that $\sigma=1$, and in this subsection, let us assume that $\,b(\cdot, \cdot) \,$ in (\ref{eq: PS}) is (globally) Lipschitz continuous on $\, \mathbb R^{2}_{+}\,$, i.e., there exists a constant $\, \kappa > 0 \,$ such that 
\begin{equation} \label{eq: BL}
\lvert b(x_{1}, m_{1}) - b(x_{2}, m_{2})  \rvert \le \kappa \big ( \lvert x_{1} - x_{2} \rvert + \lvert m_{1} - m_{2} \rvert \big) \, 
\end{equation}
for all $\,x_{1}, x_{2}, m_{1}, m_{2} \in \mathbb R_{+}\,$, and impose the following condition on the drift function $\,b(\cdot, \cdot)\,$: 
\begin{equation} \label{eq: DC}
 \sum_{i=1}^{N} b( x^{i}, \overline{x}) \, \equiv\,  0 \, 
\end{equation}
for every $\, x := (x^{1}, \ldots , x^{N}) \in \mathbb R_{+}^{N}\,$ and $\, \overline{x} \, :=\, (x^{1} + \cdots + x^{N}) \, / \, N\,$.  This condition holds for instance if $b( x^{i}, \overline{x}) = x^{i} - \overline{x}$ is a linear mean-reverting drift, and naturally translates a global stability of the monetary level for the whole financial system.\\ 

Given a standard Brownian motion $\, W_{\cdot}\,$, we shall consider a system $\, ( X_{\cdot} \, :=\, (X_{\cdot}^{1}, \ldots , X_{\cdot}^{N}), M_{\cdot} \, :=\, (M_{\cdot}^{1}, \ldots , M_{\cdot}^{N})) \,$ described by (\ref{eq: PS}) together with conditions (\ref{eq: BL})-(\ref{eq: DC}) on a filtered probability space $\,(\Omega, \mathcal F, \mathbb F, \mathbb P)\,$, with filtration $\, \mathbb F \, :=\, (\mathcal F_{t}, t \ge 0)\,$. In particular, we are concerned with identifying cases where dynamics (\ref{eq: PS}) might induce multiple simultaneous defaults with positive probability, i.e., 
\[
\mathbb P \big ( \exists (i, j), \, \,\exists t \in [0, \infty) \, \,  \text{ such that } \, \, X_{t}^{i} \, =\, X_{t}^{j} \, =\,  0 \big) > 0 \, . 
\]

In the dynamics (\ref{eq: PS}) of $\,X_{\cdot}\,$ the last term containing the counting process $\,M_{\cdot}\,$ describes how those banks behave at the default event times. In principle, there are other specifications of their behaviors. For example, replacing $\,1/N\,$ by $\, 1/ (N-1)\,$ in (\ref{eq: PS}), we observe the dynamics of $\, X_{\cdot} \, =\, (X_{\cdot}^{1}, \ldots , X_{\cdot}^{N})\,$ becomes 
\begin{equation} \label{eq: PSa}
{\mathrm d} X_{t}^{i} \, =\, b (X_{t}^{i} , \overline{X}_{t}) {\mathrm d} t + {\mathrm d} W_{t}^{i} + \overline{X}_{t-} \Big( {\mathrm d} M_{t}^{i} - \frac{1}{\, N-1\, } \sum_{j \neq i} {\mathrm d} M_{t}^{j} \Big)
\end{equation}
for $\,i \, =\, 1, \ldots , N\,$, $\, t \ge 0\,$, and by direct calculations we may verify that under (\ref{eq: DC}) the average process $\, \overline{X}_{\cdot}\,$ of $\, X_{\cdot}\,$ in (\ref{eq: PSa}) moves as a Brownian motion $\, \overline{X}_{\cdot} \, = \, \overline{X}_{0} + (1 \, /\,  N)\sum_{i=1}^{N} W_{\cdot}^{i} \,$. If we consider a (weak) solution on a probability space $\, (\Omega, \mathcal F, \mathbb F, \mathbb P ) \,$ of the system $\, X_{\cdot}\,$ with dynamics (\ref{eq: PSa}) which takes values in $\,[0, \infty)^{N}\,$ only, then  the first passage time $\, \inf \{ t \ge 0 : \overline{X}(t) \, =\,  0 \} \,$ of $\,0\,$ for $\, \overline{X}_{\cdot}\,$ is finite almost surely and hence, the $\,N\,$-dimensional process $\, X_{\cdot}\,$ in (\ref{eq: PSa}) hits the origin almost surely under (\ref{eq: DC}), i.e., 
\begin{equation} \label{eq: MD}
\mathbb P \big ( \exists t \in [0, \infty) \, \, \text{ such that } \, \, X_{t}^{1} \, =\, \cdots \, =\, X_{t}^{N} \, =\, 0 \big ) \, =\,  1 \, . 
\end{equation}
Another example with such property occurs with the following dynamics for $\,i \, =\, 1, \ldots , N\,$ 
\begin{equation} \label{eq: PSb} 
{\mathrm d} X_{t}^{i} \, =\, b (X_{t}^{i} , \overline{X}_{t}) {\mathrm d} t + {\mathrm d} W_{t}^{i} + \overline{X}_{t-} \Big( {\mathrm d} M_{t}^{i} - \frac{1}{\, N\, } \sum_{j =1}^{N} {\mathrm d} M_{t}^{j} \Big) 
\end{equation}
with (\ref{eq: DC}). Again, this dynamics implies that  the average process $\, \overline{X}_{\cdot}\,$ is a Brownian motion, and hence the $\,N\,$-dimensional process $\,X_{\cdot}\,$ in (\ref{eq: PSb}) entails the property (\ref{eq: MD}) with the above reasoning. 

On the other hand, in the dynamics (\ref{eq: PS}) the average process $\, \overline{X}_{\cdot}\,$ jumps up at the default event times as we observe 
\begin{equation} \label{eq: XBAR}
\,  \overline{X}_{\cdot}\, =\,  \overline{X}_{0} + \frac{\,1\,}{\,N\,}\sum_{i=1}^{N} W_{\cdot}^{i} + \frac{\,1\,}{\,N^{2}\,}\sum_{i=1}^{N} \int^{\cdot}_{0}\overline{X}_{s-} {\mathrm d} M_{s}^{i} \, =\, \overline{X}_{0} + \overline{W}_{\cdot} + \frac{\,1\,}{\,N\,} \int^{\cdot}_{0} \overline{X}_{s-} {\mathrm d} \overline{M}_{s}, 
\end{equation}
where $\, \overline{M}_{\cdot}\,$ and $\, \overline{W}_{\cdot}\,$ are the sample averages of $\, M_{\cdot}\,$ and $\, W_{\cdot}\,$, respectively. 
In between default times, $\,\overline{X}_{\cdot}\,$ behaves as the $\,N\,$-dimensional diffusion with drift $\,b (\cdot, \cdot)\,$ and unit diffusion coefficients. When there is a jump (default) in the process $\, (X_{\cdot}, M_{\cdot}) \,$ at time $\,t\,$, the jump size $\,X_{t}^{i}\,$ of defaulted bank $\,i\,$ is given by 
\begin{equation} \label{eq: JS} 
X_{t}^{i} - X^{i}_{t-} \, =\, \overline{X}_{t-} \cdot \Big( M_{t}^{i} - M^{i}_{t-} - \frac{\,1\,}{\,N\,} \sum_{j \neq i} \big( M^{j} _{t} - M^{j}_{t-}\big)  \Big) \, ; \quad i \, =\, 1, \ldots , N \, . 
\end{equation}
Economically speaking, in case of defaults, the creation of a new financial institution requires additional funding from another global financial agency (e.g., government) outside the system.

\subsubsection{Connections with the Neuron firing model of Delarue et al. \cite{MR3349003,MR3322871}}

 Up to a well chosen transformation, we now observe that the dynamics of the financial banking system considered here, shares some similarity with the dynamics derived in the firing Neuronal systems of Delarue  et al. \cite{MR3349003,MR3322871}.
  
Note that the value of $\,X_{t} \,$ lies in the state space $\, [0, \infty)^{N}\,$ for each $\,t \ge 0\,$. Thus if we change the state space from $\, [0, \infty)^{N}\,$ to $\,(-\infty, 1]^N\,$ by transforming $\,X_{t}\,$ to $\, \widehat{X}_{t} \, :=\,  ( \widehat{X}_{t}^{1}, \ldots , \widehat{X}_{t}^{N}) \,$ with $\, \widehat{X}_{\cdot}^{i} \, :=\, (\overline{X}_{\cdot} - X_{\cdot}^{i}) \, /\, \overline{X}_{\cdot} \,$, $\,i \, =\, 1, \ldots , N\,$, then we see  $\, \sum_{i=1}^{N} \widehat{X}_{t}^{i} \, \equiv\,  0 \,$ and by \textsc{It\^o}'s rule the dynamics of $\, \widehat{X}_{\cdot}\,$ is given by 
\begin{equation} \label{eq: XHAT}
{\mathrm d} \widehat{X}_{t}^{i} \, =\,  \Big( - \frac{\, b ( \overline{X}_{t} ( 1 - \widehat{X}_{t}^{i} ), \overline{X}_{t})\, }{ \overline{X}_{t}} + \frac{ \widehat{X}_{t}^{i}}{\, N \overline{X}_{t}^{2}\, }\Big) {\mathrm d} t + {\mathrm d} \widehat{W}^{i}_{t}  - \Big( 1 + \frac{1}{N}\Big) {\mathrm d} \widehat{M}_{t}^{i} + \Big( 1 +  \frac{ 1 - \widehat{X}_{t-}^{i}\, }{N} \Big) \frac{1}{\, N\, } \sum_{j =1}^{N} {\mathrm d} \widehat{M}_{t}^{j}  \, ,  
\end{equation}
\[
{\mathrm d} \widehat{W}_{t}^{i} \, :=\, \Big(  - \frac{1}{\, \overline{X}_{t} \, } \Big) {\mathrm d} W_{t}^{i} + \Big( \frac{\, 1  - \widehat{X}_{t}^{i} \, }{ \, \overline{X}_{t}} \Big) \frac{1}{N} \sum_{j=1}^{N} {\mathrm d} W_{t}^{j}\, \, , \quad 
\widehat{M}_{t}^{i} \, :=\, \sum_{k=1}^{\infty} {\bf 1}_{\{ \widehat{\tau}_{k}^{i} \le t\}} \, \equiv\, M_{t}^{i} \, , 
\]
\[
  \widehat{\tau}_{k}^{i} \, :=\,  \inf \Big \{ t > \widehat{\tau}_{k-1}^{i} : \widehat{X}_{t}^{i} + \frac{1}{\, N\, }\sum_{j\neq i} \big( \widehat{M}_{s}^{j} - \widehat{M}_{s-}^{j}\big) \ge 1 \Big\} \, =\,  \tau_{k}^{i}\, , \quad 
\widehat{\tau}_{0}^{i} \equiv 0 \,;  \quad \, i \, =\, 1, \ldots , N\,, \,k \in \mathbb N\,
\]
until the time $\, \inf \{ t : \overline{X}_{t}\, =\, 0\}\,$. 

This system $\, ( \widehat{X}^{1}_{\cdot}, \ldots , \widehat{X}^{N}_{\cdot}, \overline{X}_{\cdot}) \,$ defined by (\ref{eq: XBAR}) and (\ref{eq: XHAT}) resembles with the particle system $\, \widetilde{X}_{t} \, :=\, (\widetilde{X}_{t}^{1}, \ldots , \widetilde{X}_{t}^{N}) \,$ for neurons studied in 
\cite{MR3349003,MR3322871},    
namely, 
\begin{equation} \label{eq: DIRT}
\begin{split}
{\mathrm d} \widetilde{X}_{t}^{i} \, &=\,  \mathrm b ( \widetilde{X}_{t}^{i}) {\mathrm d} t + {\mathrm d} \widetilde{W}_{t}^{i} -  {\mathrm d} \widetilde{M}_{t}^{i}  + \frac{ \, \alpha \, }{N} \sum_{j=1}^{N} {\mathrm d} \widetilde{M}_{t}^{j} \,  
\end{split}
\end{equation}
for $\,  i \, =\, 1, \ldots , N, \, t \ge 0 \, $ in modeling a very large network of interacting spiking neurons. Here $\, \mathrm b: (-\infty, 1] \to \mathbb R \,$ is a Lipschitz continuous function, $\, \widetilde{W}_{t} := ( \widetilde{W}_{t}^{1}, \ldots , \widetilde{W}_{t}^{N})\,$, $\, t \ge 0 \,$ is the standard Brownian motion and  
\[
\widetilde{M}_{t}^{i} \, :=\, \sum_{k=1}^{\infty} {\bf 1}_{\{ \widetilde{\tau}_{k}^{i} \le t \}} \, , \quad \widetilde{\tau}_{k}^{i} \, :=\,  \inf \Big \{ s > \widetilde{\tau}_{k-1}^{i} : \widetilde{X}_{s-} + \frac{\, \alpha \, }{\, N\, } \sum_{j=1}^{N} ( \widetilde{M}_{s}^{j} - \widetilde{M}_{s-}^{j}) \ge 1 \Big \} \, ; \quad k \in \mathbb N \, 
\]
with $\,\widetilde{ \tau}_{0}^{i} \, =\, 0\,$, $\,i \, =\, 1, \ldots , N\,$ for $\,t \ge 0 \,$. It is known that if the parameter $\, \alpha \,$ lies in $\, (0, 1)\,$, there exists a unique solution (called ``physical solution'') to (\ref{eq: DIRT}), such that 
\[
\widetilde{X}_{t}^{i} \, =\, X_{t-}^{i} + \frac{\, \alpha \,  \lvert \widetilde{\Gamma}_{t} \rvert }{N} \text{ if } i \not \in \widetilde{\Gamma}_{t} \, , \quad \widetilde{X}_{t}^{i} \, =\, X_{t-}^{i} + \frac{\, \alpha \, \lvert \widetilde{\Gamma}_{t} \rvert }{N} - 1 \text{ if } i  \in \widetilde{\Gamma}_{t}  \, , 
\]
where $\, \lvert \widetilde{\Gamma}_{t} \rvert\,$ is the cardinality of $\, \widetilde{\Gamma}_{t}\,$,  a random subset of indexes $\,\{1, \ldots , N\}\,$ defined by the union 
\[
\widetilde{\Gamma}_{t} \, :=\, \bigcup_{0 \le k \le N-1} \widetilde{\Gamma}_{t, k} \, , 
\]
of recursively defined sets $\, \widetilde{\Gamma}_{t, 0} \, :=\, \{ i \in \{1 , \ldots , N\} : X_{t-}^{i} \, =\,  1\}\,$, 
\[
\widetilde{\Gamma}_{t, k+1} \, :=\, \Big \{ i \in \{ 1, \ldots , N \} \setminus \bigcup_{\ell=0}^{k} \widetilde{\Gamma}_{t,\ell} \, :\,  X_{t-}^{i} + \frac{ \alpha\, }{N}  \Big \lvert \bigcup_{ \ell=0}^{k} \widetilde{\Gamma}_{t, k}  \Big \rvert  \ge 1 \Big \}  \, , 
\]
for $\,k \, =\, 0, 1, \ldots , N-2\,$, $\,t \ge 0 \,$. Again here $\, \lvert \, \cdot \, \rvert \,$ represents the cardinality of set.

\subsubsection{Physical solutions}

In a similar spirit, we shall construct a solution to (\ref{eq: PS}) with a specific boundary behavior at default times. Let us define the following map $\, \Phi (x) \, :=\, (\Phi^{1}(x), \ldots , \Phi^{N}(x)): [0, \infty) ^{N} \mapsto [0, \infty)^{N}\,$ and set-valued function $\, \Gamma : \mathbb R_{+}^{N} \to \{ 1, \ldots , N\} \,$ defined by $\,\Gamma_{0}(x) \, :=\, \{ i \in \{1, \ldots , N\} : x^{i} \, =\, 0 \} \,$, 
\[
\Gamma_{k+1} (x) \, :=\, \Big \{ i \in \{1, \ldots , N \} \setminus \bigcup_{\ell=1}^{k} \Gamma_{\ell}(x) \, : \, x^{i} - \frac{\overline{x}}{N}  \cdot  \Big \lvert \bigcup_{\ell=1}^{k} \Gamma_{\ell} (x) \Big \rvert \le 0 \Big\} \, ; \quad k = 0, 1, 2, \ldots , N-3 
\]
\begin{equation}  \label{eq: MAP}
\Gamma(x) \, :=\, \bigcup_{k=0}^{N-2} \Gamma_{k}(x) \, , \quad 
\Phi^{i}(x) \, :=\, x^{i} + \overline{x} \, \Big( \Big( 1 + \frac{1}{\, N\, } \Big) \cdot {\bf 1}_{ \{ i \in \Gamma(x) \}} - \frac{1}{N} \cdot \lvert \Gamma(x) \rvert \Big ) \, 
\end{equation}
for  $\, x \, =\, (x^{1}, \ldots , x^{N}) \in \mathbb R_{+}^{N}\,$, $\,i \, =\,  1, \ldots , N\,$ with $\,\overline{x} \, :=\, (x^{1} + \cdots  + x_{N}) \, / \, N \ge 0\,$. Note that $\, \Phi ([0, \infty)^{N} \setminus \{{\bm 0}\}) \, \subseteq\, [0, \infty)^{N} \setminus \{{\bm 0}\} \,$ and $\, \Phi ({\bm 0}) \, =\, {\bm 0} \, =\, (0, \ldots , 0) \,$.\\ 

Given the initial configuration $\,  X_{0} \, :=\, (X_{0}^{1}, \ldots , X_{0}^{N}) \in (0, \infty)^{N}\,$ and a standard $N$-dimensional Brownian motion $W$, we take the unique strong solution $\, Y_{\cdot}^{1} \, :=\, (Y_{\cdot}^{1, 1}, \ldots , Y_{\cdot}^{1, N}) \,$ to 
\begin{equation} \label{eq: CYa}
Y_{t}^{1, i}  \, =\,  X_{0}^{i} + \int^{t}_{0} b(Y_{s}^{1, i} , \overline{Y}_{s}^{1}) \, {\mathrm d} s + W_{t}^{i} \, ; \quad i \, =\, 1, \ldots , N \,, \, t \ge 0  \, , 
\end{equation}
thanks to the Lipschitz continuity of $\, b(\cdot, \cdot) \,$ as in (\ref{eq: BL}). Here $\, \overline{Y}_{\cdot}^{1} \, :=\, (Y_{\cdot}^{1} + \cdots + Y_{\cdot}^{N}) \, / \, N\,$. Let us define for $\, 0 \le t < \tau^{1}\,$ 
\begin{equation} \label{eq: CXa}
X_{t}^{i} \, :=\, Y_{t}^{1, i} \, , \quad M_{t}^{i} \, :=\, 0 \, , 
\end{equation}
where $\, \tau^{1} \, :=\, \min_{1\le i \le N} \inf \{ s \ge 0 : Y^{1, i}_{s} \, =\,  0 \} \,$, and at $\, \tau^{1}\,$ let us define 
\[
X^{i}_{\tau^{1}} \, :=\,  \Phi^{i} ( Y^{1}_{\tau^{1}}) \, , \quad M_{\tau^{1}}^{i} \, :=\,  {\bf 1}_{\{ \, i \in \Gamma ( Y^{1}_{\tau^{1}}) \, \}} \, , 
\]
where the map $\, \Phi (\cdot) \,$ and $\, \Gamma(\cdot) \,$ are defined in (\ref{eq: MAP}). Then, whenever $\, X_{\tau^{k}} \in [0, \infty)^{N} \setminus \{{\bm 0}\}\,$ for $\,k \, =\, 1, 2, \ldots \,$, we construct  recursively  the unique strong solution $\, Y_{\cdot}^{k+1} \, :=\, ( Y_{\cdot}^{k+1, 1}, \ldots , Y_{\cdot}^{k+1, N}) \,$ to the system of stochastic differential equations 
\begin{equation} \label{eq: CYb}
Y_{t}^{k+1, i}  \, =\,  X_{\tau^{k}}^{i} + \int^{t}_{0} b(Y_{s}^{k+1, i} , \overline{Y}_{s}^{k+1})  \,{\mathrm d} s + W_{t}^{i} \, ; \quad i \, =\, 1, \ldots , N \,, \, t \ge 0  \, , 
\end{equation}
where $\, \overline{Y}^{k+1}_{\cdot}\,$ is the average of elements of $\, Y_{\cdot}^{k+1}\,$, and define 
\[
\tau^{k+1} \, :=\,  \min_{1\le i \le N} \inf \{ s \ge \tau^{k} \, : \, Y_{s}^{k+1, i} \, =\,  0 \} \, , \quad 
X_{t}^{i} \, :=\, Y^{k+1, i}_{t} \, , \quad M_{t}^{i} \, :=\,  M_{\tau^{k}}^{i} \, \quad \text{ for } \tau^{k} \le t < \tau^{k+1} \, , 
\]
\begin{equation} \label{eq: CXb}
X_{\tau^{k+1}}^{i} \, :=\, \Phi^{i}(Y^{k+1, i}_{\tau^{k+1}})  \, , \quad M_{\tau^{k+1}}^{i} \, :=\,  M_{\tau^{k}}^{i} + {\bf 1}_{\{ \,  i \in \Gamma ( Y^{k+1}_{\tau^{k+1}}) \,  \}} \, .   
\end{equation}
If $\, X_{\tau^{k_{0}}} \, =\, {\bm 0}\,$ for some $\, k_{0} < +\infty\,$, we set $\, \tau^{\ell} \, =\, \overline{\tau}_{0} \,$,  for every $\, \ell \ge k_{0}\,$
\begin{equation} \label{eq: BTAU0}
\overline{\tau}_{0} \, :=\,  \inf \{ s > 0 \, : \, \max_{1\le i \le N} X_{s}^{i} \, =\, 0 \} \, =\, \inf\{ s > 0 : \overline{X}_{s} \, =\,  0 \} \, ,
\end{equation}
and stop the process, i.e., $\, X_{t} \, \equiv\, {\bm 0} \,$ for $\,t \ge \overline{\tau}_{0}\,$. This way we construct $\, (X_{\cdot}, M_{\cdot}) \,$ until time $\, \tau^{k} (\le  \overline{\tau}_{0})\,$ for every $\, k \ge 1\,$. 

\begin{proposition} \label{lm: 2.1} Given a standard Brownian motion $\,W_{\cdot}\,$ and the initial configuration $\,X_{0} \in (0, \infty)^{N}\,$ the process $\, ( X_{\cdot}, M_{\cdot})\,$ constructed by this recipe (\ref{eq: CYa})-(\ref{eq: CXb}) is the unique, strong solution to (\ref{eq: PS}) with (\ref{eq: BL}), (\ref{eq: DC}) on $\, [0, \overline{\tau}_{0} ]\,$, such that if there is a default, i.e., $\,\lvert \Gamma ( X_{t-}) \rvert \ge 1\,$ at time $\,t \,$, then the post-default behavior is determined by $\, X_{t}^{i} \, =\,  \Phi^{i}(X_{t-}) \, $ for $\,i \, =\, 1, \ldots , N\,$. 
\end{proposition}

\begin{proof} Because of the similarity of (\ref{eq: PS}) to the particle system (\ref{eq: DIRT}) discussed in \cite{MR3349003},  
we may adopt the main idea of the proof of their Lemma 3.3. Indeed,  at every stopping time $\, \tau^{k} \,$,  we observe that $\,  \lvert \Gamma (X_{\tau^{k}-}) \rvert (\ge 1)\,$ of default events occur, that is, the sample path of the process $\, (X_{\cdot}, M_{\cdot})\,$ has positive jumps. Because of (\ref{eq: MAP}) and the jump sizes of $\, M_{\cdot}^{j}\,$ with 
\[
\sum_{j\neq i} ( M^{j}_{\tau^{k}} - M^{j}_{\tau^{k}-}) \, =\,   \sum_{j\neq i} {\bf 1}_{\{ j \in \Gamma ( Y^{k}_{\tau^{k}}) \} } \, =\,  \lvert \Gamma ( Y^{k}_{\tau^{k}} ) \rvert - {\bf 1}_{ \{\, i \, \in \, \Gamma ( Y^{k}_{\tau^{k}}) \,  \}} \,  , 
\]
the jump sizes of $\, X_{\cdot}\,$ in (\ref{eq: CXb}) given by 
\[
X_{\tau^{k}}^{i} - X^{i}_{\tau^{k}-} \, =\,  \Phi^{i}( Y^{k}_{\tau^{k}}) - Y^{k}_{\tau^{k}} = \overline{Y}_{\tau^{k}} \cdot \Big( \Big( 1  + \frac{\,1\,}{\,N\,} \Big) {\bf 1}_{\{\, i \, \in \,  \Gamma ( Y^{k}_{\tau^{k}})\, \}} - \frac{\,1\,}{\,N\,} \lvert \Gamma ( Y^{k}_{\tau^{k}}) \rvert \Big) 
\]
\[
\, =\,  \overline{X}_{\tau^{k}-} \cdot \Big( M_{\tau^{k}}^{i} - M^{i}_{\tau^{k}-} - \frac{\,1\,}{\,N\,} \sum_{j \neq i} \big( M^{j} _{\tau^{k}} - M^{j}_{\tau^{k}}\big)  \Big) \,  ; \quad k \, =\, 1 , 2 , \ldots ,
\]
are equal to those in (\ref{eq: JS}) induced by (\ref{eq: PS}). In between the stopping times, the solution to (\ref{eq: CYb}) is uniquely determined by the same dynamics as  (\ref{eq: PS}) with (\ref{eq: BL}) on a probability space $\, (\Omega, \mathcal F, \mathbb P) \,$. 
\end{proof}

Note that if $\,N \, =\, 1\,$, then $\, \overline{X}_{\cdot} \, =\, X_{\cdot}\,$ with $\,\sum_{i=1}^{1} b(x_{i}, x_{i}) \equiv 0 \,$ and hence $\, \overline{\tau}_{0} < + \infty\,$ a.s. in Lemma \ref{lm: 2.1}. In general, the probability that the first passage time $\, \overline{\tau}_{0}\,$ of zero for the average process is finite depends on the specification of drift function $\,b(\cdot, \cdot)\,$ in (\ref{eq: PS}). For the system with (\ref{eq: PSa}) or (\ref{eq: PSb}), instead of (\ref{eq: PS}), we may construct the corresponding physical solutions as in Proposition \ref{lm: 2.1}. As we have seen in (\ref{eq: MD}), the first passage time of zero for the average process is finite almost surely in the system (\ref{eq: PSa}). 

\subsection{Mean-Field Approximation}
\subsubsection{Informal Derivation of Mean-Field Limits}

Let us discuss a mean-field approximation of \textsc{McKean-Vlasov} type for the system (\ref{eq: PS}) with (\ref{eq: BL})-(\ref{eq: DC}). In this section let us assume $\,b(x, m) \, =\,  - a(x - m) \,$, $\,x, m \in [0, \infty) \,$ for some $\, a > 0\,$, and assume further that the empirical distribution 
\[
F^{N}_{t}(\cdot) \, :=\, \frac{1}{\, N\, } \sum_{i=1}^{N} \delta_{X^{i}_{t}} \, ; \quad t \ge 0 
\]
converges weakly to a law of process $\, \{ \mathcal X_{t} , t \ge 0 \} \,$ described by 
\begin{equation} \label{eq: MF}
\mathcal X_{t} \, =\, \mathcal X_{0} - a \int^{t}_{0} (\mathcal X_{s} -   \mathbb E[ \mathcal X_{t}] ) {\mathrm d} s + W_{t} + \int^{t}_{0} \mathbb E[ \mathcal X_{s-}]   {\mathrm d} ( \mathcal M_{s} - \mathbb E[ \mathcal M_{s} ] ) \, ; \quad t \ge 0 \, ,  
\end{equation}
where $\,W_{\cdot}\,$ is the standard Brownian motion, $\, \mathcal M_{t} \, :=\, \sum_{k=1}^{\infty} {\bf 1}_{\{\tau^{k} \le t \}} \,$, $\, \tau^{k} \, :=\, \inf \{ s > \tau^{k-1} : \mathcal X_{t-}  \le 0 \} \, $, $\, k \ge 1 \, $, $\, \tau^{0} \, =\,  0 \, $. 
Then taking expectations of both sides of (\ref{eq: MF}), we obtain 
\[
\mathbb E[ \mathcal X_{t}] \, =\, \mathbb E[ \mathcal X_{0} ] \, =: x_{0}\, \, ; \quad t \ge 0 \, .  
\]
When $\, \mathbb E [ \mathcal X_{0}] \, =:\, x_{0} \,$ for some $\,x_{0} > 0\,$, substituting this back into (\ref{eq: MF}), we obtain 
\begin{equation} \label{eq: MFsta}
\mathcal X_{t} \, =\, \mathcal X_{0} - a \int^{t}_{0} (\mathcal X_{s} - x_{0}) {\mathrm d} s + W_{t} + x_{0} ({\mathcal M}_{t} - \mathbb E[ {\mathcal M}_{t}] ) \, ; \quad t \ge 0 \, . 
\end{equation} 
Transforming the state space from $\,[0, \infty)\,$ to $\, (-\infty, 1]\,$ by $\, 
\widehat{\mathcal X}_{t} \, :=\, (x_{0} - \mathcal X_{t}) \, / \, x_{0} \, $, we see 
\begin{equation} \label{eq: MFa}
\widehat{\mathcal X}_{t} \, =\,  - \int^{t}_{0} a \widehat{\mathcal X}_{s} {\mathrm d} s + \widehat{W}_{t} - \widehat{\mathcal M}_{t} + \mathbb E[ \widehat{\mathcal M}_{t}] \, ; \quad t \ge 0 \, , 
\end{equation}
where $\, \widehat{W}_{\cdot} \, =\,  W_{\cdot} \, / \, x_{0}\,$, $\, \widehat{\mathcal M}_{\cdot} \, =\, \mathcal M_{\cdot}	\,$. This nonlinear \textsc{McKean-Vlasov}-type equation can be seen as the mean field limit of the transformed process in (\ref{eq: XHAT}). 

This transformed process $\, \widehat{\mathcal X}_{\cdot}\,$ is  similar to the nonlinear \textsc{McKean-Vlasov}-type stochastic differential equation 
\begin{equation} \label{eq: MFb}
\widetilde{\mathcal X}_{t}\, =\,  \widetilde{\mathcal X}_{0} + \int^{t}_{0} {\mathrm b}( \widetilde{\mathcal X}_{s}) {\mathrm d} s  + \widetilde{W}_{t}  - \widetilde{\mathcal M}_{t} + \alpha \mathbb E[ \widetilde{\mathcal M}_{t}] \, ; \quad t \ge 0 \, , 
\end{equation}
studied in \cite{MR3349003,MR3322871}.  
Here $\,\mathcal X_{0} < 1\,$, $\,\alpha \in (0, 1)\,$, $\, \mathrm b : (-\infty, 1] \to \mathbb R\,$ is assumed to be Lipschitz continuous with at most linear growth. $\, \widetilde{W}_{\cdot}\,$ is the standard Brownian motion, $\, \widetilde{\mathcal M}_{\cdot} \, =\, \sum_{k=1}^{\infty} {\bf 1}_{\{ \widetilde{\tau}^{k} \le \cdot\}}\,$ with $\, \widetilde{\tau}^{k} \, :=\, \inf \{ s > \tau^{k-1} : \widetilde{\mathcal X}_{s-} \ge 1 \} \,$, $\,k \ge 1\,$, $\, \widetilde{\tau}^{0} \, =\, 0\,$. When we specify $\, \widetilde{\mathcal X}_{0} \, =\, 0\,$, $\, {\mathrm b} (x) \, =\, - a x \,$, $\, x \in \mathbb R_{+}\,$, and $\, \alpha \, =\, 1\,$, the solution $\, (\widehat{\mathcal X}_{\cdot}, \widehat{\mathcal M}_{\cdot})\,$ to (\ref{eq: MFb}) reduces to the solution $\,(\widetilde{\mathcal X}_{\cdot}, \widetilde{\mathcal M}_{\cdot})\,$ to (\ref{eq: MFa}), however, the previous study of (\ref{eq: MFb}) does not guarantee the uniqueness of solution to (\ref{eq: MFb}) in the case $\, \alpha \, =\, 1\,$.

\subsubsection{Uniqueness of the Mean Field Limit} 
Following \cite{MR3322871}, 
we may reformulate the solution $\,( \widehat{\mathcal X}_{\cdot}, \widehat{\mathcal M}_{\cdot})\,$ to (\ref{eq: MFa}) by 
\begin{equation} \label{eq: ZHAT}
\widehat{Z}_{t} \, :=\, \widehat{\mathcal X}_{t} + \widehat{\mathcal M}_{t} \, =\, - a  \int^{t}_{0} ( \widehat{Z}_{s} - \widehat{\mathcal M}_{s}) {\mathrm d} s + \widehat{W}_{t} + \mathbb E[ \widehat{\mathcal M}_{t}]  \, , 
\end{equation}
\begin{equation}
\widehat{\mathcal M}_{t} \, =\,  \lfloor \sup_{0 \le s \le t } ( \widehat{Z}_{s})^{+} \rfloor \, ; \quad t \ge 0 \, .  
\end{equation}
Here $\, \lfloor \cdot \rfloor \,$ is the integer part. Given a candidate solution $\,e_{t} \,$ for $\, \mathbb E [ \widehat{\mathcal M}_{t}]\,$, $\,t \ge 0\,$, we shall consider 
\begin{equation} \label{eq: ZHATa}
\widehat{Z}^{e}_{t} \, =\,  -a \int^{t}_{0} ( \widehat{Z}_{s}^{e} - \widehat{\mathcal M}_{s}^{e}) {\mathrm d} s + \widehat{W}_{t} + e_{t} \,  , \quad 
\widehat{\mathcal M}^{e}_{t} \, =\, \lfloor \sup_{0 \le s \le t } ( \widehat{Z}_{s}^{e})^{+} \rfloor \, ; \quad t \ge 0 \, , 
\end{equation}
where the superscripts $\,e\,$ of $\, \widehat{Z}^{e}_{\cdot}\,$ and $\, \widehat{\mathcal M}^{e}_{\cdot}\,$ represent the dependence on $\, e_{\cdot}\,$. Then uniqueness of the solution to (\ref{eq: MFa}) is reduced to uniqueness of the fixed point $\, e^{\ast}_{\cdot }\, =\,  \mathfrak M_{\cdot} (e^{\ast}) \,$ of the map $\,\mathfrak M : C(\mathbb R_{+}, \mathbb R_{+}) \to C (\mathbb R_{+}, \mathbb R_{+}) \,$ defined by 
\begin{equation} \label{eq: ME}
\mathfrak M_{t}(e) \, :=\,  \mathbb E \big [ \lfloor \sup_{0 \le s \le t} ( \widehat{Z}^{e}_{s})^{+} \rfloor \big] \, =\,  \mathbb E [ \widehat{\mathcal M}^{e}_{t}] \, ; \quad t \ge 0 \, . 
\end{equation}
This can be verified by the observation 
\[
\widehat{\mathcal X}_{t}\, =\, \widehat{Z}^{e^{\ast}}_{t} - \widehat{\mathcal M}^{e^{\ast}}_{t} \, =\,  - a\int^{t}_{0} ( \widehat{Z}_{s}^{e^{\ast}} - \widehat{\mathcal M}^{e^{\ast}}_{s}) {\mathrm d} s + \widehat{W}_{t} + e^{\ast}_{t} \, =\,  - a\int^{t}_{0} \widehat{\mathcal X}_{s} {\mathrm d} s + \widehat{W}_{t} + \mathbb E [ \widehat{\mathcal M}_{t}^{e^{\ast}} ] \, 
\]
for every $\, t \ge 0 \,$.  

\subsubsection{Numerical Approximation of Fixed Point}
The map $\, e \to \mathfrak M(e) \,$ in (\ref{eq: ME}) is monotone, in the sense that if $\, e^{1}_{\cdot}, e^{2}_{\cdot} \in C(\mathbb R_{+}, \mathbb R_{+})\, $ with $\,e^{1}_{t} \le e^{2}_{t}\,$ for every $\, t \ge 0\,$, then 
\begin{equation}\, \mathfrak M_{t}(e^{1}) \le \mathfrak M_{t} (e^{2}) \,; \quad \,t \ge 0 \,. 
\end{equation}
With this idea, let us consider the following numerical approximation. Start with $\,e_{\cdot}^{(0)} \equiv 0 \, \,$ and define recursively 
\begin{equation} \label{eq: RE}
e^{(n+1)}_{t} \, =\, \mathfrak M _{t}(e^{(n)}) \, ; \quad  n \ge 0 \, , \quad t \ge 0 \, . 
\end{equation} 
By the definition of the map, $\,t \to e^{(n+1)}_{t}\,$ is strictly increasing for $\,n \ge 0\,$. 
Then by the monotonicity of the map $\, \mathfrak M\,$ and $\, e^{(1)}_{t} \ge e^{(0)}_{t} \equiv 0\,$, $\,t \ge 0\,$, we have 
\begin{equation}
\, e^{(0)}_{\cdot} \le e^{(1)}_{\cdot} \le \ldots \,,  
\end{equation}
and hence, we conjecture that if the limit 
\begin{equation} \label{eq: PC}
e^{\ast}_{t} \, :=\, \lim_{n\to \infty} e^{(n)}_{t} \,  
\end{equation}
exists and is finite for every $\,t \ge 0 \,$, then $\,e^{\ast}_{\cdot}\,$ serves as the fixed point of the map $\, \mathfrak M \,$, that is, 
\[
\mathfrak M_{t}( e^{\ast}) \, =\, e^{\ast}_{t} \, ; \quad t \ge 0 \, . 
\]

To discuss the convergence (\ref{eq: PC}), we shall consider the sup norm $\, \lVert e \rVert_{T} \, :=\, \sup_{0\le s \le T} \lvert e(s)\rvert \,$ for every $\,T > 0\,$ and evaluate 
\begin{equation} \label{eq: EE}
\begin{split}
\lVert e^{(n+2)} - e^{(n+1)} \rVert _{T}\, & =\,  \lVert \mathfrak M_{\cdot}(e^{(n+1)}) - \mathfrak M_{t}(e^{(n)}) \rVert _{T }\, \\
&=\, \sup_{0 \le t \le T} \Big \lvert \mathbb E \big [ \lfloor \sup_{0 \le s \le t} ( \widehat{Z}^{e^{(n+1)}}_{s})^{+} \rfloor - \lfloor \sup_{0 \le s \le t} ( \widehat{Z}^{e^{(n)}}_{s})^{+} \rfloor \big] \Big \rvert \, 
\end{split}
\end{equation}
in terms of $\,\lVert e^{(n+1)} - e^{(n)} \rVert _{T}\,$. 
We define $\, {\bm \{} x {\bm \}}  \, :=\, x - \lfloor x\rfloor  \,$, the non-integer part of $\, x \in [0, \infty) \,$. 

\subsubsection{Case $\, a \, =\,  0 \,$ of No Drifts} 
In the special case when $\, a \equiv 0\,$, we have $\, \widehat{Z}^{e}_{\cdot}\, =\,  \widehat{W}_{\cdot} + e_{\cdot}\,$ in (\ref{eq: ZHATa}).  In this case we may evaluate (\ref{eq: EE}). First observe the identity 
\[
\, \lfloor x  \rfloor  - \lfloor y \rfloor \, =\, \lfloor  x - y  \rfloor + {\bf 1}_{\{ {\bm \{} x {\bm \}} < {\bm \{} y {\bm \}} \}}\, ; \quad 0 \le y \le x <  \infty \,  
\]
on the integer part $\, \lfloor \cdot \rfloor \,$ and non-integer part $\, {\bm \{} \cdot {\bm \}} \,$. Applying this identity inside the expectation in (\ref{eq: EE}), we obtain for every $\,n \ge 0 \,$ 
\begin{equation} \label{eq: INEQ}
\begin{split}
& \lVert e^{(n+2)} - e^{(n+1)} \rVert _{T} -  \lfloor \lVert e^{(n+1)} - e^{(n)}\rVert_{T} \rfloor   \\
& =  \sup_{0 \le t \le T} \Big \lvert \mathbb E \big [ \lfloor \sup_{0 \le s \le t} ( \widehat{W}_{s} + e^{(n+1)}_{s})^{+} \rfloor -  \lfloor \sup_{0 \le s \le t} ( \widehat{W}_{s} + e^{(n)}_{s})^{+} \rfloor \big] \Big \rvert -  \lfloor \lVert e^{(n+1)} - e^{(n)}\rVert_{T} \rfloor 
\\
\, & =\,  \sup_{0 \le t \le T} \Big \lvert   \mathbb E [ \lfloor \sup_{ 0 \le s \le t} ( \widehat{W}_{s} + e_{s}^{(n+1)} )^{+} - \sup_{0 \le s \le t} ( \widehat{W}_{s} + e_{s}^{(n)} )^{+}\rfloor ]  -  \lfloor \lVert e^{(n+1)} - e^{(n)}\rVert_{T} \rfloor 
\\
& ~~~~~~~~~~~~~~~~~~~~~~~~~~~~ + \mathbb P ( {\bm \{} \sup_{0 \le s \le t} ( \widehat{W}_{s} + e^{(n+1)}_{s})^{+} {\bm \}}  < {\bm \{} \sup_{0 \le s \le t} ( \widehat{W}_{s} + e^{(n)}_{s})^{+} {\bm  \}} ) \Big\rvert  
\\
& \le  \sup_{0 \le t \le T} \mathbb P \big ( {\bm \{} \sup_{0 \le s \le t} ( \widehat{W}_{s} + e^{(n+1)}_{s})^{+} {\bm \}}  < {\bm \{} \sup_{0 \le s \le t} ( \widehat{W}_{s} + e^{(n)}_{s})^{+} {\bm  \}} \big) \, . 
\end{split}
\end{equation} 
In the last inequality of (\ref{eq: INEQ}) we used 
\begin{equation} \label{eq: INEQa}
\lvert  \sup_{ 0 \le s \le t} ( \widehat{W}_{s} + e_{s}^{(n+1)} )^{+} - \sup_{0 \le s \le t} ( \widehat{W}_{s} + e_{s}^{(n)} )^{+}\rvert  \le \lVert e^{(n+1)} - e^{(n)}\rVert_{T} \, ; \quad 0 \le t \le T \, . 
\end{equation}
By (\ref{eq: INEQ}) we have an easy upper bound
\[
\lVert e^{(n+2)} - e^{(n+1)} \rVert _{T} \le   \lfloor \lVert e^{(n+1)} - e^{(n)}\rVert_{T} \rfloor + 1 \, ; \quad n \ge 0 \, ,  
\]
and hence 
\[
e^{(n)}_{T} \, =\, \lVert e^{(n)}\rVert_{T} \, \le \, \sum_{k=1}^{n} \lVert e^{(k)} - e^{(k-1)} \rVert_{T} \le n ( \lfloor \lVert e^{(1)} \rVert_{T} \rfloor + 1) + {\bm \{} \lVert e^{(1)}\rVert_{T } {\bm \}}  < + \infty \, ; \quad n \ge 0 \, .  
\]
Here $\, e^{(1)}_{\cdot} \, =\,  \mathfrak M_{\cdot}(e^{(0)}) \, =\, \mathfrak M_{\cdot}(0)\,$ is evaluated as 
\[
e^{(1)}_{t} \, =\, \mathbb E [ \lfloor \sup_{0 \le s \le t} ( \widehat{W}_{s} )^{+} \rfloor ] \, =\,  \sum_{k=1}^{\infty} \mathbb P ( \sup_{0 \le s \le t} W_{s} \ge k x_{0}) \, =\,  \sum_{k=1}^{\infty} \mathbb P ( \lvert W_{1} \rvert \ge k x_{0} /  \sqrt{t\, } ) 
\]
\[
\, =\,  \sum_{k=1}^{\infty} \int_{k x_{0} / \sqrt{t\,}} ^{\infty} \frac{\, 2 \, e^{-u^{2}/2}}{\sqrt{2 \pi}} {\mathrm d} u 
\le \sum_{k=1}^{\infty} \frac{2 \, e^{-k^{2}x_{0}^{2} / (2t)} \sqrt{t\,}}{\, x_{0} k \sqrt{2 \pi\,}\, } 
\]
\[
\le  \sum_{k=1}^{\infty} \frac{2 \, e^{-k^{2}x_{0}^{2} / (2t)} \sqrt{t\,}}{\, x_{0} \sqrt{2 \pi\,}\, } \, =\, \sqrt{\frac{t}{ 2 \pi x_{0}^{2}}} \big( \Theta_{3}(0, e^{-x_{0}^{2}/(2 t)}) - 1) \, \le \, \frac{t}{\, x_{0}\, } \, ; \quad t \ge 0 \, , 
\]
where $\, \Theta_{3}(\cdot, \cdot) \,$ is the \textsc{Jacobi} elliptic theta function. The first inequality follows from the tail estimate of the Gaussian probability. Thus the curve $\, t\to e^{(1)}_{t}\,$ is bounded by the line with the slope $\, 1\, /\, x_{0}\,$ and zero intercept. 
The first and second derivatives $\, \dot{e}^{(1)}_{t}\,$, $\, \ddot{e}^{(1)}_{t}\,$ of $\, t \to e^{(1)}_{t}\,$ are given by 
\[
\dot{e}^{(1)}_{t} \, :=\,  \frac{ \, \,  {\mathrm d} e^{(1)}_{t}}{ {\mathrm d}\, t } \, =\, \sum_{k=1}^{\infty} \frac{k x_{0} e^{-k^{2} x_{0}^{2}/ (2 t)}}{\sqrt{ 2\pi t^{3}\, }} \,  \ge 0 \, , 
\]
\[
\ddot{e}^{(1)}_{t} \, :=\,  \frac{ \, \,  {\mathrm d}^{2} e^{(1)}_{t}}{ {\mathrm d}\, t^{2} } \, =\, \sum_{k=1}^{\infty} \frac{k x_{0} e^{-k^{2} x_{0}^{2}/ (2 t)}}{\sqrt{ 2\pi t^{3}\, }} \Big( - \frac{\, 3\, }{2 t}  + \frac{\, k^{2} x_{0}^{2}\, }{ 2t^{2}} \Big)\, , 
\]
for $\, t \ge 0 \,$ with $\, \dot{e}^{(1)}_{0+} \, =\,  0 \, =\, \ddot{e}^{(1)}_{0+} \,$ and $\, \lim_{t\to \infty} \dot{e}^{(1)}_{t} \, =\,  0 \, =\,  \lim_{t\to \infty} \ddot{e}^{(1)}_{t}\, $. Then it is natural to consider the family of 
functions 
\begin{equation} \label{eq: LF}
\mathcal L \, :=\, \Big \{ e \in C([0, \infty), [0, \infty) ) \, :\,  \, 
\,  e_{0} \, =\, 0\, , \, e_{t} \, \le\, \ell (t) \, :=\,  \frac{t}{\, x_{0}\, } \, ; \quad t \ge 0 \,  \Big\} \, . 
\end{equation}

\begin{proposition}  \label{lm: CLL}
Assume $\, x_{0} \ge 1 \,$ and $\, a \, =\,  0\,$. For every $\, e \in \mathcal L\,$ in (\ref{eq: LF}) we have $\, \mathfrak M(e) \in \mathcal L\,$. In particular, $\, e^{(n)}_{\cdot}\, $ defined in (\ref{eq: RE}) belongs to $\, \mathcal L\,$ for every $\,n \ge 0\,$. 
\end{proposition}

\begin{proof}
For every $\, e \in \mathcal L\,$ we have $\, e_{t} \le \ell (t) \, =\, t / x_{0}\,$, $\, t\ge 0 \,$ and hence if $\, x_{0} \ge 1\,$, then 
\[
0 \le \mathfrak M_{t}(e) \, =\,  \mathbb E \big[ \lfloor \sup_{0 \le s \le t } ( \widehat{W}_{s } + e_{s})^{+} \rfloor \big] \, =\, \sum_{k=1}^{\infty} \mathbb P \big ( \sup_{0 \le s \le t} ( \widehat{W}_{s} + e_{s})^{+} \ge k \big) 
\le \sum_{k=1}^{\infty} \mathbb P\big ( \sup_{0 \le s \le t } ( W_{s} + s)^{+} \ge k x_{0} \big) 
\]
\begin{equation} \label{eq: RENEW}
\, =\, \frac{1}{\, 2\, } \sum_{k=1}^{\infty}  \Big( \text{Erfc} \Big( \frac{\, k x_{0}\, }{\sqrt{ \, 2t \, }} - \sqrt{\frac{t}{2}} \Big) + e^{2 x_{0} k } \, \text{Erfc} \Big( \frac{\, k x_{0}\, }{\sqrt{ \, 2t \, }} + \sqrt{\frac{t}{2}} \Big ) \Big) 
\le \frac{t}{\, x_{0}\, }  \, 
\end{equation}
for every $\,t \ge 0 \,$. The last inequality in (\ref{eq: RENEW}) may be directly verified in some numerical approximation of the infinite series by the corresponding finite sum. See Appendix \ref{Appendix: (2.31)} for the formal proof of the last inequality in (\ref{eq: RENEW}) by the renewal theory. If $\,x_{0} < 1\,$, then the last inequality (\ref{eq: RENEW}) does not necessarily hold for some small  $\, t \ge 0\,$. Thus we obtain the claim. 
\end{proof}

The differentiability of $\,t \to  \mathcal M_{t} (e) \,$ may be shown as in Proposition 3.1 of \cite{MR3349003}.

\bigskip 

If $\, x > y \,$ but $\, {\bm \{ } x {\bm \}} <  {\bm \{ } y {\bm \}}  \,$ for some $\,x , y \in \mathbb R_{+}\,$, then $\, {\bm \{} x - y {\bm \}} \, =\,  1 + {\bm \{ } x {\bm \}} - {\bm \{ } y {\bm \} } \ge { \bm \{ }  x {\bm \}} \,$ and hence, $\, \lfloor x \rfloor \le  x \, =\, \lfloor x \rfloor + {\bm \{ } x {\bm \}} \, \le  \lfloor x \rfloor  + {\bm \{ } x - y {\bm \}} \,$. This observation with (\ref{eq: INEQa}) implies that if $\,\lVert e^{(n+1)} - e^{(n)}\rVert_{T} < 1\,$, i.e., $\,\lVert e^{(n+1)} - e^{(n)}\rVert_{T} \, =\, {\bm \{} \lVert e^{(n+1)} - e^{(n)}\rVert_{T} {\bm \}}  \,$, then 
\[
\Big \{  {\bm \{} \sup_{0 \le s \le t} ( \widehat{W}_{s} + e^{(n+1)}_{s})^{+} {\bm \}}  < {\bm \{} \sup_{0 \le s \le t} ( \widehat{W}_{s} + e^{(n)}_{s})^{+} {\bm  \}} \Big \} 
\]
\[
\subseteq  \bigcup_{k=1}^{\infty } \Big \{ 
k \le \sup_{0 \le s \le t} ( \widehat{W}_{s} + e^{(n+1)}_{s})^{+} < k +  {\bm \{ } \lVert e^{(n+1)} - e^{(n)}\rVert_{T} {\bm \}}  \Big \} \, 
\]
and hence,  
\begin{equation}  \label{eq: SL}
\begin{split}
&  \sup_{0 \le t \le T} \mathbb P \big ( {\bm \{} \sup_{0 \le s \le t} ( \widehat{W}_{s} + e^{(n+1)}_{s})^{+} {\bm \}}  < {\bm \{} \sup_{0 \le s \le t} ( \widehat{W}_{s} + e^{(n)}_{s})^{+} {\bm  \}} \big) 
\\ 
 & \le \sup_{0 \le t \le T} \sum_{k=1}^{\infty} \mathbb P \Big (  \sup_{0 \le s \le t} ( \widehat{W}_{s} + e^{(n+1)}_{s})^{+}  \in (k,  k + {\bm \{} \lVert e^{(n+1)} - e^{(n)}\rVert_{T} {\bm \}}  ) \Big) \, . 
 \end{split}
\end{equation}

Now let us write $\, \varepsilon \, :=\, {\bm \{} \lVert e^{(n+1)} - e^{(n)}\rVert_{T} {\bm \}}  \in (0, 1) \,$. Combining (\ref{eq: SL}) with the inequality 
\begin{equation} \label{eq: SHEPP}
\sum_{k=1}^{\infty} \mathbb P \Big (  \sup_{0 \le s \le t} ( \widehat{W}_{s} + e^{(n+1)}_{s})^{+}  \in (k,  k + \varepsilon ) \Big)   \le \, \sum_{k=1}^{\infty} \mathbb P \Big ( \sup_{0 \le s \le t}\Big( \widehat{W}_{s} + \frac{s}{\, x_{0}\, } \Big)^{+} \in (k, k+  \varepsilon ) \Big ) , 
\end{equation}
we may find $\, \delta_{0, T, x_{0}} \in (0, 1)\, $ such that 
\begin{equation} \label{eq: INEQb}
\begin{split}
 & \sup_{0 \le t \le T} \mathbb P \big ( {\bm \{} \sup_{0 \le s \le t} ( \widehat{W}_{s} + e^{(n+1)}_{s})^{+} {\bm \}}  < {\bm \{} \sup_{0 \le s \le t} ( \widehat{W}_{s} + e^{(n)}_{s})^{+} {\bm  \}} \big) \\
& \le \sup_{0 \le t \le T} \, \sum_{k=1}^{\infty} \mathbb P \Big ( \sup_{0 \le s \le t}\Big( \widehat{W}_{s} + \frac{s}{\, x_{0}\, } \Big)^{+} \in (k, k+  \varepsilon )  
 \Big ) \\
& \, =\,  \sum_{k=1}^{\infty} \frac{1}{2} \Big [ \text{Erfc} \Big( \frac{x_{0} k }{\sqrt{2 T}} - \sqrt{ \frac{T}{2}} \Big)  + e^{2 x_{0} k } \text{Erfc} \Big( \frac{x_{0} k }{\sqrt{2 T}} + \sqrt{ \frac{T}{2}} \Big) \\
& ~~~~~~~~~~~~~~~~~~~~~~~ - \text{Erfc} \Big( \frac{x_{0} (k+\varepsilon) }{\sqrt{2 T}} - \sqrt{ \frac{T}{2}} \Big)   - e^{2 x_{0} (k+\varepsilon)} \text{Erfc} \Big( \frac{x_{0} (k+\varepsilon) }{\sqrt{2 T}} + \sqrt{ \frac{T}{2}} \Big) \Big]  \\
& \le \delta_{0, T, x_{0}} \cdot \varepsilon \, =\, \delta_{0, T, x_{0}} \cdot  {\bm \{ } \lVert e^{(n+1)} - e^{(n)}\rVert_{T}  {\bm \}} \, 
\end{split}
\end{equation}
for every $\,n \ge 0\,$. Note that $\, \lim_{T /x_{0}^{2} \to \infty} \delta_{0, T, x_{0}} \, =\,  1 \,$. Thus if $\,\lVert e^{(n+1)} - e^{(n)}\rVert_{T} < 1\,$, then combining (\ref{eq: INEQ}) with (\ref{eq: INEQb}), we obtain 
\begin{equation} \label{eq: CTRa}
\lVert e^{(n+2)} - e^{(n+1)} \rVert _{T} \le \delta_{0, T, x_{0}}  \,  \lVert e^{(n+1)} - e^{(n)} \rVert_{T}  \, ; \quad n \ge 0 \, . 
\end{equation}
Since we have $\, e^{(1)}_{t} \le t \, / \, x_{0}\,$ by Lemma \ref{lm: CLL} and $\, \lVert e^{(1)} - e^{(0)} \rVert_{t} \, =\,  e^{(1)}_{t} < 1\,$ for $\, 0 \le t < x_{0}\,$, we have the following conjecture.

\noindent {\bf Conjecture :} Assume that $\, \{e^{(n)}\}\, $ is generated by the recipe in (\ref{eq: RE}). For every $\,T \,$ there exists $\, \delta \in (0, 1) \,$ such that 
\begin{equation} \label{eq: CONJ}
 \sup_{0 \le t \le T} \mathbb P \big ( {\bm \{} \sup_{0 \le s \le t} ( \widehat{W}_{s} + e^{(n+1)}_{s})^{+} {\bm \}}  < {\bm \{} \sup_{0 \le s \le t} ( \widehat{W}_{s} + e^{(n)}_{s})^{+} {\bm  \}} \big) \, \le \, \delta \, {\bm \{} \lVert e^{(n+1)} - e^{(n)} \rVert_{T} {\bm \}} \, ; \quad n \ge 0 \, . 
\end{equation}
If this conjecture holds, then we see the contraction 
\begin{equation} \label{eq: CONT}
\lVert e^{(n+2)} - e^{(n+1)} \rVert _{T} \, \le \, \lfloor \lVert e^{(n+1)} - e^{(n)}\rVert_{T} \rfloor + \delta \, {\bm \{} \lVert e^{(n+1)} - e^{(n)} \rVert_{T} {\bm \}} \le \big( 1  - ( 1 - \delta ) c \big) \,  \lVert e^{(n+1)} - e^{(n)} \rVert_{T} 
\end{equation}
conditionally on $\, \, {\bm \{} \lVert e^{(n+1)} - e^{(n)} \rVert_{T} {\bm \}} \, / \, \lVert e^{(n+1)} - e^{(n)} \rVert_{T} \, \ge \, c > 0 \,$ for some constant $\, c > 0 \,$.  If $\,x_{0} \ge 1\,$, a unique limit $\, e^{\ast}_{t} \, =\,  \lim_{n\to \infty} e^{(n)}_{t} \, $, $\, 0 \le t < x_{0} \, $ exists and satisfies the fixed point property: 
\[
e^{\ast}_{t} \, =\,  \mathfrak M_{t}(e^{\ast}) \, ; \quad 0 \le t < x_{0} \, . 
\]

Figure \ref{fig: iteration-p} shows the convergence of Picard iteration of the map $\, \mathfrak M_{\cdot}(e^{i}) \,$, $\, i \, =\, 1, 2, \,$ with initial input $\, e^{0}_{\cdot}\equiv 0\,$ in (\ref{eq: ME}), when $\, a \, =\,  0\,$ and the initial value $\,x_{0}\,$ is distributed in a stationary distribution from section \ref{sec:computee0}.

\begin{figure}[H]
\begin{center}
\includegraphics[scale=0.3]{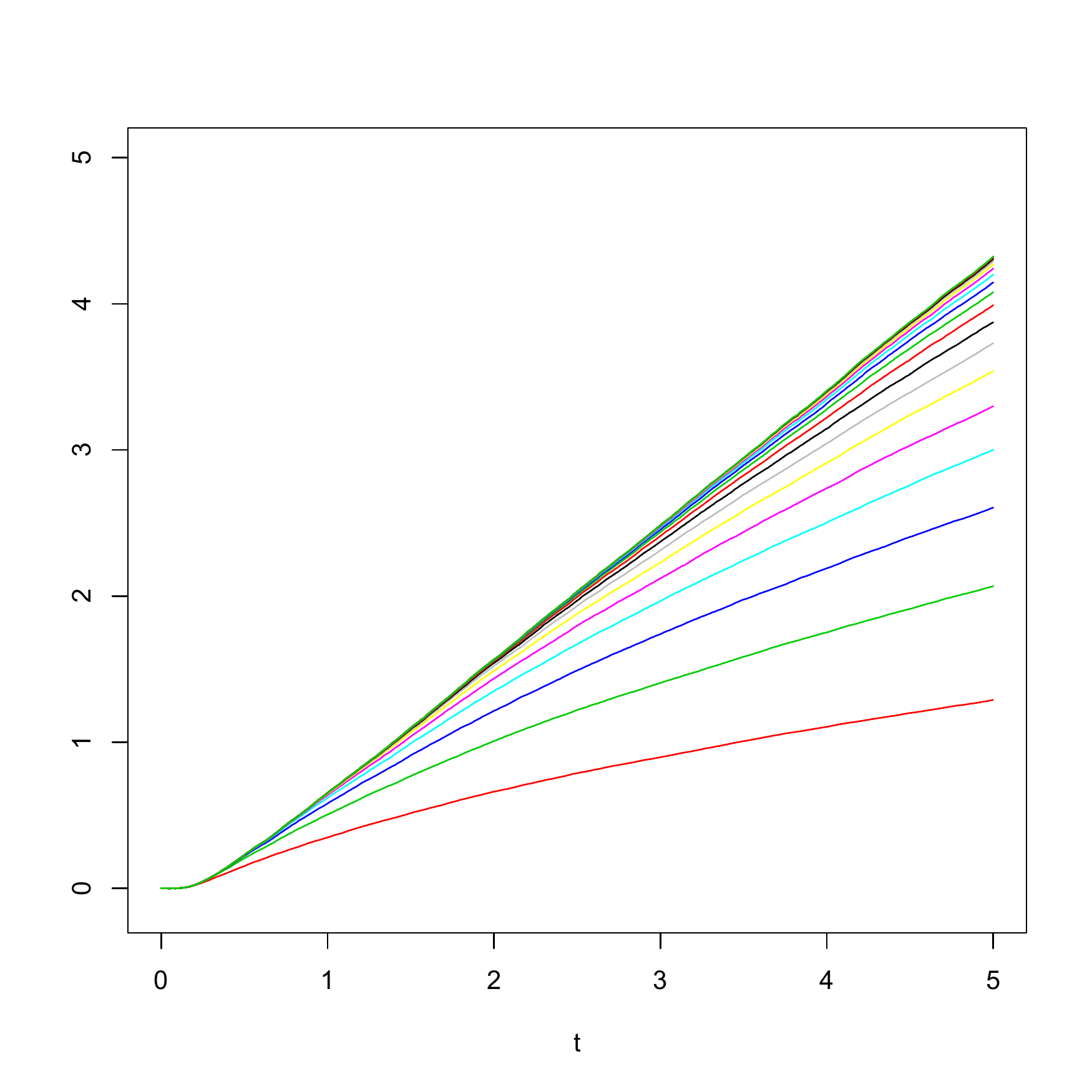} 
\end{center}
\caption{\label{fig: iteration-p}
The iteration of the map $\,\mathfrak M_{t} (e^{i})\,$, $\, 0 \le t \le T\,$, $\, i \, =\,  1, 2, \ldots , 21\,$ in (\ref{eq: ME}) is shown under the stationary initial distribution. A fixed point $\,e^{\ast}_{t} \,$, $\, 0 \le t \le T\,$ of the map $\,\mathfrak M_{\cdot} (\cdot)\,$ is shown as the maximum curve, when $\, a \, =\,  0\,$.} 
\end{figure}

\subsection{Evolutionary FP equation}

\paragraph{\textbf Derivation of the Fokker-Plank equation}

Let us consider the stochastic integral equation 
\begin{equation}
\mathcal X_{t} \, =\,  \mathcal X_{0} + (- a) \int^{t}_{0} ( \mathcal X_{s} - \mathbb E [ \mathcal X_{s}] ) {\mathrm d} s + W_{t} + \int^{t}_{0} \mathbb E [ \mathcal X_{s-} ] {\mathrm d} \mathcal M_{s} -  \alpha \int^{t}_{0 }\mathbb E [ \mathcal X_{s-}] {\mathrm d}_{s} \mathbb E  [ \mathcal M_{s}] \, ; \quad t \ge 0 \, , 
\end{equation}
where $\, a \ge 0 \,$, $\,\alpha \in \mathbb R \,$, $\, \mathcal M_{t} \, :=\, \sum_{k=1}^{\infty} {\bf 1}_{\{\tau^{k} \le t \} }\,$ and $\, \tau^{k} \, := \inf \{ s > \tau^{k-1}: \mathcal X_{s-} \le 0 \} \,$, $\,k \ge 1\,$ with $\, \tau^{0} \, :=\,  0 \,$, and $\, W_{\cdot}\,$ is a Brownian motion on a filtered probability space. Assume for a moment that the process $\, \{ \mathcal X_{t}, t \ge 0 \} \,$ is well defined with a uniquely determined probability distribution on some probability space $\, (\Omega, \mathcal F, \mathbb P)\,$. Here $\, \mathcal M_{t}\,$ is cumulative number of default events $\, \{ 0 \le s \le t: \mathcal X_{s} \, =\,  0 \} \,$ until time $\,t\ge 0 \,$. We take $\, t \to \mathcal M_{t}\,$ as a c\`adl\`ag process, i.e., right continuous with left limits. 
Assuming $\, t \to \mathbb E [ \mathcal X_{t}] \,$ and $\, t \to \mathbb E [ \mathcal M_{t}]  \, =:\,  e_{t}\,$ are smooth, let us introduce its expectation $ \overline{x}_t := \mathbb E[ \mathcal X_{t}]$ and derivative $\, \dot{e}_{t} \, =\, {\mathrm d} \mathbb E [ \mathcal M _{t}] / {\mathrm d} t \,$, $\, t \ge 0 \,$. Thus we may rewrite the dynamics 
\begin{equation}
\mathcal X_{t} \, =\,  \mathcal X_{0} + (- a) \int^{t}_{0} ( \mathcal X_{s} - \overline{x}_{t} ) {\mathrm d} s + W_{t} +    \int^{t}_{0 }\overline{x}_{s} {\mathrm d} \mathcal M_{s } -  \alpha\int^{t}_{0} \overline{x}_{s}  \dot{e}_{s} {\mathrm d} {s} \, ; \quad t \ge 0 \, , 
\end{equation}
where $\, \overline{x}_{t}\, =\,  \overline{x}_{0} \cdot \exp ( ( 1- \alpha) e_{t}) \,$, $\, t \ge 0 \,$. 

The probability density function $\, p(t, x) {\mathrm d} x \, =\, \mathbb P ( \mathcal X_{t} \in {\mathrm d}  x) \,$ of $\, \mathcal X_{t}\,$ solves the Fokker-Plank equation 
\begin{equation}
\label{eq:FP-eq}
	\partial_{t} {p}(t, x) + \partial_{x} [ (-a(x- \overline{x}_{t}) - \alpha \overline{x}_{t} \dot{e}_{t})\, {p}(t, x)] - \frac{\,1\,}{\,2 \,} \partial^{2}_{xx} {p}(t, x) 
	\, =\,  
	\dot{e}_{t} \, \delta_{ \overline{x}_{t}} ( {\mathrm d} x)  \, 
\end{equation}
for $\, t > 0 \,$, $\,   x > 0 \,$, where $\,\delta_{x}({\mathrm d} x) \,$  is a Dirac measure at $\, x\,$. 
For the boundary condition we assume that  
\begin{equation}
\label{eq:FP-BC} 
\lim_{x \downarrow 0} {p}(t, x) \, =\,  0 , \quad  \lim_{x\to + \infty}  {p}(t, x) \, =\, 0 \, , \quad\lim_{x\to \infty} \partial_{x} {p}(t, x) \, =\,  0 \;,  
\end{equation}
\begin{equation}
\label{eq:FP-init} 
\lim_{t\downarrow 0 } p(t, x)  \, =\,  \mathbb P ( \mathcal X_{0} \in {\mathrm d} x ) / {\mathrm d} x \, , 
\end{equation}
and 
\begin{equation}
\label{eq:FP-dot-et}
\overline{x}_{t} \, =\,  \int^{\infty}_{0} x p(t, x) {\mathrm d} x \, , \quad \dot{e}_{t} \,=\, \frac{\,1\,}{\,2\,}\partial_{x} p ( t, 0) \, ; \quad t \ge 0 \, . 
\end{equation}
See Appendix~\ref{sec:deriv-KFP} for more details on the derivation of this equation.

We note that if the mean-field term $\,\overline{x}_{t}\,$ is forced to be a constant (say $\, {x}_{0} > 0 \,$), then the corresponding PDE is 
\[
	\partial_{t} {p}(t, x) + \partial_{x} [ (-a(x- {x}_{0}) - \alpha {x}_{0} \dot{e}_{t})\, {p}(t, x)] - \frac{\,1\,}{\,2 \,} \partial^{2}_{xx} {p}(t, x) 
	\, =\,  
	\dot{e}_{t} \, \delta_{ {x}_{0}} ( {\mathrm d} x)  \, 
\]
for $\, (t, x) \in (0, \infty) \times (0, \infty) \,$, and then after a change of variables $\, y = (x_{0} - x)/x_{0}\,$, $\, \widehat{p}(t, (x_{0} - x)/x_{0}) \, :=\,  p(t, x)\,$, we obtain another Fokker-Planck equation 
\begin{equation} \label{eq: FPeq2}
\partial_{t} \widehat{p}(t, y) + \partial_{y} [ ( - a y + \alpha \dot{e}_{t}) \widehat{p}(t,y)] - \frac{\,1\,}{\,2 x_{0}^{2}\,} \partial_{yy}^{2} \widehat{p}(t, y) \, =\,  \dot{e}_{t} \delta_{0} ( {\mathrm d} y ) \, 
\end{equation}
for $\, (t, y) \in (0, \infty) \times (-\infty, 1)\,$ with the condition corresponding to (\ref{eq:FP-BC})-(\ref{eq:FP-dot-et}). 

Our Fokker-Planck equation (\ref{eq: FPeq2}) is an extension to the Fokker-Planck equation studied in~\cite{MR2853216}:
\[
\partial_{t} \widehat{p}(t, y) + \partial_{y} [ ( - y + \alpha \dot{e}_{t}) \widehat{p}(t,y)] - \frac{\,1\,}{\,2 \,} \partial_{yy}^{2} \widehat{p}(t, y) \, =\,  \dot{e}_{t} \delta_{0} ( {\mathrm d} y ) \,.  
\] 
Indeed, with $\, a \, =\,  1\,$ and $\, x_{0} \, =\, 1 \,$, (\ref{eq: FPeq2}) reduces to the study in~\cite{MR2853216}.  

\medskip

\paragraph{\textbf Notion of solution and blow-up phenomenon }
We borrow the following notion of solution from~\cite{MR2853216}. 
\begin{definition}
\label{def:weak-sol-FP}
	We say that a pair of nonnegative functions $(p, \dot e)$ with $p \in L^\infty(\RR^+; L^1_+(0, +\infty))$ and $\dot e \in L^1_{loc,+}(\RR^+)$ is a weak solution of~\eqref{eq:FP-eq}--\eqref{eq:FP-dot-et} with initial condition $p_0(\cdot) := p(0, x) $, if for any test function $(t,x) \mapsto \phi(t,x)$, $\phi \in C^\infty([0,+\infty) \times [0,T])$ such that $\partial_{xx}^2 \phi$, $x \partial_x \phi \in L^\infty([0,+\infty) \times (0,T))$, and we have
	\begin{align*}	 
	& \int_0^T \int_0^{+\infty} p(t,x) \left[-\partial_t \phi(t,x)  - \partial_x \phi(t,x) (-a(x-x_{0}) - x_{0} \dot{e}_{t}) - \frac{\,1\,}{\,2 \,} \partial^{2}_{xx} \phi(t, x)  \right] \, {\mathrm d}x \, {\mathrm d} t 
	\\
	\, & =\,  
	\int_0^T \dot{e}_{t} \, \left[ \phi(x_{0}) - \phi(0) \right]\, dt
	 + \int_0^{+\infty} p_0(x) \phi(0, x) \, {\mathrm d}x - \int_0^{+\infty} p(T, x) \phi(T,x) \, {\mathrm d}x.
	\end{align*}
\end{definition}
By choosing test functions of the form $\phi(t,x) = \psi(t)\phi(x)$ and differentiating with respect to time variables, this definition is equivalent to having the following equation satisfied for every $\phi \in C^\infty([0,+\infty) )$ with $x \partial_x \phi \in L^\infty((0,+\infty))$, 
	\begin{equation}
	\label{eq:weak-sol-FP}
	\begin{split}	 
	& \frac{d}{dt}  \int_0^{+\infty} \phi(x) p(t,x) {\mathrm d}x
	\\
	&=
	  \int_0^{+\infty}  \left[\partial_x \phi(x) (-a(x-x_{0}) - x_{0} \dot{e}_{t}) + \frac{\,1\,}{\,2 \,} \partial^{2}_{xx} \phi(x) \right] p(t,x) \, {\mathrm d}x 
	+ \dot{e}_{t} \, \left[ \phi(x_{0}) - \phi(0) \right].
	\end{split}
	\end{equation}

The complete analysis of the existence and uniqueness of the weak solution is beyond the scope of our current study. In the following, we shall point out that the weak solution does not exist globally in time due to the blow-up phenomena, if the initial distribution $\,p_{0}(\cdot) =p (0, \cdot)\,$ concentrates near the origin. 

\begin{theorem} [Blow-up phenomenon] \label{thm: Blowup}
	Fix  $a \in \RR$  and $\, x_{0} > 0 \,$. If there exists $\,\mu > \max ( 2 a x_{0}, 1) \,$ such that the initial condition $\,p_{0}(\cdot) \,$ satisfies 
\begin{equation} \label{eq:proof-blowup-cond-p0}
\int^{\infty}_{0} e^{-\mu x} p_{0}(x) d x  \ge \frac{\,1 - e^{-\mu x_{0}}\,}{\,\mu x_{0} \,}	\, , 
\end{equation} 
then there are no global-in-time weak solutions to~\eqref{eq:FP-eq}--\eqref{eq:FP-dot-et}.
\end{theorem}

\begin{proof}
	
The proof follows the lines of~\cite[Theorem 2.2]{MR2853216}, adapted to our setting. Let us assume there exists a global-in-time weak solution, in the sense of Definition~\ref{def:weak-sol-FP} with a test function $\phi(t,x) = \phi(x) = e^{- \mu x}$. Let us define the Laplace transform $M_\mu(t) = \int_0^\infty \phi(x) p(t,x) dx$ of $\,p(t,x)\,$. Notice that $x_0 \mu \dot{e}_{t} \geq 0$ for all $t \geq 0$ and $M_\mu(0) \geq \frac{\lambda}{x_0}$ by~\eqref{eq:proof-blowup-cond-p0} with $\lambda := \frac{\phi(0) - \phi(x_0)}{\mu} > 0$.  By~\eqref{eq:weak-sol-FP}, we have
	\begin{align}	 
	 \frac{d}{dt}  M_\mu(t)
	&=
	  \int_0^{+\infty}  \left[-\mu \phi(x) (-a(x-x_{0}) - x_{0} \dot{e}_{t}) + \frac{\,1\,}{\,2 \,} \mu^2 \phi(x) \right] p(t,x) \, dx 
	 - \lambda \mu \dot{e}_{t}
	 \notag
	\\
	&\geq
	 \mu \left[  x_{0} \dot{e}_{t} + \frac{\,1\,}{\,2 \,} \mu -a x_{0} \right]  M_\mu(t)
	 - \lambda \mu \dot{e}_{t}
	 \label{eq:proof-blowup-1}
	 \\
	 & \geq x_0 \mu \dot{e}_{t} \left[  M_\mu(t)
	 - \frac{\lambda}{x_0} \right],
	 \label{eq:proof-blowup-2}
	\end{align}
	where we used the fact that $x \geq 0$ and $\mu \geq 2 a x_0$.  Hence, by the Gronwall inequality, \eqref{eq:proof-blowup-2} implies 
	\begin{equation}
	\label{eq:proof-blowup-Mmu-1}
	M_\mu(t)\geq \frac{\lambda}{x_0}, \qquad \forall t \geq 0.
	\end{equation}
	
	Going back to~\eqref{eq:proof-blowup-1}, we obtain 
	$$
		\frac{d}{dt}  M_\mu(t) 
		\geq
		 \mu \left[ \frac{\,1\,}{\,2 \,} \mu -a x_{0} \right]  M_\mu(t),
	$$
	which implies, again by the Gronwall inequality,  
	$$
		M_\mu(t) \geq e^{\mu \left[ \frac{\,1\,}{\,2 \,} \mu -a x_{0} \right] t} M_\mu(0) \geq e^{\mu \left[ \frac{\,1\,}{\,2 \,} \mu -a x_{0} \right] t} \cdot \frac{\lambda}{x_0}.
	$$
	Since $  \frac{\,1\,}{\,2 \,} \mu -a x_{0} > 0$, the right hand side grows to $+\infty$ as $t \to +\infty$.
	On the other hand, since $\phi(x) = e^{-\mu x} \leq 1$ and $p$ is a probability density,
	$$
		M_\mu(t) = \int_0^\infty \phi(x) p(t,x) dx \leq \int_0^\infty p(t,x) dx \leq 1,
	$$
	which yields a contradiction when $t$ is large enough.
\end{proof}

For example, if the initial condition $\, p_{0}(\cdot)\,$ takes a form of triangular distribution supported by an open interval $\,(0, 2c)\,$ with $\, p_{0}(x) \, =\, {\bf 1}_{\{0 < x < c\}} \cdot  x/c^{2} + {\bf 1}_{\{c < x < 2c\}} \cdot (2c - x) / c^{2}\,$, then by choosing $\,c = x_{0} / (2a) \,$, $\, x_{0} \, =\,  0.2\,$, $\,\mu \, =\,  1\,$, $\,a \, =\, 5\,$, one can see the condition (\ref{eq:proof-blowup-cond-p0}) holds with 
\[
\int^{\infty}_{0} e^{-\mu x} p_{0}(x) {\mathrm d} x = \frac{\,( e^{x_{0}^{2}} - 1)^{2} e^{-2x_{0}^{2}}\,}{\,x_{0}^{4}\,} > \frac{\,1 - e^{-x_{0}}\,}{\,x_{0}\,} 	\, . 
\]
The probability of this triangular distribution $\,p_{0}(\cdot) \,$ is concentrated near the origin, and by Theorem \ref{thm: Blowup}, there is no global-in-time weak solution to~\eqref{eq:FP-eq}--\eqref{eq:FP-dot-et}.

The existence of steady states and the convergence to such stationary distributions have been addressed respectively in ~\cite[Theorems 3.1 and 4.1]{MR2853216}. However, our dynamics does not fit in the assumptions made for the aforementioned results (we are in the regime where, using the notations of~\cite{MR2853216}, $b = V_F - V_R = 1$ and the function $h$ depends on $N$). For this reason, we address the question of steady states in the next section by directly finding an explicit solution.

\subsection{Explicit solution for the stationary FP equation}\label{sec:computee0}

Let us look for a stationary solution to the Fokker-Planck equation~\eqref{eq:FP-eq}--\eqref{eq:FP-dot-et}. In other words, we look for a function $\, p: (-\infty, 1) \to \mathbb{R} \,$ such that, at least in a weak sense, 
\begin{equation}
\label{eq:statioFP-eq}
	\frac{d}{dx} [ (-a(x-x_{0}) - x_{0} e_0)\, {p}(x)] - \frac{\,1\,}{\,2 \,} p''(x) 
	\, =\,  
	e_0 \, \delta_{x_{0}} ( {\mathrm d} x)  \, 
\end{equation}
for $\, t > 0 \,$, $\,   x > 0 \,$. 
For the boundary condition we assume that  
\begin{equation}
\label{eq:statioFP-BC}
\lim_{x \downarrow 0} {p}(x) \, =\,  0 , \quad  \lim_{x\to + \infty} x \, {p}(x) \, =\, 0 \, , \quad\lim_{x\to \infty} p'(x) \, =\,  0 \;,  
\end{equation}
and 
\begin{equation}
\label{eq:statioFP-dot-et}
e_0 \,=\, \frac{\,1\,}{\,2\,} p' ( 0) \, . 
\end{equation}

We show that there exists an explicit solution to the stationary Fokker-Planck equation.
\begin{theorem}
\label{thm:characterize-statio-p-e0}
	{\bf (i) } When $a>0$, the following is a solution to~\eqref{eq:statioFP-eq}--\eqref{eq:statioFP-dot-et}~: for all $x \in [0, +\infty)$,
	\begin{equation}
	\label{eq:FPstationary-solutionp}
 	p(x)
 	\, = \,
	\displaystyle
	2 e_0 \left( \int_0^{\min(x,x_0)} e^{a y^2 + 2 x_0 (e_0-a) y} dy \right) e^{-a x^2 - 2 x_0 (e_0-a) x} \, , 
	\end{equation}
	where $e_0$ is uniquely determined by~:
	 \begin{equation}
	 \label{eq:FPstationary-solutione0}
	\frac{e_0}{a} \int_{\left(\sqrt{\frac{2}{a}} e_0 - \sqrt{2a}\right) x_0}^{\sqrt{\frac{2}{a}}e_0 x_0 }  e^{\frac{y^2}{2}} \Nc(-y)  dy  
	\, = \, 1 \,,
	 \end{equation}
	 where $\Nc(z) = \int_{-\infty}^z e^{\frac{-x^2}{2}} d x$, and we have an upper bound estimate 
\begin{equation} \label{eq:FPstationary-solution-e0est}
0 < e_{0} < \max \Big( \frac{\,2a \,}{\, \log (2 a x_{0}^{2}) \,}\, , \frac{\,1 \,}{\,2 x_{0}^{2}\,} (e^{2} ( 1 + 2 a x_{0}^{2}) - 1)\Big) \, . 
\end{equation}	
	 
	 {\bf (ii) } When $a = 0$, a solution is given by~:
	 \begin{equation}
	 \label{eq:FPstationary-solutionp-a0}
 	p(x)
 	\, = \, 
	2 e_0 \left( \int_0^{\min(x, x_0)} e^{2 x_0 e_0 y} dy \right) e^{ - 2 x_0 e_0 x} \;,
	\end{equation}
	with
	\begin{equation}\label{eq:FPstationary-solutione0-a0}
		e_0  = \frac{1}{x_0^2} \;.
	\end{equation}
\end{theorem}
The proof is given in Appendix~\ref{sec:proofThmCharac}. The above result allows us to study how the default rate depends on the parameters of the model. As shown on Figure~\ref{fig:e0-fct-alpha-e0}, $e_0$ is decreasing with respect to $a$ (the mean-reverting term has a stronger effect and stabilizes the system) and it is also decreasing with respect to $x_0$ (banks start further away from $0$).

\begin{figure}[!h]
 \begin{center}
 \includegraphics[scale=0.7]{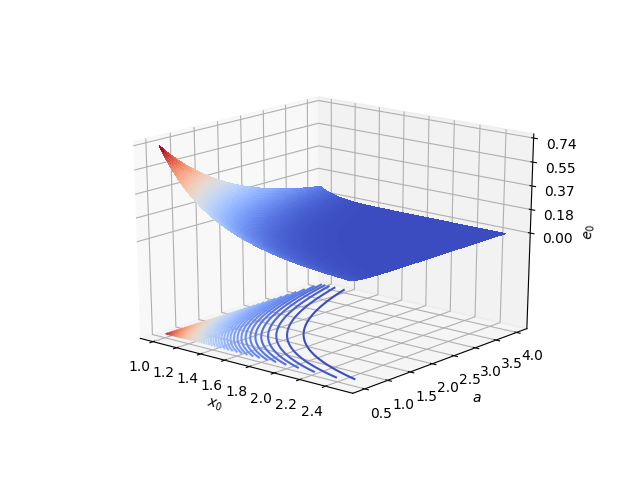}
  \end{center}
\caption{
{\label{fig:e0-fct-alpha-e0}}
$e_0$ as a function of $a$ and $x_0$.
}
\end{figure}

\begin{figure}[H]
\centering
\begin{subfigure}{.33\textwidth}
  \centering
\includegraphics[scale=0.27]{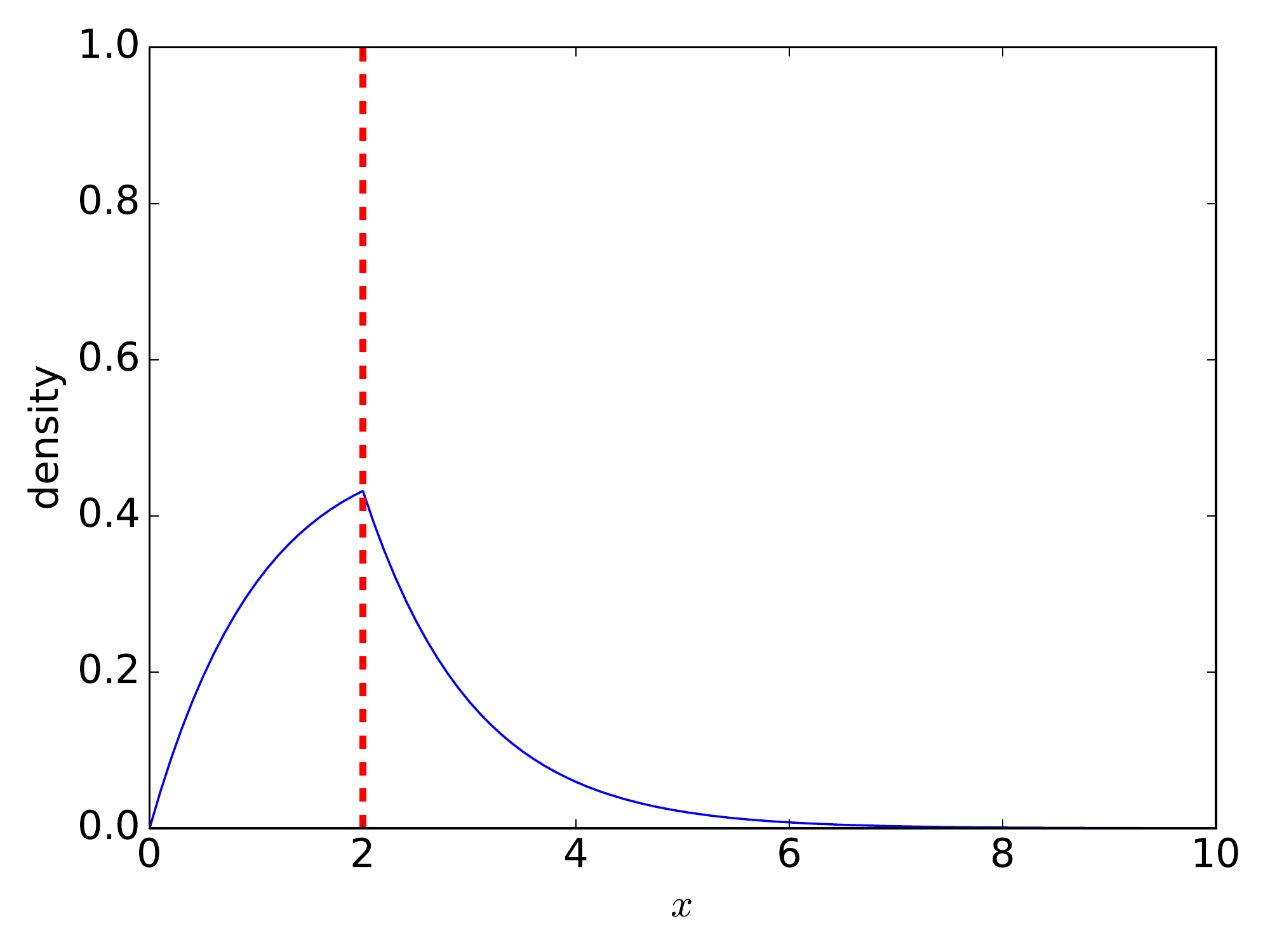}
  \caption{}
  \label{fig:test1-density-1}
\end{subfigure}%
 \hspace{-0.2cm}
\begin{subfigure}{.33\textwidth}
  \centering
  \includegraphics[scale=0.27]{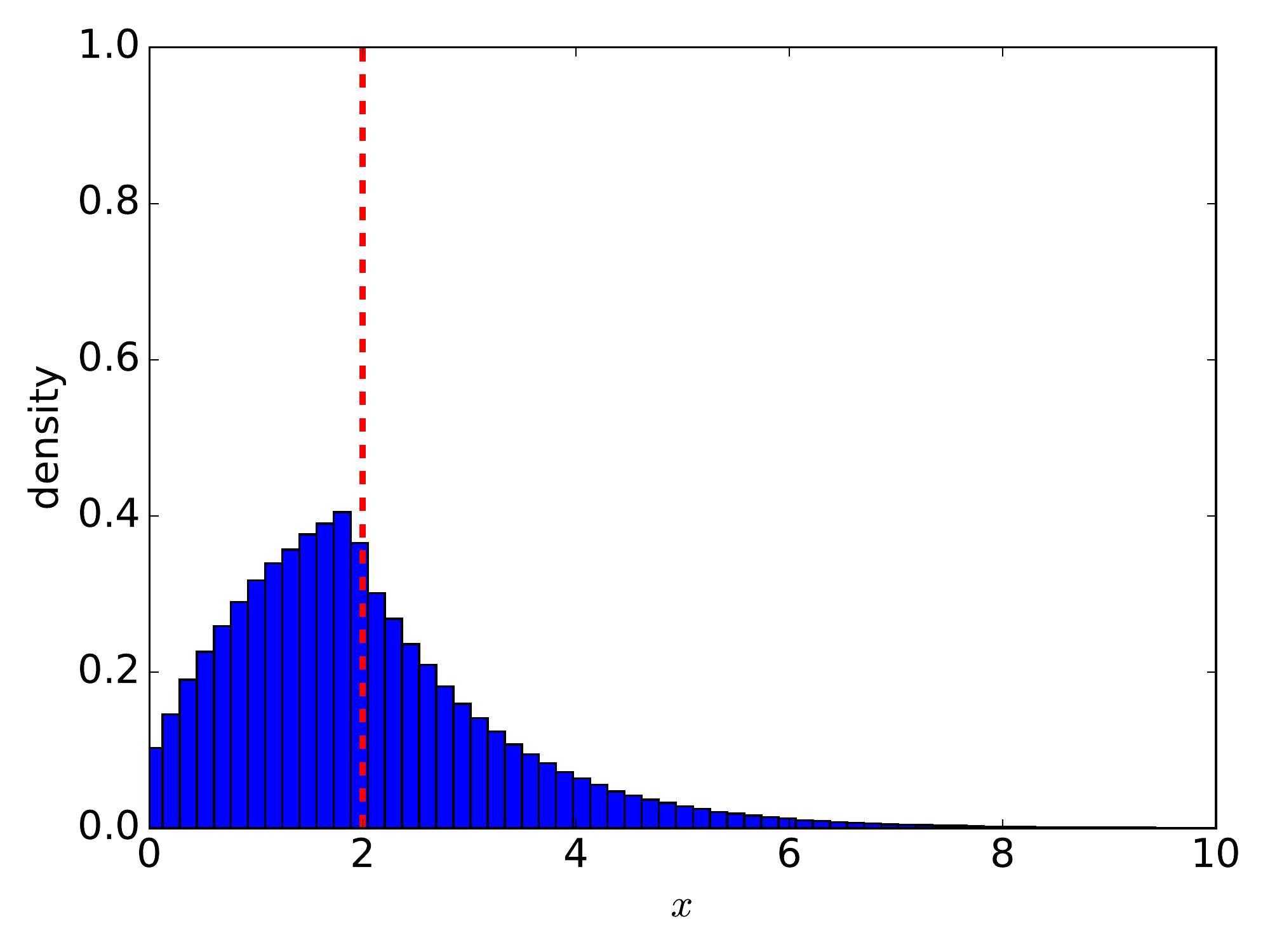}
  \caption{}
  \label{fig:test1-density-2}
\end{subfigure}%
 \hspace{-0.2cm}
\begin{subfigure}{.33\textwidth}
  \centering
    \includegraphics[scale=0.27]{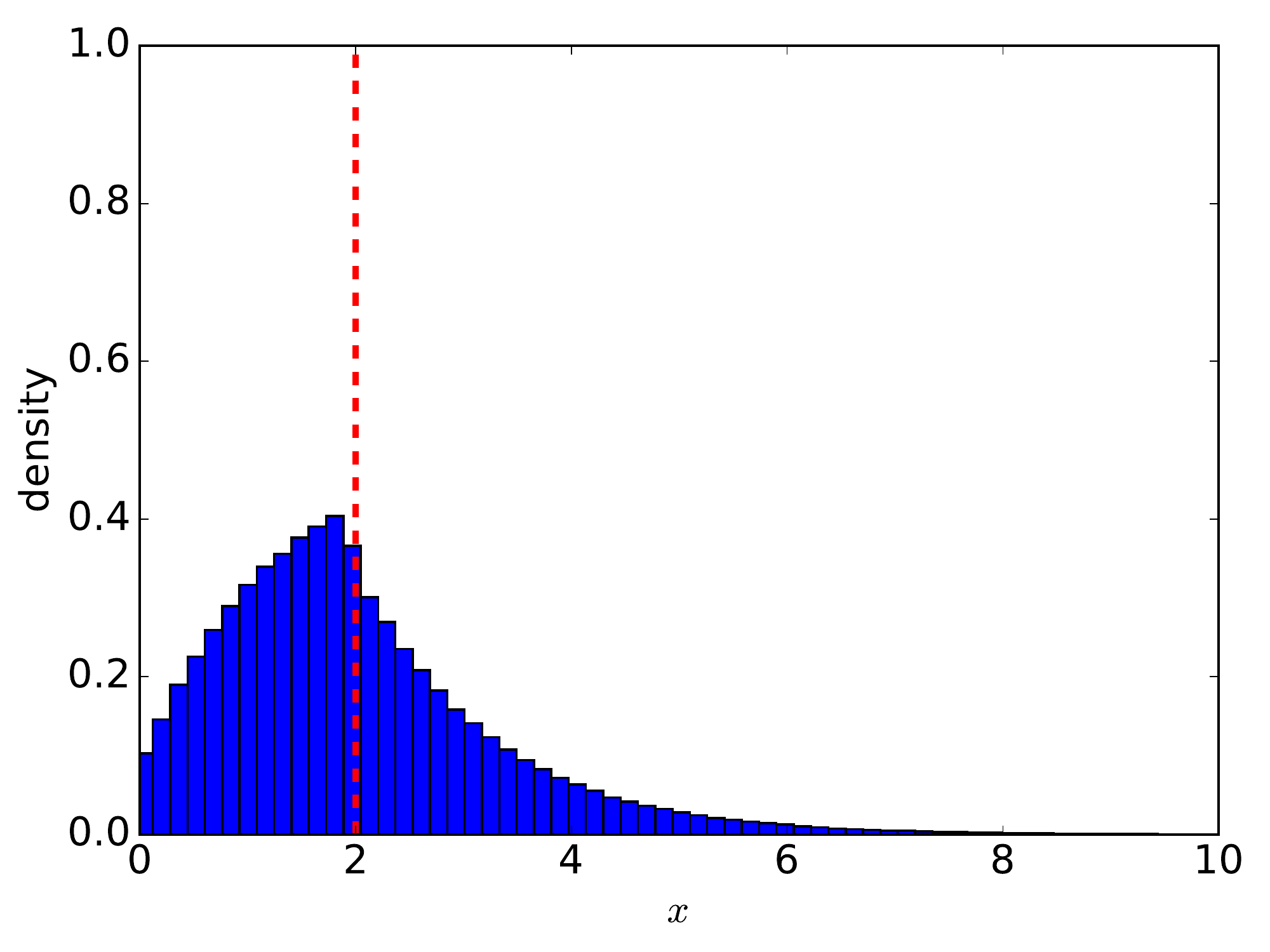}
  \caption{}
  \label{fig:test1-density-3}
\end{subfigure}
\caption{
{\label{fig:density-histo-alpha015}}
Density for  $a = 0.01125, X_0 = 2.0$~: analytical solution~\eqref{eq:FPstationary-solutionp} (left) and its approximation using $10^6$ Monte-Carlo samples following the dynamics~\eqref{eq: PS} (middle) and
~\eqref{eq: MFsta-PS}  (right), for $t = 100$.
In red (dashed line) are plotted respectively $x_0$, and the empirical averages $\overline{X}_{t}$ and $\overline{\mathcal{X}}_{t}$.
}
\end{figure}

\begin{figure}[H]
\centering
\begin{subfigure}{.33\textwidth}
  \centering
\includegraphics[scale=0.27]{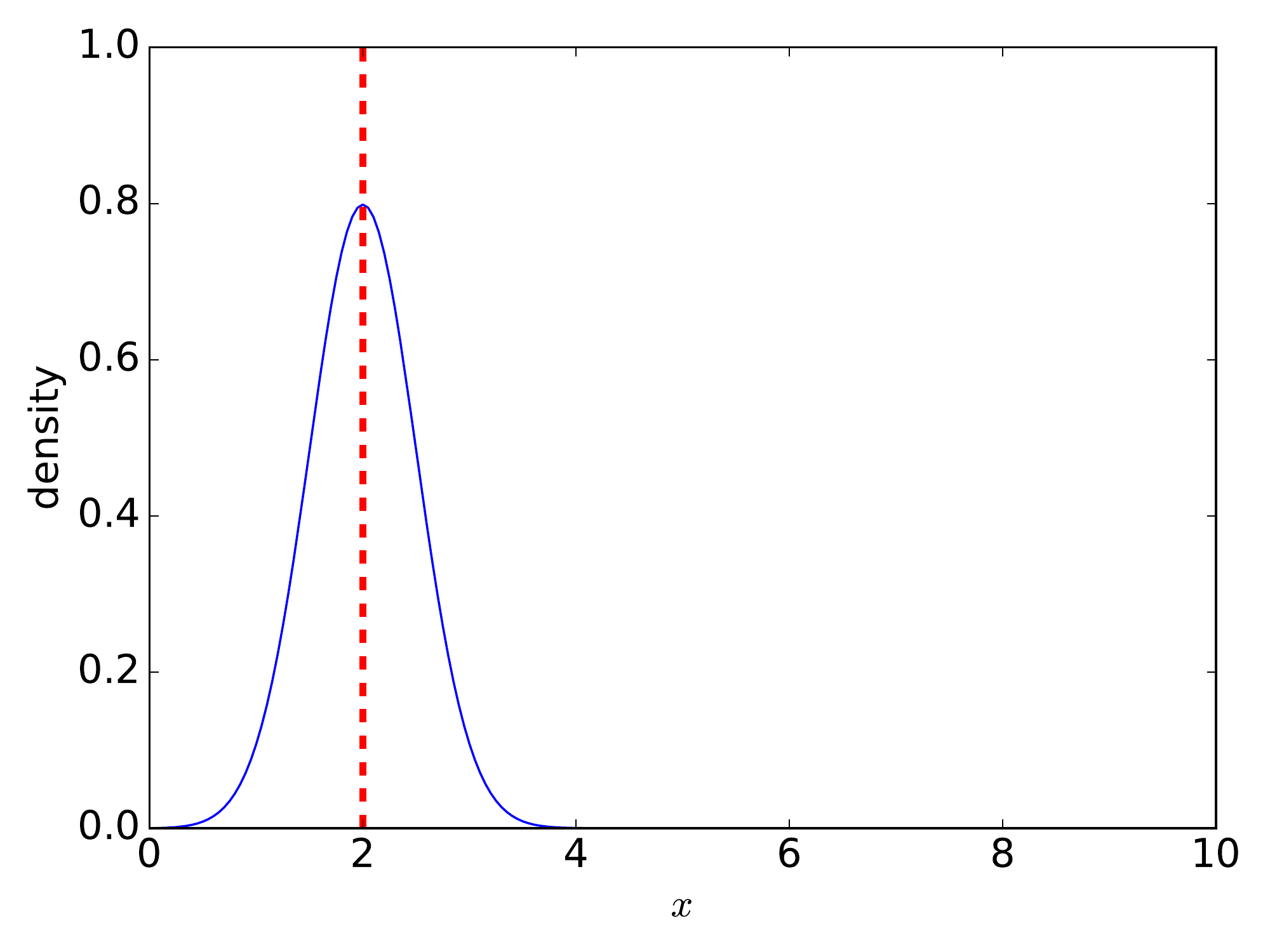}
  \caption{}
  \label{fig:test2-density-1}
\end{subfigure}%
 \hspace{-0.2cm}
\begin{subfigure}{.33\textwidth}
  \centering
 \includegraphics[scale=0.27]{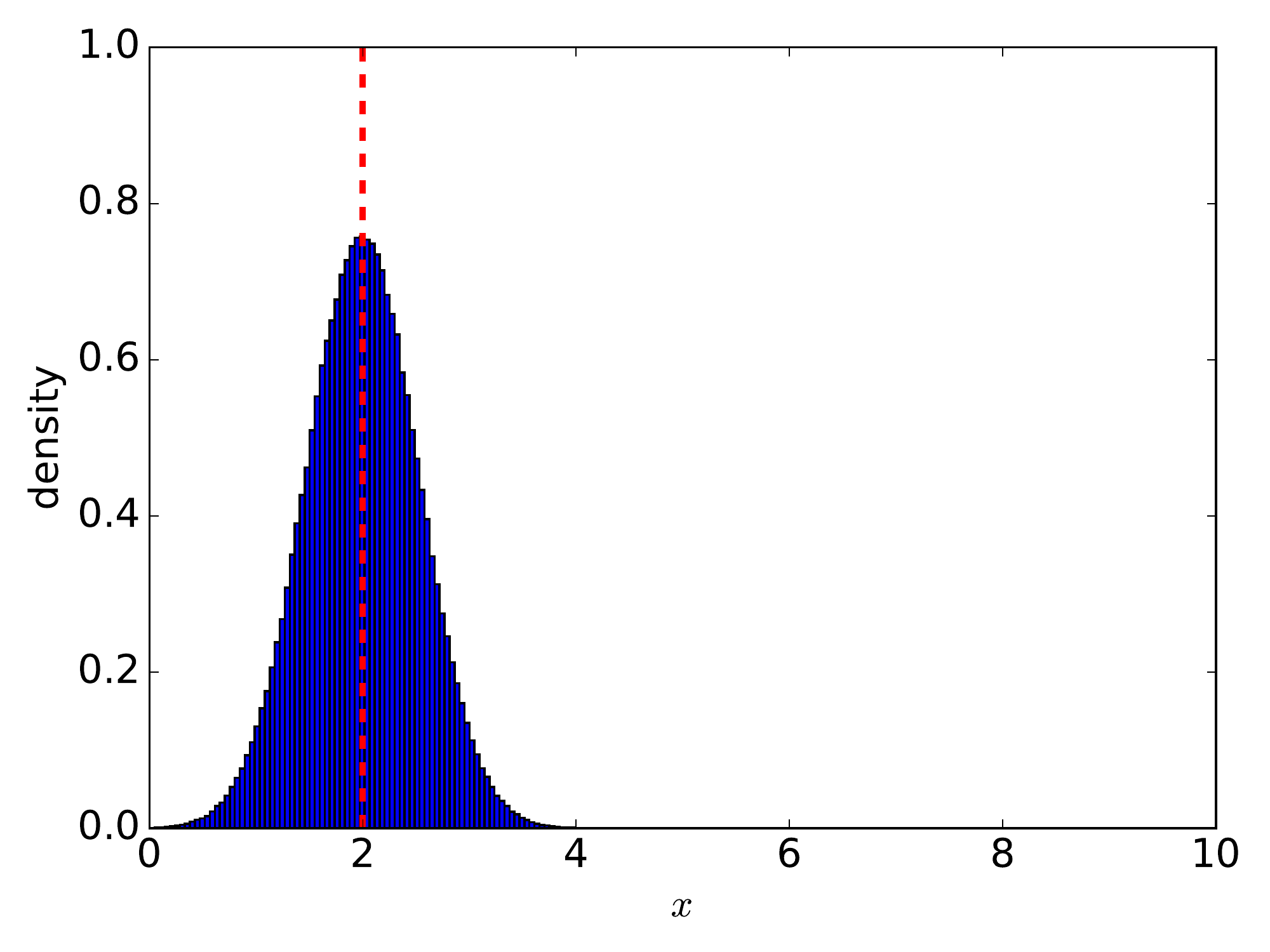}
  \caption{}
  \label{fig:test2-density-2}
\end{subfigure}%
 \hspace{-0.2cm}
\begin{subfigure}{.33\textwidth}
  \centering
 \includegraphics[scale=0.27]{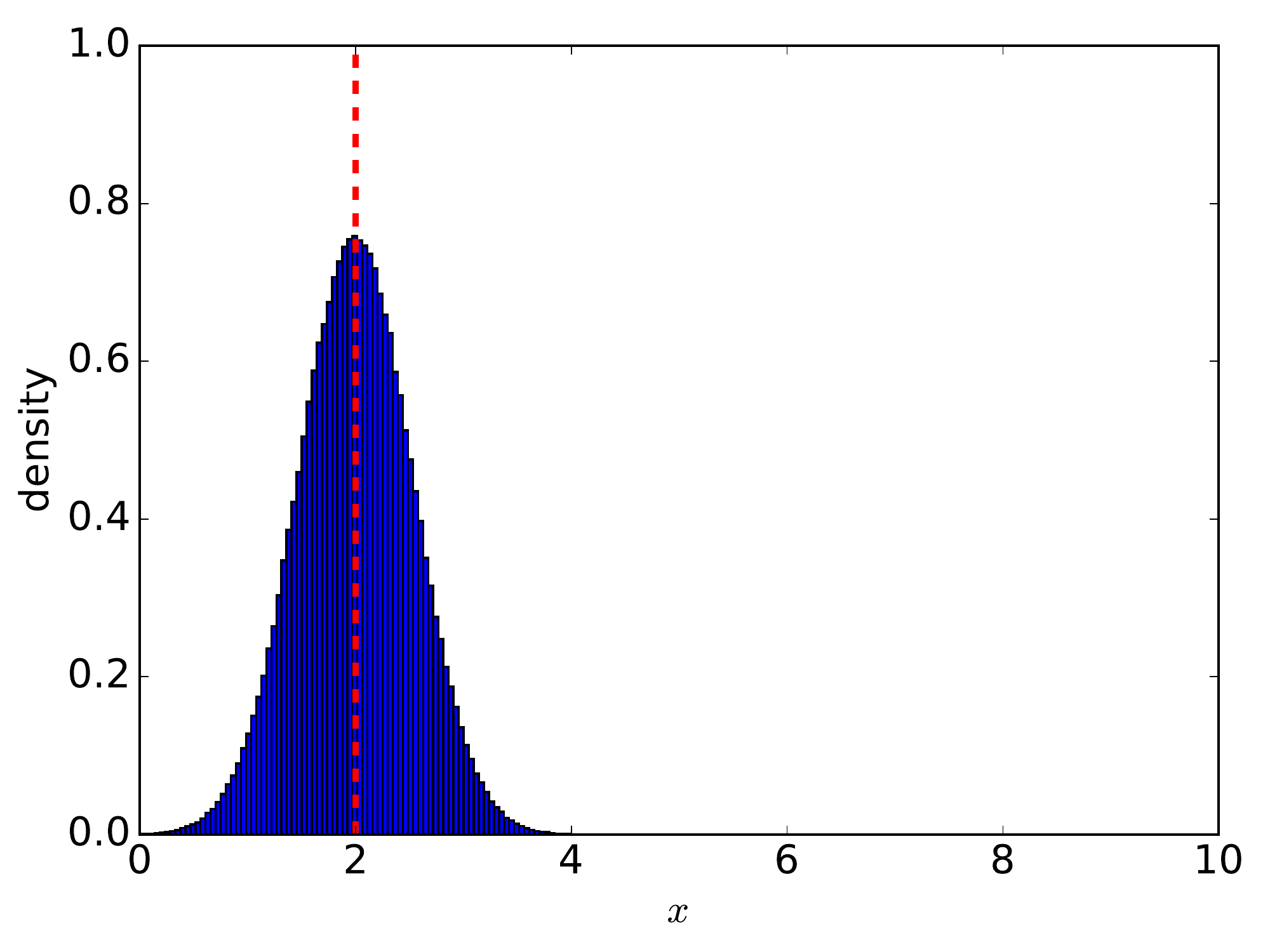}
  \caption{}
  \label{fig:test3-density-3}
\end{subfigure}
\caption{
{\label{fig:density-histo-alpha2}
Density for $a = 2.0, X_0 = 2.0$~: analytical solution~\eqref{eq:FPstationary-solutionp} (left) and its approximation using $10^6$ Monte-Carlo samples following the dynamics~\eqref{eq: PS} (middle) and 
~\eqref{eq: MFsta-PS} (right), for $t = 100$.
In red (dashed line) are plotted respectively $x_0$, and the empirical averages $\overline{X}_{t}$ and $\overline{\mathcal{X}}_{t}$.
}
}
\end{figure}

\subsection{Numerical approximation of the solution}

As shown in Figures~\ref{fig:test1-density-1}~and~\ref{fig:test2-density-1}, a larger value of $\alpha$ induces more concentration around $x_0$ in the density~\eqref{eq:FPstationary-solutionp}.

The numerical results also show that, when $t$ and $N$ are large,  this density is well approximated by the empirical distribution of Monte-Carlo simulation of trajectories according to either~\eqref{eq: PS} or an $N$-particle approximation of the mean field dynamics~\eqref{eq: MFsta} with states $\mathcal X^i_., \mathcal M^i_.$, namely:
\begin{equation} \label{eq: MFsta-PS} 
\begin{split}
	\mathcal X_{t}^{i} \, &=\, \mathcal X_{0}^i + \int^{t}_{0} b (\mathcal X_{s}^{i} , x_{0}) {\mathrm d} s + W_{t}^{i} + \int^{t}_{0} x_{0} \Big( {\mathrm d} \mathcal M_{s}^{i} - \frac{1}{\, N\, } \sum_{j \neq i} {\mathrm d} \mathcal M_{s}^{j} \Big) \, ; \quad t \ge 0 \, ,  
	\\
	\mathcal M_{t}^{i} \, &:=\, \sum_{k=1 }^{\infty} {\bf 1}_{\{\tau_{k}^{i} \le t\}} \, , \quad \tau_{k}^{i} \, :=\, \inf \Big \{ s > \tau_{k-1}^{i} : \, \mathcal X_{s-}^{i} - \frac{x_{0}}{N} \sum_{j \neq i } (\mathcal M_{s}^{j} - \mathcal M_{s-}^{j}) \, \le \,  0 \Big\} \, ; \quad k \in \mathbb N \, , 
\end{split}
\end{equation}
Conversely, it also means that the long time behavior of both systems of SDEs is well approximated by the mean field dynamics~\eqref{eq:FPstationary-solutionp} in the stationary regime.

Figure~\ref{fig:average} displays the evolution of the average liquidity amounts, $\overline X$ and $\overline{\mathcal{X}} = \frac{1}{N} \sum_{i=1}^N\mathcal{X}^i$. Although the systems are expected to have the same mean-field limit in which the mean is constant, here the values fluctuate around $x_0 = 2.0$ because there is only a finite number of agents. Note that the fluctuations are the smallest when using a large $a$ (hence a strong mean-reverting effect) and a fixed level for births (namely, $x_0$).

For the dynamics~\eqref{eq: PS} (respectively~\eqref{eq: MFsta-PS}), the default rate is given by ${\mathrm d} \left(\sum_i M^{i}\right) / {\mathrm d} t$ (resp. ${\mathrm d} \left(\sum_i \mathcal{M}^{i}\right) / {\mathrm d} t$), whereas the cumulative number of defaults is simply $\sum_i M^{i}$ (resp. $\sum_i \mathcal{M}^{i}$). Their evolutions are displayed on Figure~\ref{fig:deaths}. At time $t=0$, the default rate is $0$ because the initial distribution is $X_0 = x_0 >0$ a.s.; then, it converges towards a stationary value.

\begin{figure}[H]
 \begin{center}
 \includegraphics[scale=0.4]{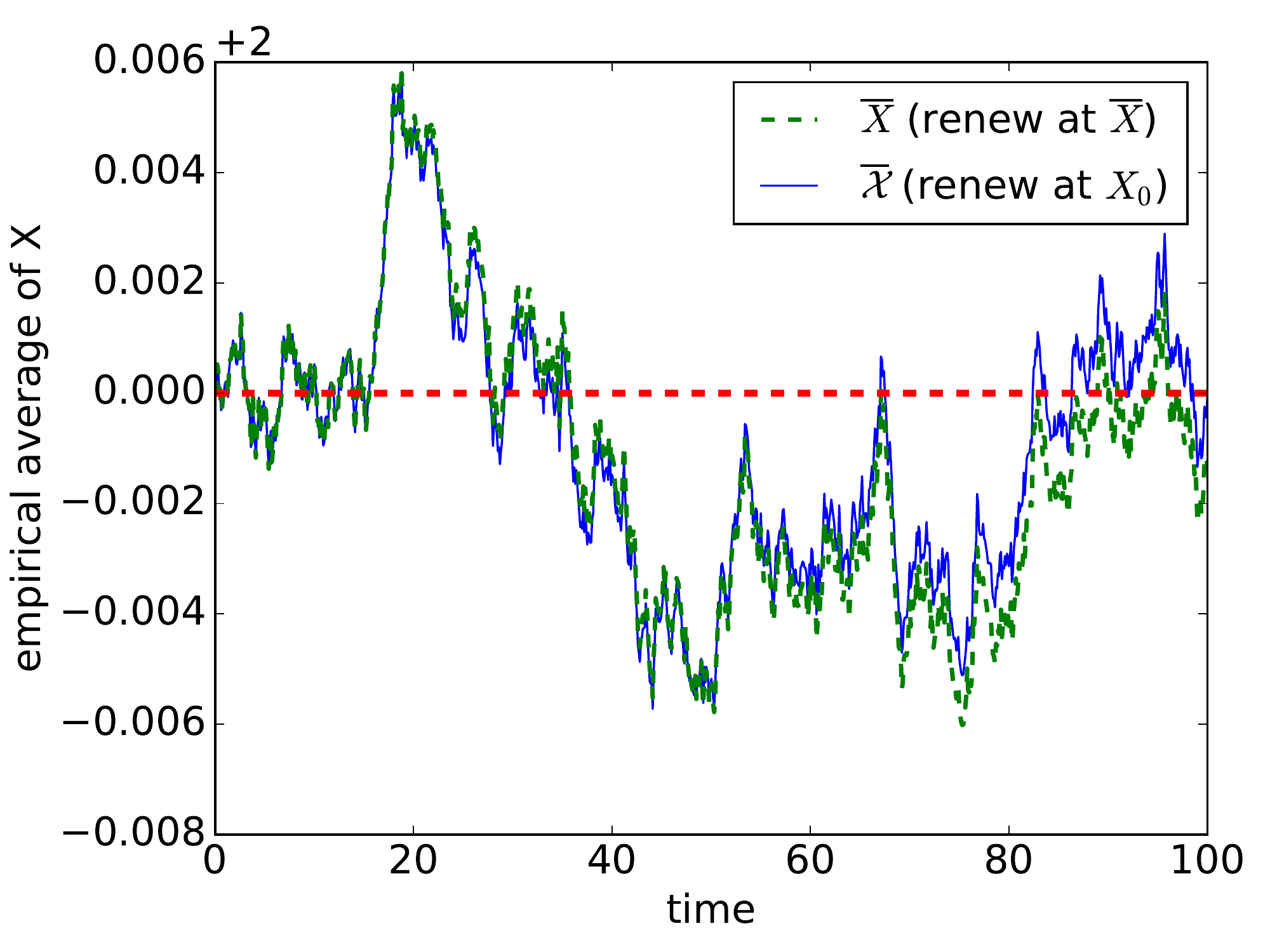}
 \includegraphics[scale=0.4]{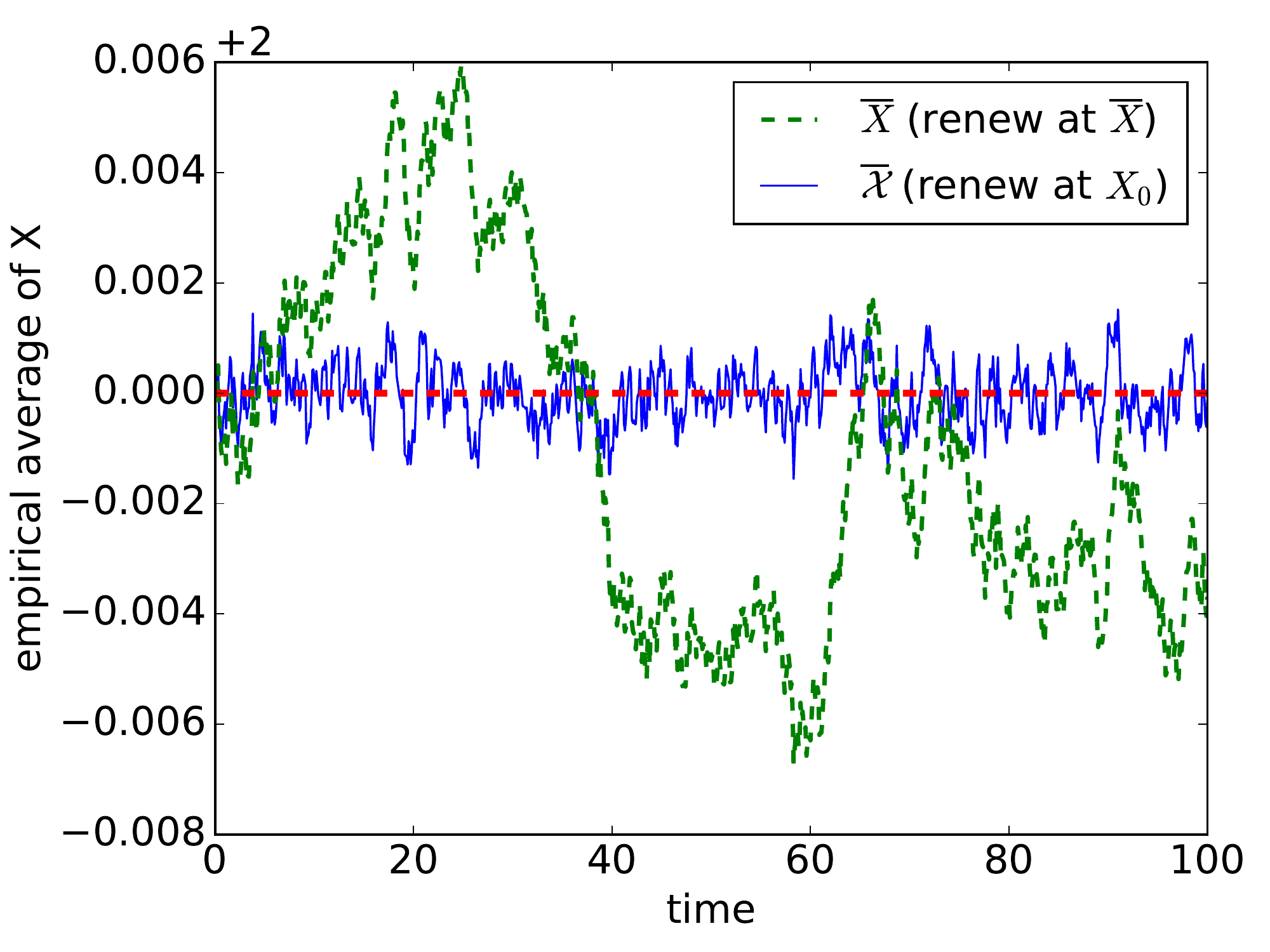}
  \end{center}
\caption{
{\label{fig:average}}
Evolution of the empirical average of the $N$-dimensional process $X$, for $N = 10^6$ Monte-Carlo samples following the dynamics~\eqref{eq: PS} 
(in green, dashed curve) and ~\eqref{eq: MFsta-PS} (in blue), with $a = 0.01125$ (left) and $a = 2.0$ (right). In both cases, the initial distribution is concentrated at $X_0 = 2.0$ (dashed line, in red).
}
\end{figure}

\begin{figure}[H]
 \begin{center}
 \includegraphics[scale=0.4]{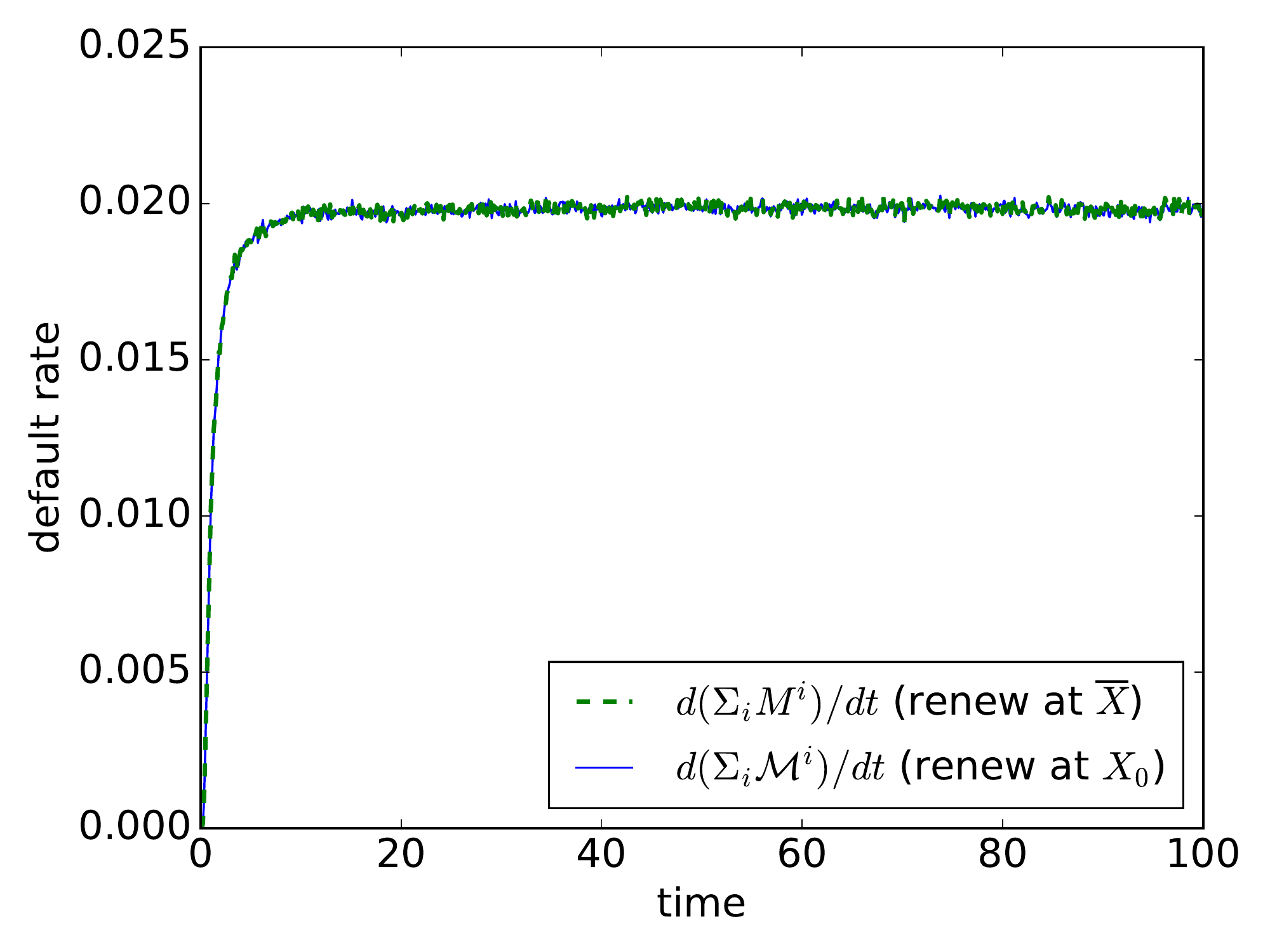}
 \includegraphics[scale=0.4]{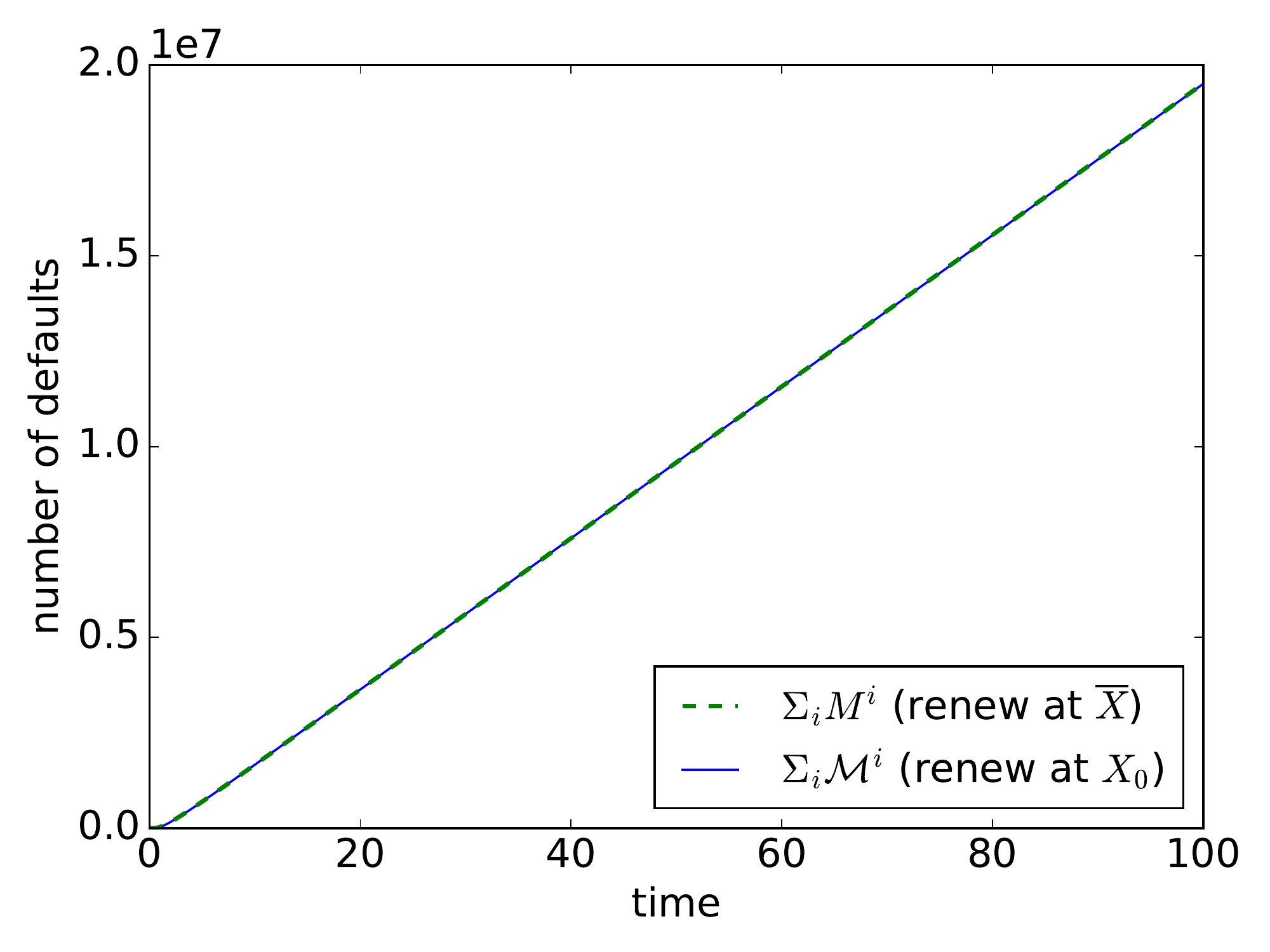}
  \end{center}
\caption{
{\label{fig:deaths}}
Evolutions of the default rate  (left) and of the cumulative number of defaults (right), for the dynamics~\eqref{eq: PS} (resp.~\eqref{eq: MFsta-PS}), with $N = 10^6$ Monte-Carlo samples, $a = 0.01125$ and $X_0 = 2.0$. The curves correspond to the dynamics~\eqref{eq: PS} (in green, dashed curve) and~\eqref{eq: MFsta-PS}  (in blue).
}
\end{figure}



\section{Mean Field Game Model} \label{sec: 3}

We now turn our attention to the situation where the agents can control their dynamics and try to optimize a certain criterion. Based on the previous section, we focus directly on the macroscopic description. We want to study Nash equilibria and, to this end, we use the framework of mean field games. The players interact through two mean field terms: the average wealth and the rate of defaults.

\subsection{Formulation of the problem} \label{sec: 3.1}

In order to describe the mean field game problem, we will use the following notations. Let $\mathcal P$ 
be the family of probability densities over $\RR_+ = [0, +\infty)$ with a right-hand derivative at $x=0$. For a probability density $m \in \mathcal P$, we define its first moment 
\begin{equation} \label{eq: mbf}
	\mbf(m) = \int_0^{+\infty} x m(x) dx \, ,
\end{equation}
and, 
\begin{equation} \label{eq: ebf}
	\ebf(m) = \frac{\sigma^2}{2} \frac{d}{d x} m(0) \, . 
\end{equation}
For a random element $Z$, we denote by $\mathcal L(Z)$ the law of $Z$. In particular, we consider a flow of density function $m_{t} $ of the marginal law at time $t \ge 0$ for a generic stochastic process.  Here and thereafter, to allow more flexibility in the numerical investigation, we do not assume anymore that the diffusion coefficient $\sigma$ is necessarily $1$.

\paragraph{\textbf{Controlled dynamics. } }
For each flow of densities $m = (m_t)_{t \geq 0}$ and each control $\xi = (\xi_t)_{t \geq 0}$ with $m_t  : \RR_+ \to \RR$ and $\xi_t  : \RR_+ \to \RR$ for each $t \geq 0$, we consider the following two dynamics $ (X_{t}^{m, \xi} )_{t \ge 0 }$ and $(\tilde X_{t}^{\xi})_{t \ge 0 }$ ~:
\begin{align}
	X_t^{m, \xi} 
	& = X_0 + \int_0^t b\left(X^{m, \xi}_s, m_s, \xi_s(X^{m, \xi}_s)\right) dt + \int_0^t \sigma d W_s \, ; \quad t \geq 0 \, ,
	\label{eq:controlledDyn}
	\\
	\tilde X_t^{\xi} 
	&= X_0 + \int_0^t b\left(\tilde X^{\xi}_s, \mathcal L(\tilde X^{\xi}_s), \xi_s(\tilde X^{\xi}_s)\right) dt + \int_0^t \mbf\left(\mathcal L(\tilde X^{\xi}_s)\right) d \tilde M^\xi_s + \int_0^t \sigma d W_s \, ; \quad t \geq 0 \, ,
	\label{eq:controlledDyn-meanfield}
\end{align}
where the initial value $X_{0}$ is a positive random variable, $\,W\,$ is a standard Brownian motion, and $ \tilde M^{\xi}_{t}$ is the cumulative number of defaults on or before time $t$ which occur at $\tau^{k, \xi}$, $k \ge 1$, i.e.,  
$$
	\tilde M^{\xi}_{t} \, =\, \sum_{k=1}^{\infty} {\bf 1}_{\{\tilde \tau^{k,\xi} \le t \}} \,,
	\quad \tilde \tau^{k,\xi} \, =\, \inf \{ s > \tilde \tau^{k-1,\xi} : \tilde X^{\xi}_{t-}  \le 0 \} \, , \, k \ge 1 \, ,
	\quad \tilde \tau^{0,\xi} \, =\,  0 \, .
$$ 
With (\ref{eq: mbf})-(\ref{eq: ebf}) the drift functional $b$ is defined by
$$
	b(x, m, \xi) = \xi + a (\mbf(m) -x) - {  \gamma \ebf(m) \mbf(m) } \, ,
$$
{where $\gamma \in (0,1]$ is a constant. The case $\xi=0$ and $\gamma=1$ corresponds to the uncontrolled dynamics studied in the previous section. }

Furthermore, for the expression~\eqref{eq:controlledDyn-meanfield} to make sense, we assume that $\mathcal L(\tilde X^{\xi}_s)$ has a density with respect to Lebesgue measure, and we identify it with this density.

The interpretation of the above two processes $ (X_{t}^{m, \xi} )_{t \ge 0 }$ and $(\tilde X_{t}^{\xi})_{t \ge 0 }$ is the following: $X^{m, \xi}$ describes the state of an infinitesimal player using control $\xi$ when the flow of densities of the rest of the population is given by $m$, while $\mathcal L(\tilde X^{\xi})$ describes the state of the population when each (infinitesimal) player uses the control $\xi$.
Here, we assume that each bank controls its rate of borrowing or lending to a central bank through the (feedback) control rate $\xi : \RR_+ \times \RR_+ \to \RR$ as a function of time and state. 

	Notice that the stochastic differential equation of $\tilde X^{\xi}$ is of the non-linear, McKean-Vlasov type, whereas the dynamics of $X^{m, \xi}$ is not, since $m$ is fixed.
	Moreover, due to the control $\xi$, $\E [ \tilde X^{m, \xi}_t ]$ is not necessarily constant.

\paragraph{\textbf{Objective function. } }
Following  \textsc{Carmona} et al. in~\cite{MR3325083} for a systemic risk model, we consider a running cost $f$ defined by
\begin{equation}
	f(x, m, \xi) =  \frac 12 \xi^2 - q \xi (\mbf(m) - x) + \frac\epsilon 2(\mbf(m) - x)^2 \, .
\end{equation} 
where $q > 0$ and $\epsilon > 0$ are parameters of the problem. 
Here, the parameter $q$ is interpreted as the incentive to borrowing or lending: The bank will borrow ($\xi >0$), if $x < \mbf(m)$ and lend ($\xi <0$), if $x < \mbf(m)$. Equivalently, after dividing by $q$, this parameter can be seen as a control by the financial regulator of the cost of borrowing or lending ($q$ large meaning low fees). Furthermore, the quadratic term $(\mbf(m) - x)^2$ in the running cost penalizes deviations from the average. We assume that $q^2 \leq \epsilon$, so that for a fixed $m$, $(x,\xi) \mapsto f(x,m,\xi)$ is convex. Notice, as a special case, that if $q^2 = \epsilon$, then it is simply $f(x, m, \xi) = \frac 12 \left( \xi - q (\mbf(m) - x)\right)^2$. Let $r>0$ be an instantaneous discounting rate parameter and $m_0$ be an initial probability density.

We introduce the following objective function (or cost functional) for each infinitesimal player with state process $(X^{m, \xi}_{t})_{t\ge 0}$ in (\ref{eq:controlledDyn}). For a density flow $m = (m_t)_{t\geq 0}$ and an admissible control $\xi = (\xi_t)_{t\geq 0}$, we let the objective cost function  be~: 
\begin{equation}\label{def:def-J-MFG}
	J^{m}(\xi) = \E\left[ \int_{0}^{\tau^{m,\xi}} e^{-r s} f(X_s^{m,\xi}, m_s, \xi_s(X_s^{m,\xi})) ds \, \right], 
\end{equation}
where $X^{m,\xi}$ is given by~\eqref{eq:controlledDyn}, and $\tau^{m,\xi}$ denotes the first time $X^{m,\xi}$ hits $0$~:
\begin{equation} \label{eq: tau m xi}
	\tau^{m,\xi} \, =\, \inf \{ s > 0 : X^{m,\xi}_{s-}  \le 0 \} \, .
\end{equation}
This is the cost that a representative (and infinitesimal) player with dynamics~\eqref{eq:controlledDyn} tries to minimize, when the dynamics of the population is described by $m$.

\paragraph{\textbf{Mean field game. }} The mean field game we consider is defined as the problem of finding $(\hat m, \hat \xi)$ such that the following two conditions are satisfied:
\begin{enumerate}
	\item $\hat \xi$ minimizes $J^{\hat m}$;
	\item $\hat m_t = \Law( \tilde X^{\hat \xi}_t)$, for all $t \geq 0$ .
\end{enumerate}
The first condition means that $\hat \xi$ is the best response control of an infinitesimal player facing the population whose distribution is given by $\hat m$. The second condition ensures consistency of the best response and the population's behavior. Overall, the mean field game can be construed as a fixed point problem.

We will also be interested in a similar game but with finite time horizon. Although this truncation might be a bit artificial, it is more convenient in order to solve numerically the problem. We thus fix a time horizon $0 < T < +\infty$ and consider, instead of~\eqref{def:def-J-MFG}, the following objective function
\begin{equation}\label{def:def-J-MFG-T}
	J^{m}_T(\xi) = \E\left[ \int_{0}^{\tau^{m,\xi} \wedge T} e^{-r s} f(X_s^{m,\xi}, m_s, \xi_s) ds \, \right] \, .
\end{equation}
We expect that for each $\xi = (\xi_t)_{t \geq 0}$, $J^{m}_T(\xi) \to J^{m}(\xi)$ as $T \to \infty$. As above, the solution to the mean field game on a finite time horizon is defined as a fixed point.

 We leave for future work the challenging questions of existence and uniqueness of a solution in the general setting. We will however provide below an explicit solution in a special case and then we compute an (approximate) solution for a discrete version of the problem in the general case.

\subsection{PDE system for the Mean Field Game}

We now want to characterize the MFG solutions in the form of a PDE system.
Following the approach of \textsc{Lasry and Lions}~\cite{MR2295621}, we obtain the system consisting of a forward Fokker-Plank (FP) equation for the evolution of the density of the population, and a backward Hamilton-Jacobi-Bellman (HJB) equation for the evolution of the value function of an infinitesimal player.

We define the Hamiltonian: for $x \in \RR_+$, $m \in \mathcal{P}$ and $p \in \RR$,
\begin{align}
\label{eq:def-H}
	H(x, m, p) 
	&= - \inf_{\xi} \left\{ b(x, m, \xi) p + f(x, m, \xi)\right\}
\end{align}
The infimum above is achieved by $\xi = \Xi(x,m,p) = -p - q (x - \mbf(m))$.
Hence
\begin{align*}
	H(x, m, p) 
	&= \frac{1}{2} p^2 - \big[(q+a)(\mbf(m)-x) - {\gamma \ebf(m) \mbf(m) } \big] p  + \frac{1}{2} (q^2-\epsilon) (\mbf(m)-x)^2 \, .
\end{align*}
To alleviate the notations, let us also introduce:
\begin{align*}
	\varphi(x, m) & = (q+a)(\mbf(m)-x) - { \gamma \ebf(m) \mbf(m) } \, , 
	\\
	\psi(x, m) & = \frac{1}{2} \left[  \varphi(x, m)^2 - (q^2-\epsilon) (\mbf(m)-x)^2\right],
\end{align*}
where $x \in \RR_+$ and $m$ is a density over $\RR_+$.
Then, the Hamiltonian can be written as
\begin{align}
\label{eq:H-quadratic}
	H(x, m, p) 
	& = \frac{1}{2} \left[ p - \varphi(x, m) \right]^2 - \psi(x,m) \, .
\end{align}

It can be shown that if $\hat \xi = (\hat \xi_t)_{t\ge 0 }$ is an optimal feedback control, then it must satisfy, for all $(t,x)$,
\begin{equation}
\label{eq:opt-ctrl-xi-u}
	\hat \xi(t,x) = \Xi(x, \hat m(t,\cdot), \partial_x \hat u(t,x) ) = - \partial_x \hat u(t,x),
\end{equation}
where $(\hat u, \hat m)$ solve the following HJB-FP PDE system~: for $(t,x) \in (0,+\infty) \times (0,+\infty)$,
\begin{subequations}
     \begin{empheq}[left=\empheqlbrace]{align}
	&\displaystyle r u(t,x) - \partial_t u(t,x) - \frac{\sigma^2}{2} \partial_{xx} u(t,x) + H(x, m(t,\cdot), \partial_x u(t,x)) = 0,
	\label{eq:sysH-HJB}
	\\
	 &\displaystyle \partial_t m(t,x) - \frac{\sigma^2}{2} \partial_{xx} m(t,x)
	 - \partial_x \big( H_p(x, m(t, \cdot), \partial_x u(t,x)) m(t,x)\big) 
	  = \ebf(m(t, \cdot)) \delta_{\mbf(m(t, \cdot))}(x),
	 \label{eq:sysH-FP} 
\end{empheq}
\end{subequations}
with the boundary conditions~: for all $t \in (0,+\infty)$,
 \b*
	 u(t,0) = 0 \qquad \mbox{and} \qquad  m(t,0)=0,
 \e*
and the initial and final conditions~: for all $x \in [0,+\infty)$,
 \begin{equation}
 \label{eq:hjbfp-init-fin}
	 m(0,x) = m_0(x) \, \qquad u(T,x) = 0 \, .
 \end{equation}

\subsection{Explicit solution for a stationary Mean Field Game} \label{sec: 3.3}

In this section, we focus on a stationary regime. Building on our explicit solution for the dynamics (see Theorem~\ref{thm:characterize-statio-p-e0}), we provide an example of stationary mean field game with an explicit solution.
This is useful for instance as a benchmark in order to test numerical methods. 
To this end, we slightly modify the problem and add a non-trivial boundary condition on $u$ at $x=0$, which can be interpreted as an exit cost (or benefit) when the bank defaults. Instead of~\eqref{def:def-J-MFG-T}, we consider the objective function
\begin{equation}\label{def:def-J-MFG-T-exitcost}
	J^{m}_T(\xi) = \E\left[ \int_{0}^{\tau^{m,\xi} \wedge T} e^{-r s} f(X_s^{m,\xi}, m_s, \xi_s) ds + e^{-r \tau^{m,\xi}} \mathbf{1}_{\{\tau^{m,\xi} < + \infty\}} \Gamma(m(\tau^{m,\xi},\cdot))\, \right] \, ,
\end{equation}
for some function $\Gamma: [0, \infty) \times \mathcal P \to \mathbb R$ to be chosen below. Here, $\tau^{m, \xi}$ is the exit time defined in (\ref{eq: tau m xi}). In this case, the boundary condition for $u$ at $x=0$ becomes: for all $t \in (0,T)$
$$
	u(t,0) = \Gamma(m(t,\cdot)).
$$

We then look for a stationary solution.
 This leads us to consider the PDE system for the unknowns $(u,m)$, functions of $x$ only~: for $x \in (0,+\infty)$,
\begin{subequations}
     \begin{empheq}[left=\empheqlbrace]{align}
	%
	&r u(x) - \frac{\sigma^2}{2} u''(x) + H(x, m, u'(x)) = 0,
	\label{eq:sysH-HJB-statio}
	\\
	 &- \frac{\sigma^2}{2} m''(x)
	 - \partial_x \big( H_p(x, m, u'(x)) m(x)\big) 
	  = \ebf(m) \delta_{\mbf(m)}(x),
	\label{eq:sysH-FP-statio}
\end{empheq}
\end{subequations}
with the boundary conditions~:
 \b*
	 u(0) = \Gamma(m) \qquad \mbox{and} \qquad  m(0)=0 \, .
 \e*

Let us look for $u$ in the form:
\begin{equation}
\label{eq:ansatz-u-ABC}
	u(x) = \frac{1}{2} A(x - \mbf(m))^2 + B(x - \mbf(m)) + C
\end{equation}
where $A,B,C$ are three real numbers to be determined. With this ansatz, we have
\begin{align*}
	H(x, m, u'(x))
	&=
	\left[\frac{1}{2} \left( q^2 - \epsilon \right)  + (q+a)A + \frac{1}{2} A^2 \right] (x - \mbf(m))^2 
	\\
	&\qquad+ \left[ (q+a)B + AB + A { \gamma  \ebf(m)\mbf(m) } \right] (x - \mbf(m)) 
	\\
	&\qquad
	+ \left[ \frac{1}{2} B^2 + B { \gamma \ebf(m)\mbf(m) } \right],
\end{align*}
so, rewriting the HJB equation as a polynomial expression in $(x - \mbf(m))$ and identifying each coefficient to $0$ yields 
\begin{subequations}
     \begin{empheq}[left=\empheqlbrace]{align}
      	A &= \frac{ -(r+2(q+a)) + \sqrt{\Delta } }{ 2 }, \qquad \Delta  = [r + 2 (q+a)]^2- 4 (q^2 - \epsilon),
	\label{eq:sol-A-discount}
	\\
	B &= \frac{-A}{q+a+A+r} {\gamma  \ebf(m)\mbf(m) }
	\\
	C &= \frac{1}{r} \left( \frac{\sigma^2}{2} A -  \frac{1}{2}B^2 - B {\gamma \ebf(m)\mbf(m) } \right)
	\\
	\Gamma(m) &= \frac{1}{2} A \mbf(m)^2 - B \mbf(m) + C \, .
\end{empheq}
\end{subequations}
The last equality above comes from the boundary condition for $u$ at $x=0$. It can be ensured by suitably choosing $\Gamma$, provided we first find $A,B,C$ satisfying the three first equations. 
Assume that { $\gamma$ satisfies }
\begin{equation}
	\label{eq:MFG-FP-B-gamma}
	{ - B - \gamma = -1 } 
\end{equation}
{ i.e., $\gamma  = 1-A/(q+a+A+r)$. }

 Then (\ref{eq:sysH-FP-statio}) 
 rewrites
\begin{equation}
\label{eq:MFG-FPstatio-Q0}
	-\frac{\sigma^2}{2} m''(x) - (A+q+a) m(x)
	 - \Big[ (A+q+a) (x - \mbf(m)) + \ebf(m)\mbf(m)   \Big] m'(x)  = \ebf(m) \delta_{\mbf(m)}(x) \, ,
\end{equation}
which is known to have a closed form solution from Theorem~\ref{thm:characterize-statio-p-e0} on the uncontrolled stationary FP equation.

Note that, due to~\eqref{eq:MFG-FP-B-gamma}, $\gamma$ can not be $1$ unless $B = A = 0$, which leads to an explicit but trivial solution in the sense that the optimal control is simply $0$ in this case. However, for $\gamma \in (0,1)$, we obtain a non-trivial explicit stationary solution for which the equilibrium stationary control is given according to~\eqref{eq:opt-ctrl-xi-u} by
$$
	\hat{\xi}(x) = - \partial_x u(x) = - A (x - \mbf(m)) - B.
$$

\begin{figure}
\centering
\begin{tabular}{cc}
\includegraphics[scale=0.3]{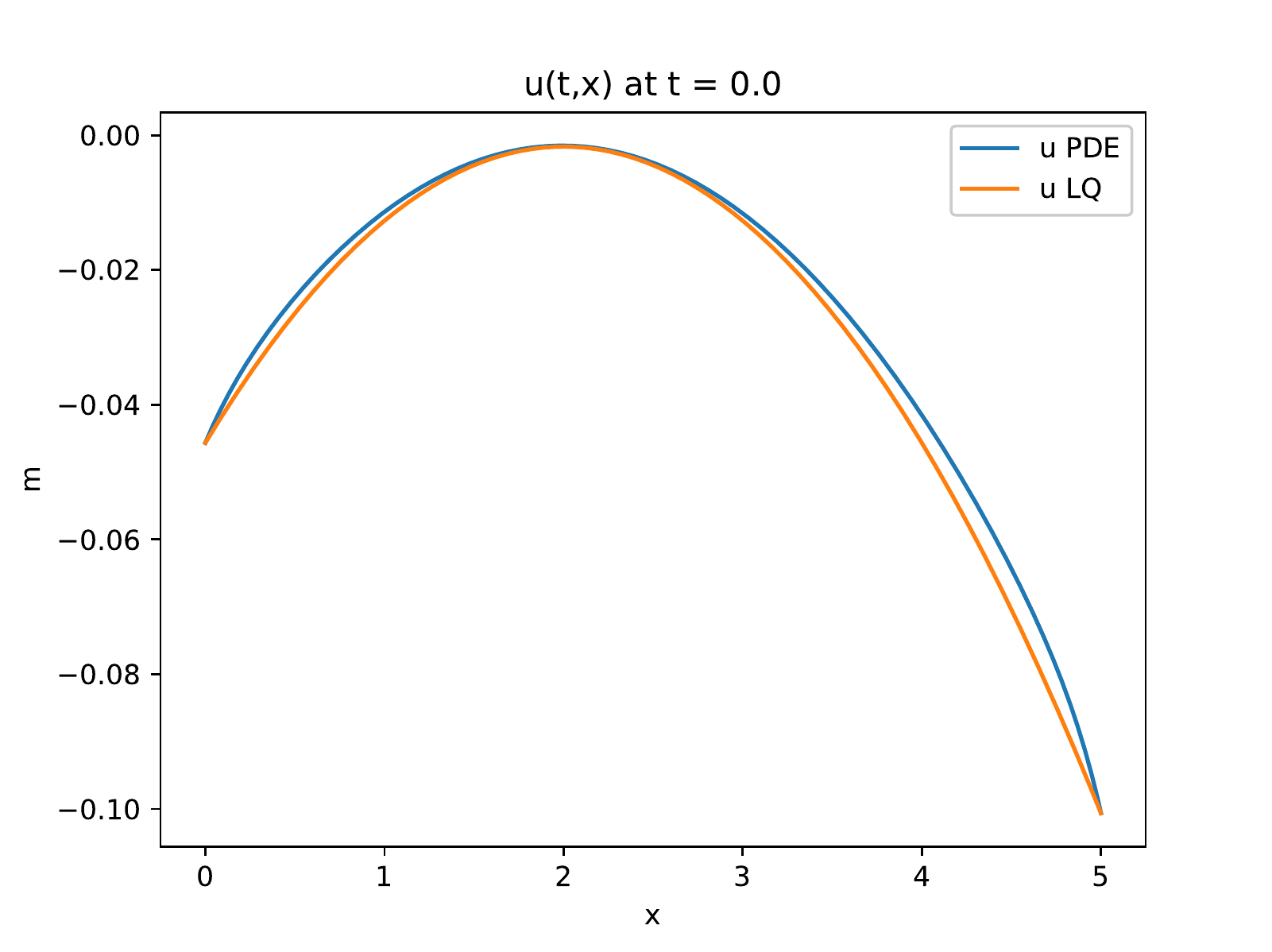} & 
  \includegraphics[scale=0.3]{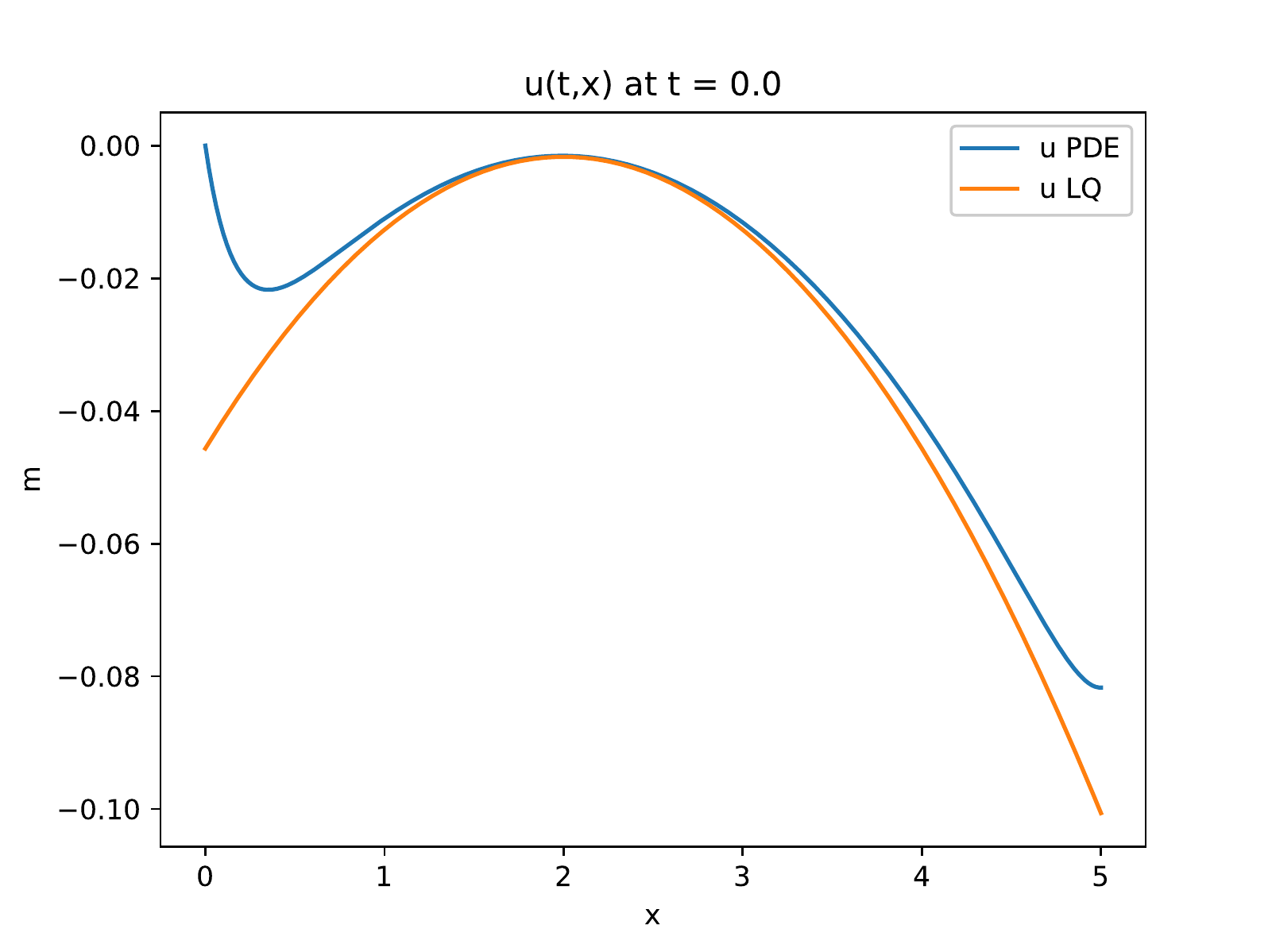} \\
%
{$t = 0$} & {$t = 0$}
\\
%
  \includegraphics[scale=0.3]{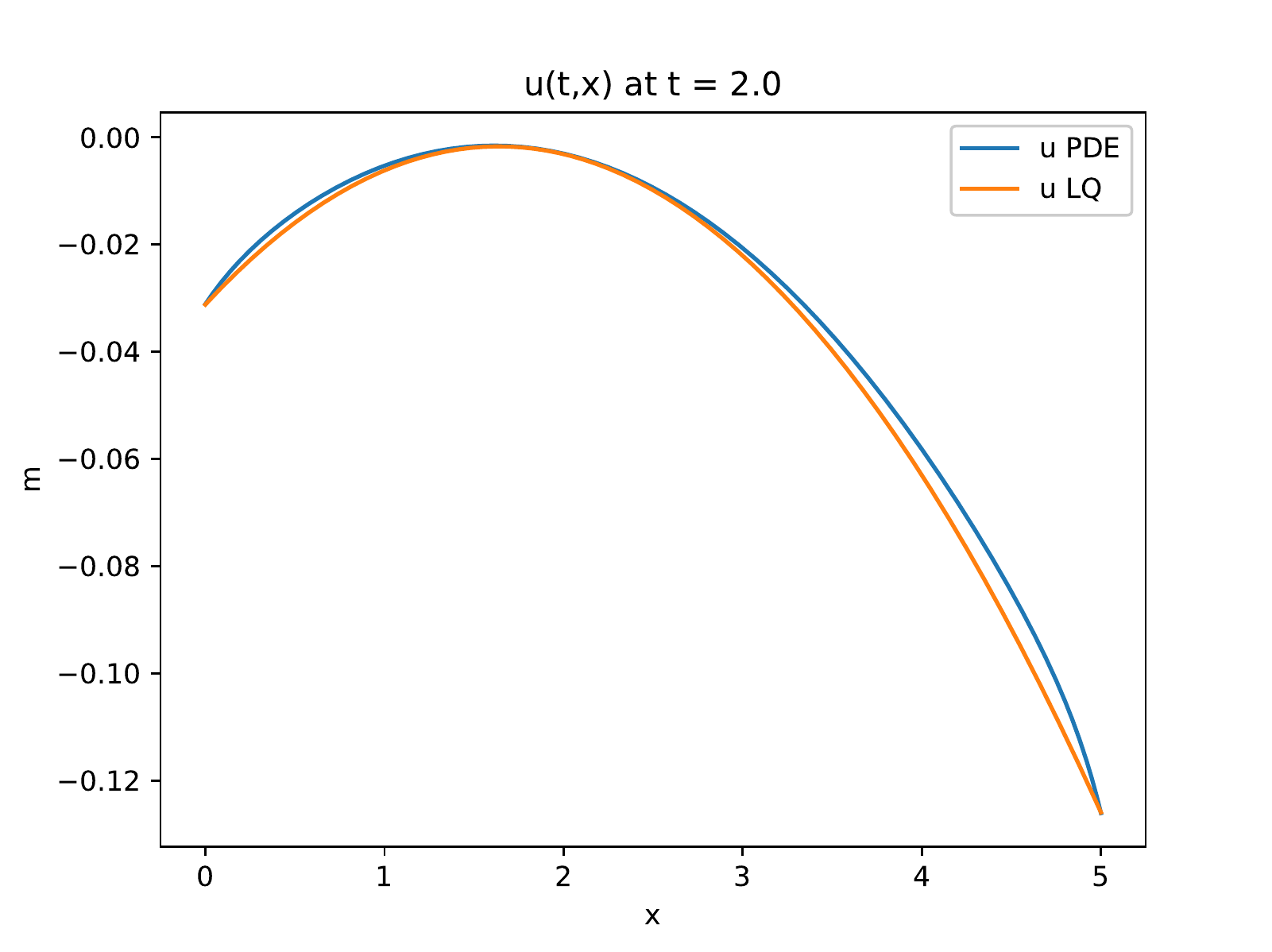} &
  \includegraphics[scale=0.3]{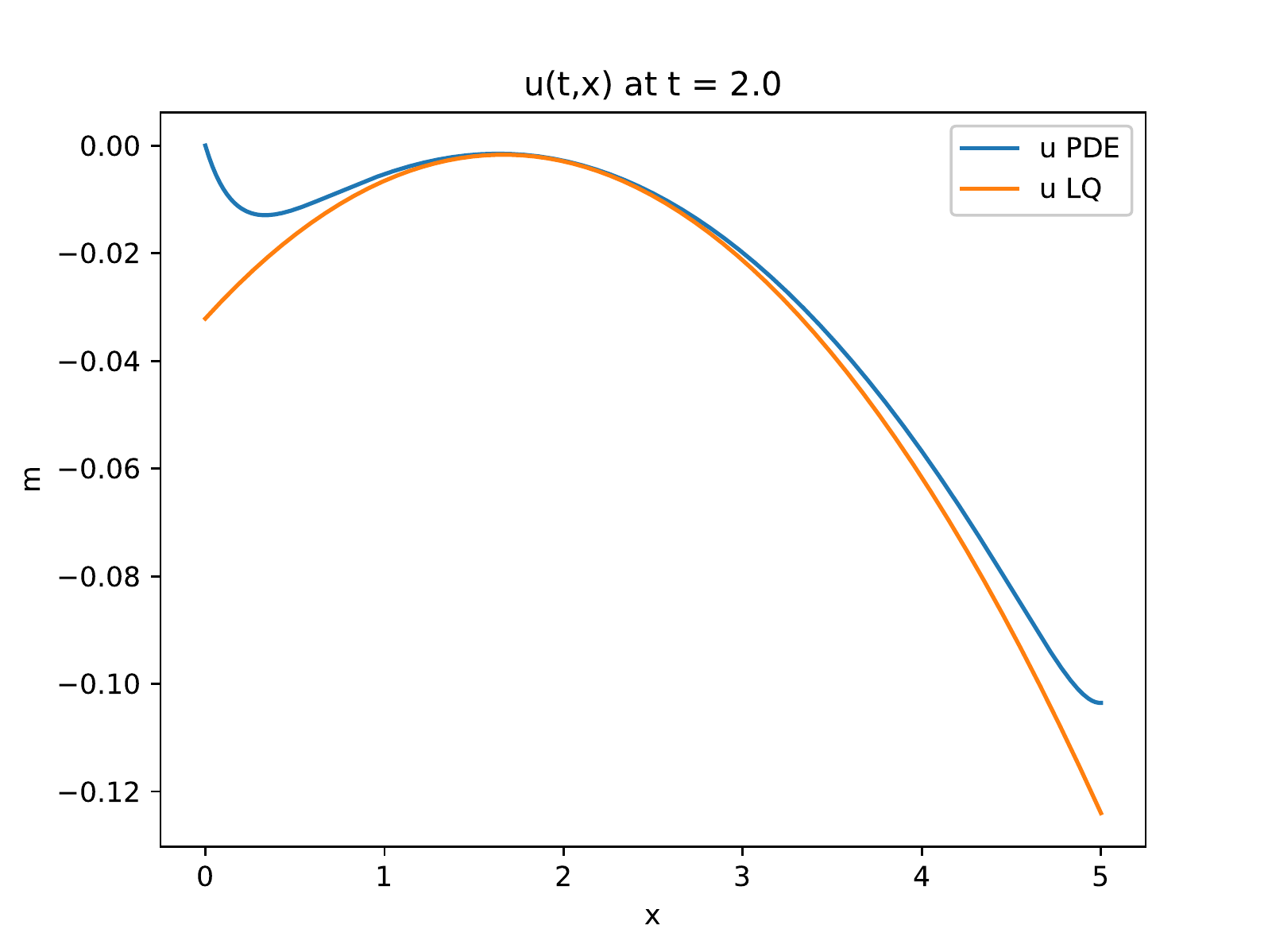} \\
%
{$t = 2$} & {$t = 2$} \\
& \\
%
  \includegraphics[scale=0.3]{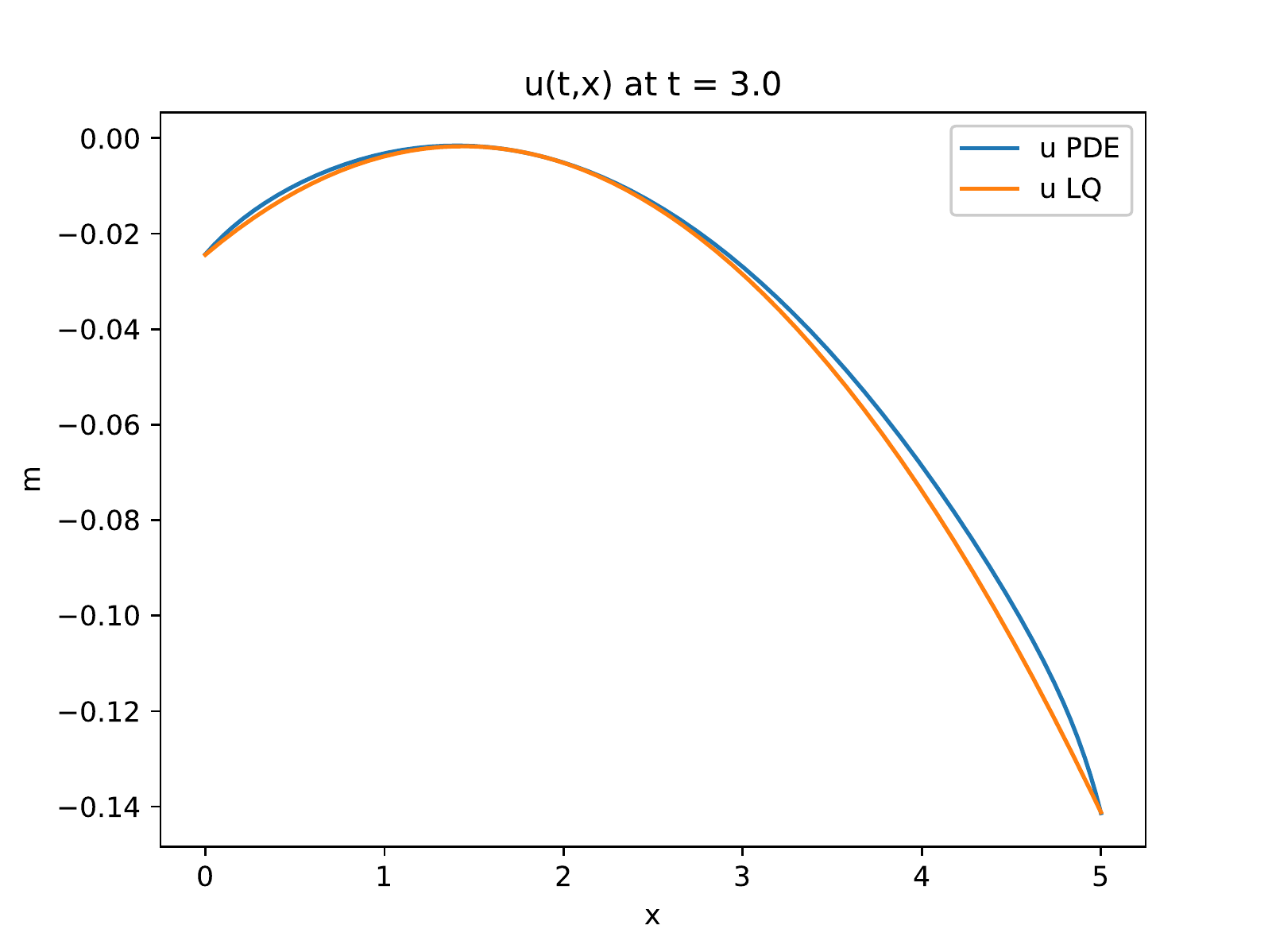} & 
  \includegraphics[scale=0.3]{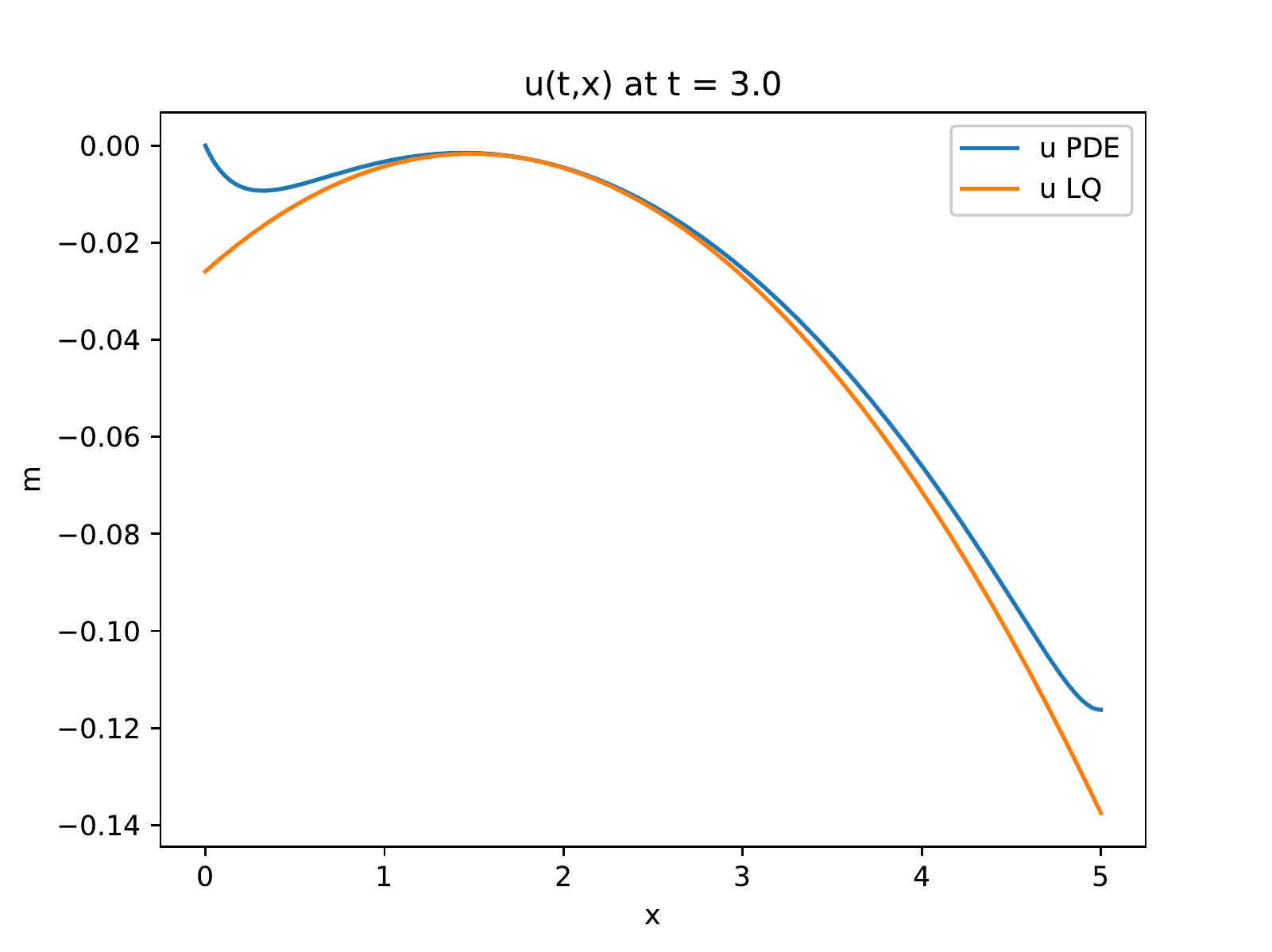} \\
%
{$t = 3$}  & {$t = 3$}  \\
&  \\
%
  \includegraphics[scale=0.3]{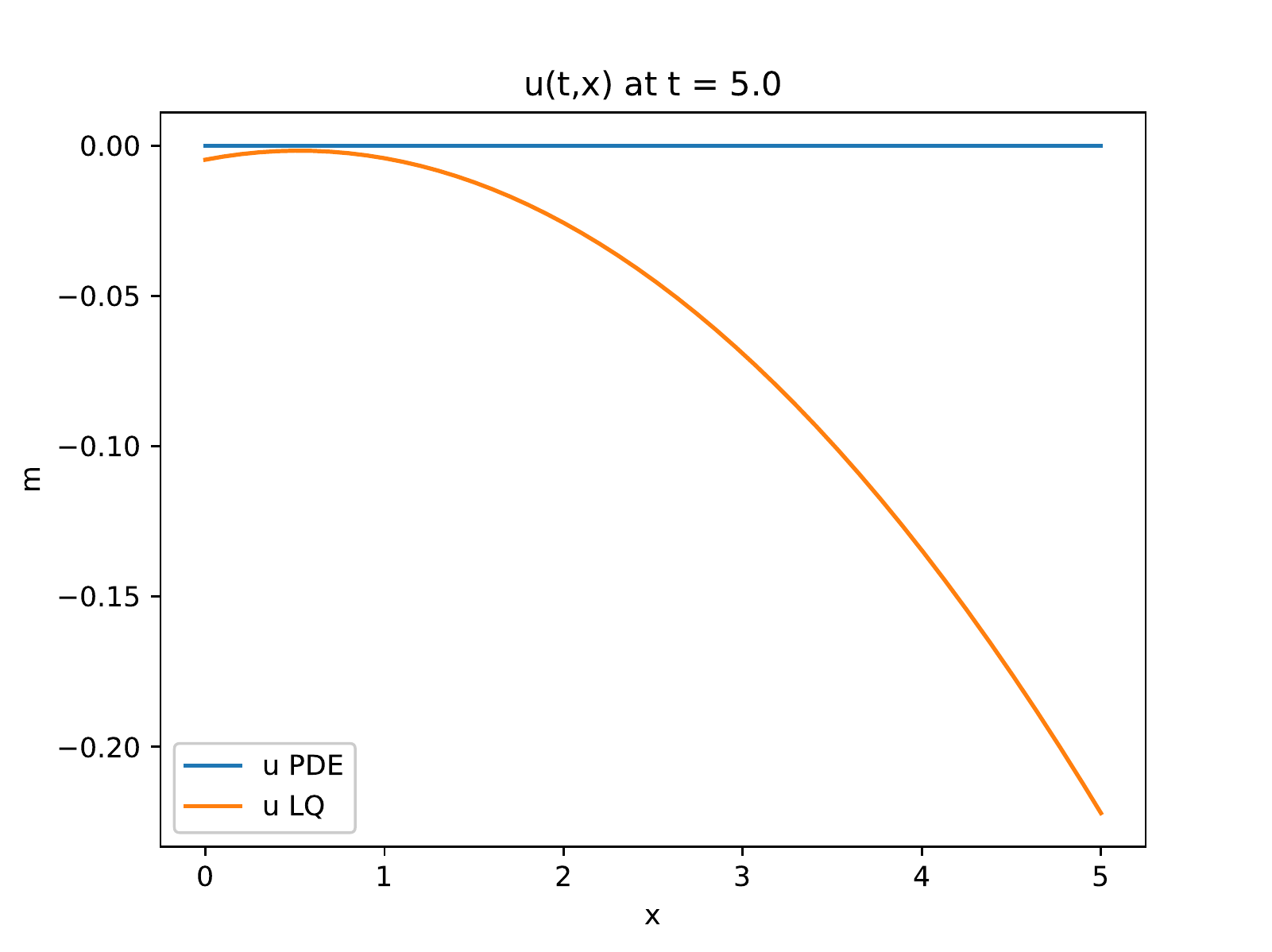} & 
  \includegraphics[scale=0.3]{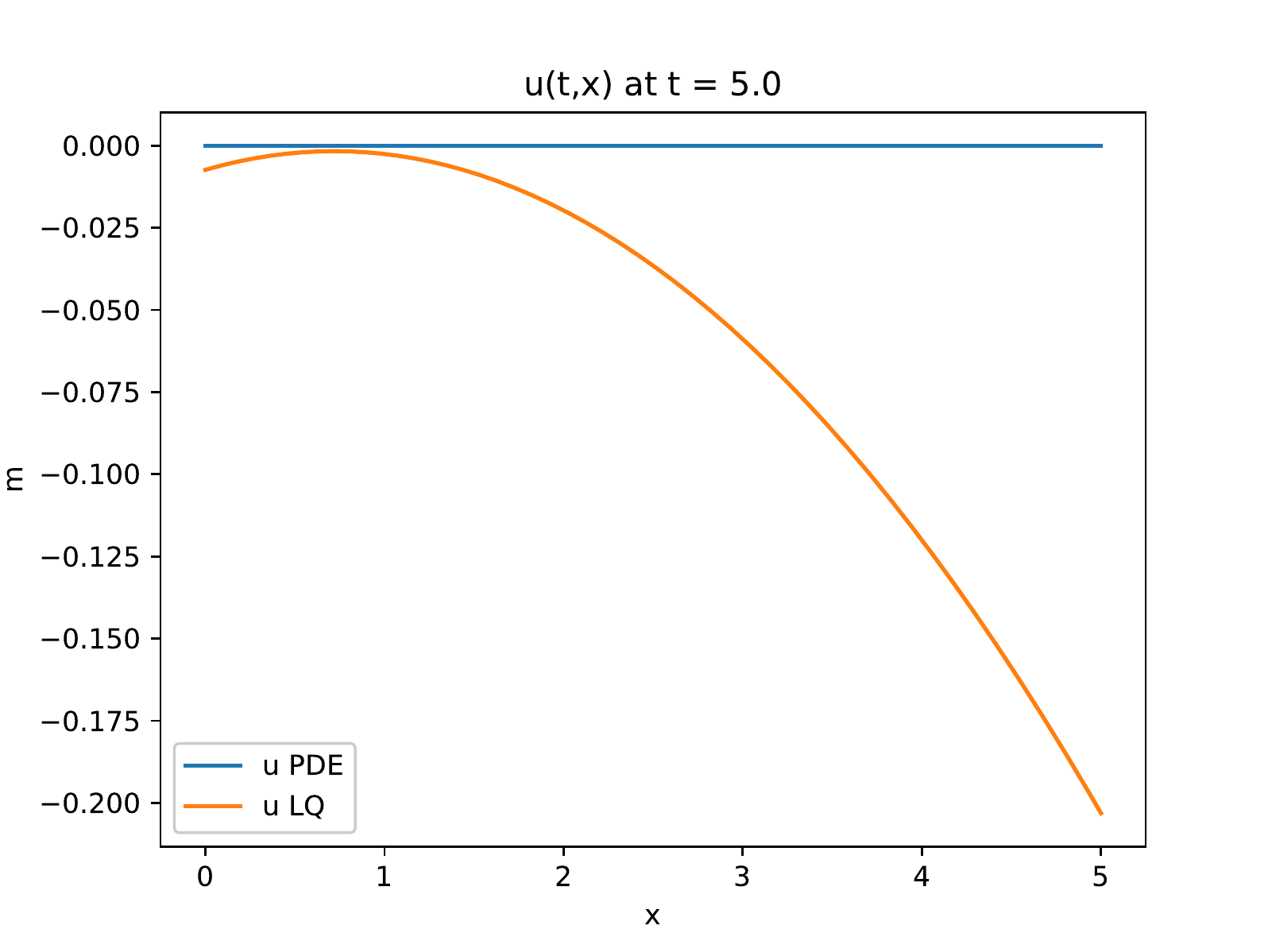} \\
{$t = 5$} & {$t = 5$}\\
\end{tabular} 
\caption{Comparison of time dependent MFG solution with quadratic ansatz. Left: Dirichlet conditions at $x = 0$ and $x=L$ in accordance with the quadratic ansatz. Right: Dirichlet condition $0$ at $x=0$ and Neumann condition $0$ at $x=L$. Here the time horizon is $T=5.0$. }
\label{fig:cmp-mfg-lqansatz}
\end{figure}

\subsection{Numerical approximation scheme for the general Mean Field Game} \label{sec: 3.4} 

We will use finite differences to discretize the HJB-FP PDE system~\eqref{eq:sysH-HJB}--\eqref{eq:sysH-FP}, adapting to our setting the scheme introduced in~\cite{MR2679575}. The main difference is that, in the FP equation, we need to deal with the defaults and with the creation of new banks at a value which is not known in advance. For numerical purposes, it will be interesting to consider the PDE system on a truncated domain. 
Let $L>0$ be a given constant, and let us denote by $D=(0,L)$ the spatial domain and by $Q_T = (0,T) \times D$ the time-space domain. We will consider the PDEs~\eqref{eq:sysH-HJB}--\eqref{eq:sysH-FP} on $Q_T$, and impose the following boundary conditions~: 
for all $t \in (0,T)$,
 \begin{equation}
	 u(t,0) = u(t,L) = 0 \, , \qquad \partial_x u(t,L) = 0 \, ,
\end{equation}
and	 
\begin{equation}
	m(t,0) = u(t,L) = 0 \, , \qquad \partial_x m(t,L) = 0 \, ,
 \end{equation}
and the initial and final conditions~: for all $x \in D$,~\eqref{eq:hjbfp-init-fin} holds.

\paragraph{\textbf{Discretization.} }
Let $N_T$ and $N_h$ be two positive integers corresponding respectively to the number of steps in time and space. We consider $(N_T+1)$ and $(N_h+1)$ points in  time and space respectively. Let $\Delta t = T/N_T$ and $h = L/N_h$, and $t_n = n\Delta t, x_i = i \,h$ for $(n,i) \in \{0,\dots,N_T\}\times\{0,\dots,N_h\}$. We also introduce an extended index function (which implicitly depend on the spatial grid) defined by: for $x \in D$, 
$$
	\ind(x) = \max\{i \, : \, x_i \leq x\}  \,  .
$$
Note that for all $x \in D$, we have $x \in [x_{\ind(\xi)}, x_{\ind(\xi)+1})$. 

We approximate $u^{(k)}$ and $m^{(k)}$ respectively by vectors $U^{(k)}$ and $M^{(k)} \in \RR^{(N_T+1)\times(N_h+1)}$, such that $u^{(k)}(t_n,x_i) \approx U^{(k),n}_{i}$ and $m^{(k)}(t_n,x_i) \approx M^{(k),n}_{i}$ for each $(n,i)$. 
For $M \in \RR^{N_h+1}$, we let 
$$
	\mbf(M) = h \sum_i x_i M_i \, ,\qquad \ebf(M) = \frac{\sigma^2}{2 h}\left( M_1 - M_0 \right).
$$
We introduce the finite difference operators
\begin{align*}
	(D_t W)^n &= \frac{1}{\Delta t}(W^{n+1} - W^n), \qquad n \in \{0, \dots N_T-1\}, \qquad W \in \RR^{N_T+1},
	\\
	(D^+ W)_i &= \frac{1}{h} (W_{i+1} - W_{i}), \qquad i  \in \{0, \dots N_h-1\}, \qquad W \in \RR^{N_h+1},
	\\
	(\Delta_h W)_i &= -\frac{1}{h^2} \left(2 W_i - W_{i+1} - W_{i-1}\right), \qquad i \in \{ 1, \dots N_h-1\}, \qquad W \in \RR^{N_h+1},
	\\
	[\grad_h W]_i &= \left( (D^+ W)_{i}, (D^+ W)_{i-1} \right)^T, \qquad i \in \{ 0, \dots N_h-1\}, \qquad W \in \RR^{N_h+1}.
\end{align*}

\paragraph{\textbf{Discrete Hamiltonian.} }
We first introduce $\tilde\varphi$ and $\tilde\psi$, defined for $(x,M) \in \RR \times \RR^{N_h+1}$ by:
\begin{align*}
	\tilde \varphi(x,M) & = (q+a)(\mbf(M)-x) - \ebf(M) \mbf(M) \, , 
	\\
	\tilde \psi(x,M) & = \frac{1}{2} \left[  \tilde\varphi(x,M)^2 - (q^2-\epsilon) (\mbf(M)-x)^2\right].
\end{align*}
We then introduce the following discrete Hamiltonian $\tilde H$, whose definition is based on~\eqref{eq:H-quadratic},
\begin{equation}
	\label{eq:defHdiscrete}
	\tilde H(x,M,p_1,p_2) = 
	\frac{1}{2} \left\{ \left[ (p_1 - \tilde\varphi(x, M))^- \right]^2 + \left[ (p_2 - \tilde\varphi(x, M))^+ \right]^2 \right\}
	- \tilde\psi(x,M) \, ,
\end{equation}
where $x \in D$, $p_1, p_2 \in \RR$, $M \in \RR^{N_h+1}$. 
In particular $\tilde H$ has the following properties:
\begin{enumerate}
	\item Monotonicity: $\tilde H$ is nonincreasing in $p_1$ and nondecreasing in $p_2$.
	\item Consistency: $\tilde H(x,M,p,p) = H(x,M,p)$ (where, in the right-hand side, $M$ is identified with the corresponding piecewise linear function defined by its values on the spatial grid) .
	\item Differentiability: $\tilde H$ is of class $\mathcal C^1$ with respect to $(x, p_1, p_2)$.
	\item Convexity:  $(p_1,p_2) \mapsto \tilde H(x,M,p_1,p_2)$ is convex. 
\end{enumerate}

\paragraph{\textbf{Discrete HJB equation.} }
As in~\cite{MR2679575}, we consider the following discrete version of~\eqref{eq:sysH-HJB}
\begin{subequations}
\label{eq:sysH-HJB-discrete}
\begin{empheq}[left=\empheqlbrace]{align}
	&r U_{i}^n - (D_t U_{i})^{n} - \tfrac{\sigma^2}{2} (\Delta_h U^{n})_{i}
	+ \tilde H(x_i, M^{n+1}, [\grad_h U^n]_i) 
	= 0 \, , 
	\label{eq:sysH-HJB-scheme-discrete-edo}
	\\ &\qquad\qquad\qquad\qquad\qquad\qquad\qquad\qquad\qquad   i \in \{1,\dots,N_h-2\} \, , \,  n \in \{0, \dots, N_T-1\} \, ,
	\notag
	\\
	&U^{n}_{i} = 0 \, , \qquad n \in \{0, \dots,  N_T-1\} \, , \,  i \in \{ 0, N_{h}-1, N_{h} \} \, ,
	\label{eq:sysH-HJB-scheme-discrete-bc}
	\\
	&U^{N_T}_{i} = 0 \, , \qquad  i \in \{ 0, \dots , N_{h} \} \, .
	\label{eq:sysH-HJB-scheme-discrete-finalc}
\end{empheq}
\end{subequations}

\paragraph{\textbf{Discrete FP equation.} } 
To define a discretization of the FP equation, we consider the weak form of~\eqref{eq:sysH-FP}. It involves, among other terms, for a smooth $w \in \mathcal{C}^{\infty}(D \times [0,T])$,
\begin{align}
	&- \int_{D} \left[ \partial_x \big( H_p(x, m(t, \cdot), \partial_x u(t,x)) m(t,x)\big) + \ebf(m(t,\cdot)) \delta_{\mbf(m(t,\cdot))}(x) \right] w(t,x) dx
	\notag
	\\
	=\, &
	\int_{D} H_p(x, m(t, \cdot), \partial_x u(t,x)) m(t,x) \, \partial_x w(t,x) dx - \ebf(m(t,\cdot)) w\big(t,\mbf(m(t,\cdot))\big) \, ,
	\label{eq:HJB-weak}
\end{align}
where we used integration by parts and the boundary conditions.

This leads us to introduce the following two discrete operators. For the first part (see~\cite{MR2679575} for more details), we introduce
\begin{align*}
	\mathcal B_i(U, M) 
	= \,&\frac{1}{h} 
	\Big( M_i  \tilde H_{p_1}(x_i, M, [\grad_h U]_i) 
		- M_{i-1} \tilde H_{p_1}(x_{i-1}, M, [\grad_h U]_{i-1})
		\\
		&\quad +
		M_{i+1} \tilde H_{p_2}(x_{i+1}, M, [\grad_h U]_{i+1})
		- M_i \tilde H_{p_2}(x_i, M, [\grad_h U]_i)
	\Big) \, .
\end{align*}
Notice that, with our definition of $\tilde H$ (see~\eqref{eq:defHdiscrete}),
\begin{align*}
	&\tilde H_{p_1}(x_i, M, [\grad_h U]_i) = - \left( \frac{U_{i+1}-U_{i}}{h} - \tilde \varphi(x,M) \right)^-, 
	\\
	&\tilde H_{p_2}(x_{i+1}, M, [\grad_h U]_{i+1}) = \left( \frac{U_{i+1}-U_{i}}{h} - \tilde \varphi(x,M) \right)^+,
\end{align*}
and similarly for the other terms. 

For the second part of~\eqref{eq:HJB-weak}, we introduce, for $M \in \RR^{N_h+1}$ and $\mu \in D$,
$$
	\beta_i(M, \mu) = 
	\begin{cases}
	\ebf(M) (x_{i+1} - \mu)/h^2 \, , & \hbox{ if } i = \ind(\mu) \, ,
	\\
	\ebf(M) (\mu - x_{i-1})/h^2 \, , & \hbox{ if } i = \ind(\mu)+1 \, ,
	\\
	0 \, , & \hbox{ otherwise.}
	\end{cases}
$$
Considering a piecewise linear function $W$ defined on $D$ by its values $W_i = W(x_i)$, $i=0,\dots,N_h$, at the mesh points, we have
$$
	\ebf(M) W(\mu) = \sum_{i=0}^{N_h} \beta_i(M, \mu) W_i  \, .
$$

Then, for the discrete version of~\eqref{eq:sysH-FP} we consider
\begin{subequations}
\label{eq:sysH-FP-discrete}
\begin{empheq}[left=\empheqlbrace]{align}
	&(D_t M_{i})^{n} - \tfrac{\sigma^2}{2} (\Delta_h M^{n+1})_{i}
	- \mathcal B_i(U^{n}, M^{n+1})
	- \beta_i(M^{n+1}, \mbf(M^{n+1}))
	= 0 \, , 
	\label{eq:sysH-FP-scheme-discrete-edo}
	\\
	&\qquad\qquad\qquad\qquad\qquad\qquad\qquad\qquad i \in \{1,\dots,N_h-2\}, n \in \{0, \dots, N_T-1\} \, ,
	\notag
	\\
	&M^{n}_{i} = 0 \, , \qquad n \in \{ 1, \dots, N_T\} \, , \,  i \in \{0, N_{h}-1, N_{h} \} \, ,
	\label{eq:sysH-FP-scheme-discrete-bc}
	\\
	&M^{0}_{i} = m_0(x_i) \, , \,  \qquad i \in \{0, \dots, N_h\} \, .
	\label{eq:sysH-FP-scheme-discrete-initialc}
\end{empheq}
\end{subequations}

\begin{remark}
	A direct discretization of~\eqref{eq:sysH-FP} would have to deal with a Dirac mass at the point $\mbf(m(t,\cdot))$, which is not necessarily a point of the mesh. One advantage of considering the weak formulation as proposed above is to avoid this issue.
\end{remark}

\paragraph{\textbf{Numerical method.} } We now describe how to compute a solution to the above discrete system~\eqref{eq:sysH-HJB-scheme-discrete-edo}--\eqref{eq:sysH-HJB-scheme-discrete-finalc} and~\eqref{eq:sysH-FP-scheme-discrete-edo}--\eqref{eq:sysH-FP-scheme-discrete-initialc}.
Note that the discrete FP equation~\eqref{eq:sysH-FP-scheme-discrete-edo} is non-linear in $M$ (due to the term $\beta_i(M^{n+1}, \mbf(M^{n+1}))$). Trying to solve this equation using e.g. Newton's method would involve highly non-sparse matrices (due to the term $\mbf(M^{n+1})$).

To avoid this issue, we propose to employ the iterative procedure described in Algorithm~\ref{algo:iter-UM}.

\begin{algorithm}
\DontPrintSemicolon
\KwData{An initial guess $(\tilde M, \tilde U)$; a number of iterations $K$.}
\KwResult{An approximation of $(u,m)$}
\Begin{
  Initialize $(M^{(0)}, U^{(0)}) \leftarrow (\tilde M, \tilde U)$. \; 
  \For{$k = 0, 1, 2, \dots, K-1$}{
	Compute $M^{(k+1)}$ solving
		\begin{align}
                    	&(D_t M_{i})^{n} - \tfrac{\sigma^2}{2} (\Delta_h M^{n+1})_{i}
                    	- \mathcal B_i(U^{(k),n}, M^{n+1})
                    	- \beta_i(M^{(k),n+1}, \mbf(M^{(k),n+1}))
                    	= 0 \, , 
                    	\label{eq:sysH-FP-scheme-discrete-edo-fixedpoint}
                    	\\
                    	&\qquad\qquad\qquad\qquad\qquad\qquad\qquad\qquad i \in \{1,\dots,N_h-2\}, n \in \{0, \dots, N_T-1\} \, ,
                    	\notag
                    	\\
                    	&M^{n}_{i} = 0 \, , \qquad n \in \{ 1, \dots, N_T\} \, , \,  i \in \{0, N_{h}-1, N_{h} \} \, ,
                    	\label{eq:sysH-FP-scheme-discrete-bc-fixedpoint}
                    	\\
                    	&M^{0}_{i} = m_0(x_i) \, , \,  \qquad i \in \{0, \dots, N_h\} \, .
                    	\label{eq:sysH-FP-scheme-discrete-initialc-fixedpoint}
		\end{align}
		\;
		Compute $U^{(k+1)}$ solving
		\begin{align}
                    	&r U_{i}^n - (D_t U_{i})^{n} - \tfrac{\sigma^2}{2} (\Delta_h U^{n})_{i}
                    	+ \tilde H(x_i, M^{(k+1),n+1}, [\grad_h U^n]_i) 
                    	= 0 \, , 
                    	\label{eq:sysH-HJB-scheme-discrete-edo-fixedpoint}
                    	\\ &\qquad\qquad\qquad\qquad\qquad\qquad\qquad\qquad\qquad   i \in \{1,\dots,N_h-2\} \, , \,  n \in \{0, \dots, N_T-1\} \, ,
                    	\notag
                    	\\
                    	&U^{n}_{i} = 0 \, , \qquad n \in \{0, \dots,  N_T-1\} \, , \,  i \in \{ 0, N_{h}-1, N_{h} \} \, ,
                    	\label{eq:sysH-HJB-scheme-discrete-bc-fixedpoint}
                    	\\
                    	&U^{N_T}_{i} = 0 \, , \qquad  i \in \{ 0, \dots , N_{h} \} \, .
                    	\label{eq:sysH-HJB-scheme-discrete-finalc-fixedpoint}
		\end{align}
		\;
    }
  \KwRet{ $(M^{(K)}, U^{(K)})$}
  }
\caption{\footnotesize Iterative method for the finite difference system~\eqref{eq:sysH-HJB-discrete} \& \eqref{eq:sysH-FP-discrete}}
\label{algo:iter-UM}
\end{algorithm}

Notice that~\eqref{eq:sysH-FP-scheme-discrete-edo-fixedpoint}--\eqref{eq:sysH-FP-scheme-discrete-initialc-fixedpoint} is a modified version of the discrete FP equation~\eqref{eq:sysH-FP-discrete}, in which part of the unknown $M$ is replaced by the estimate $M^{(k)}$ from the previous iteration.  At convergence, we have $M^{(k+1)} = M^{(k)}$ and hence~\eqref{eq:sysH-FP-discrete} is satisfied. 
In our implementation, instead of fixing the number of iterations, we use as convergence criterion the normalized $\ell^2$-norms of the difference between two iterates of $U$ and two iterates of $M$. The discrete HJB equation~\eqref{eq:sysH-HJB-scheme-discrete-edo-fixedpoint}--\eqref{eq:sysH-HJB-scheme-discrete-finalc-fixedpoint} is solved by \textsc{Newton}'s method.

\subsection{Numerical results} \label{sec: 3.5}

We now present numerical results obtained using the numerical method described above. For the results displayed in Figure~\ref{fig:mfg-dyn-statics}, we fixed $L=10$ (so that the space domain is $[0, 10]$), $T=10$, $q=0.1$, and $\epsilon = q^2$. We consider as a baseline the setting with $a= 0.5, x_0 = 2.0, r=0.5, \sigma=1.0$. We are particularly interested in the evolution of the default rate $\ebf(m_t)$ through time. We note the following behavior: The default rate increases as the interaction strength $a$ or the initial mean $x_0$ increases; it decreases when the discount rate $r$ increases or when the volatility $\sigma$ decreases. All these variations seem quite natural from the point of view of the interpretation of the model. We note in particular that the solution starting with $x_0 = 3$ is very stable since the mean almost does not change and the default rate is very close to zero. In the other settings, the default rate tends to increase rapidly at the beginning of the time interval before taking a more stable value. This can be explained by the fact that the initial distribution is a truncated Gaussian with value $0$ at $x=0$, so that the default rate is initially null but the randomness (diffusion term) quickly causes the weakest banks to default. As new banks are re-injected in the system around the mean wealth, the overall default rate finds an (almost) stationary value. However, numerical experiments conducted with $x_0$ much smaller than $1.8$ did not converge to a solution. As hinted by Theorem~\ref{thm: Blowup}, this could be related to the fact that the solution blows up before time $T=10$ when the initial distribution is concentrated near $x=0$. 

Figure~\ref{fig:mfg-dyn-evol-m-u} displays the evolution of $m$ and $u$ in the baseline case mentioned above (namely, $a= 0.5, x_0 = 2.0, r=0.5, \sigma=1.0$). Although $m$ remains concentrated close to its original mean, it seems important from a numerical perspective to consider a large enough spatial domain. Indeed, since we artificially impose a Dirichlet boundary at $x = L$, choosing $L$ too small would imply that the total mass can not be conserved.

\begin{figure}[H]
\centering
\begin{subfigure}{.45\textwidth}
  \centering
  \includegraphics[width=\linewidth]{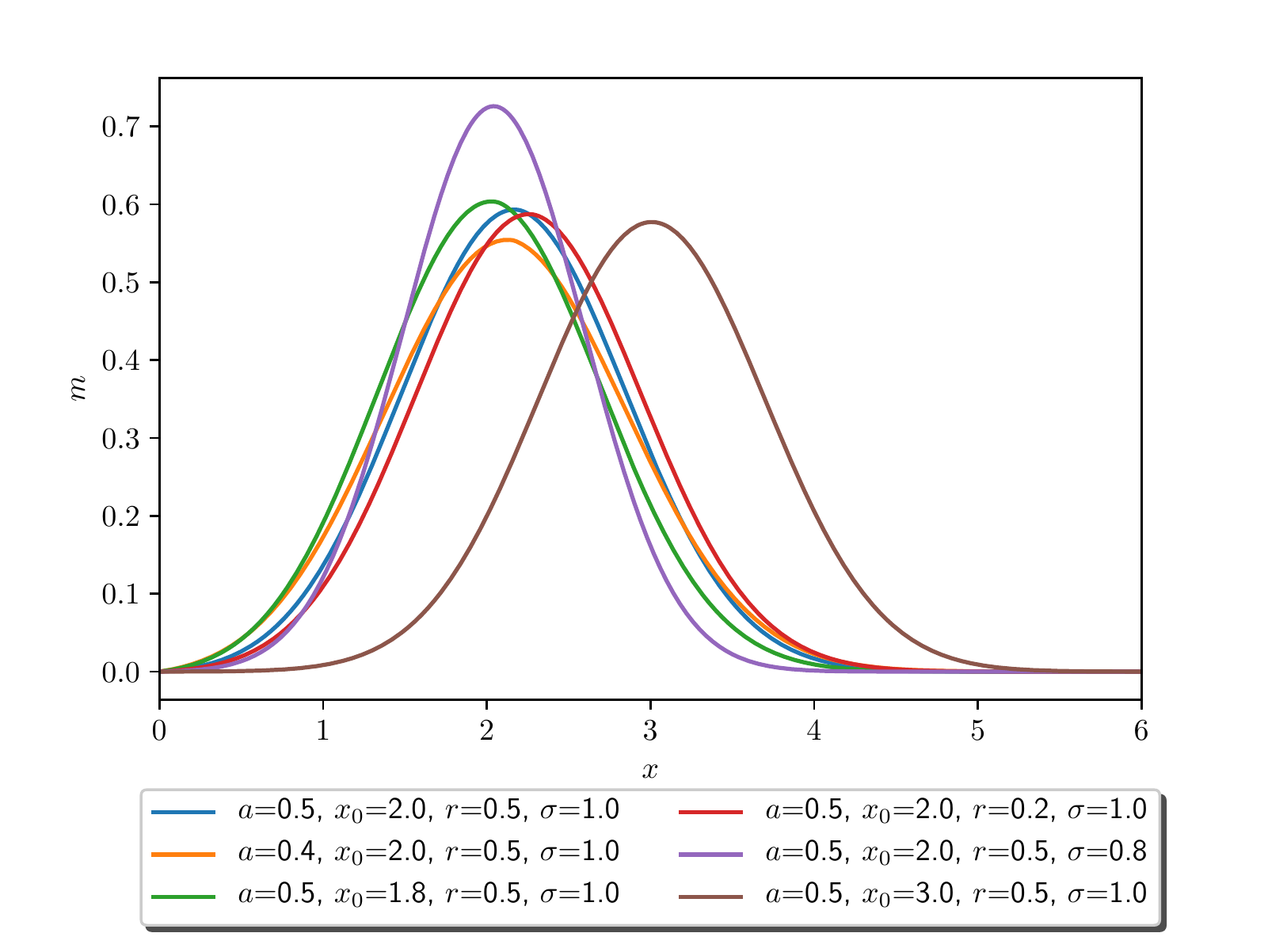}
  \caption{Final density}
  \label{fig:res-mfg-final-density}
\end{subfigure}%
\begin{subfigure}{.45\textwidth}
  \centering
  \includegraphics[width=\linewidth]{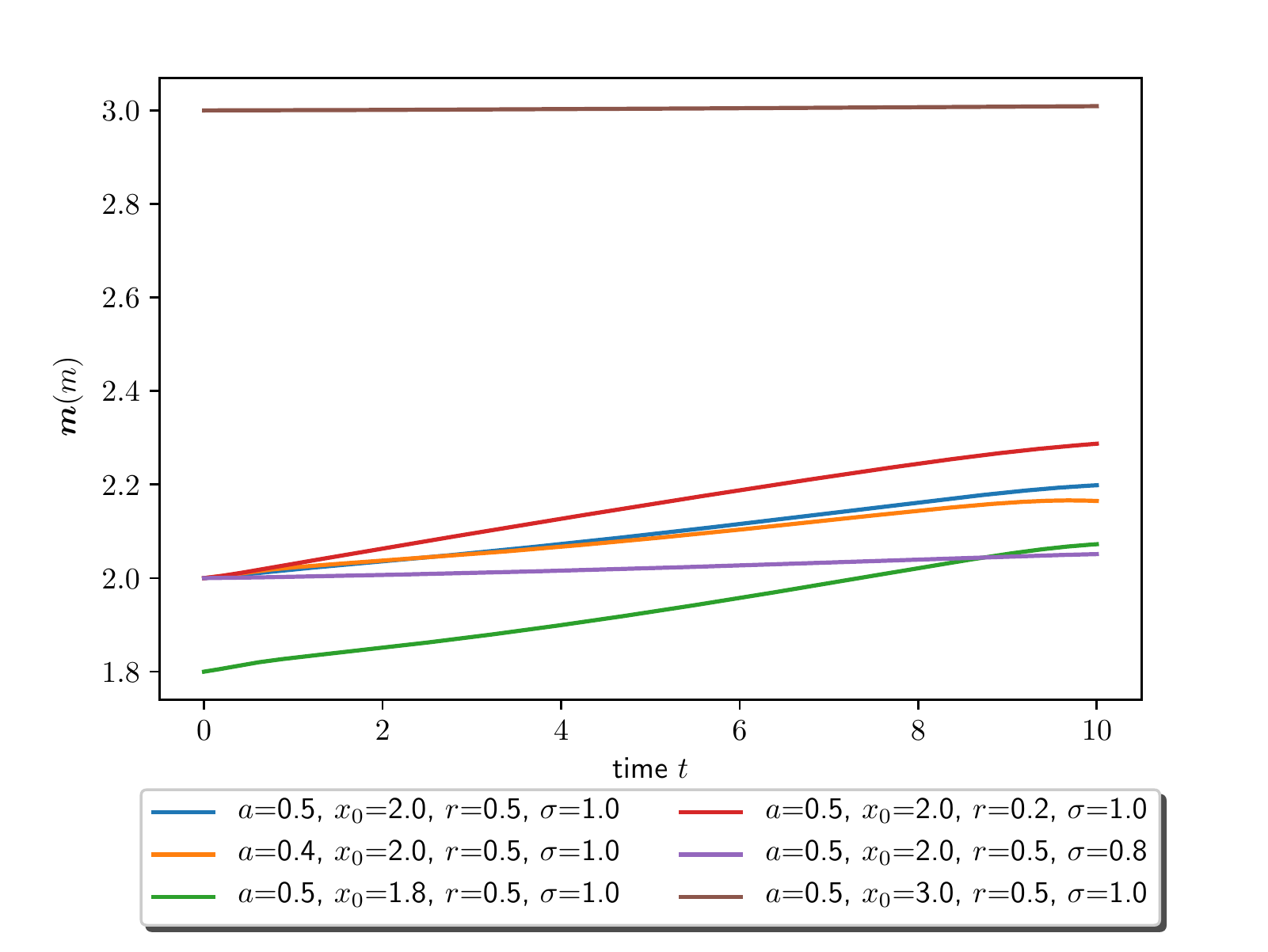}
  \caption{Mean value}
\end{subfigure}

\begin{subfigure}{.45\textwidth}
  \centering
  \includegraphics[width=\linewidth]{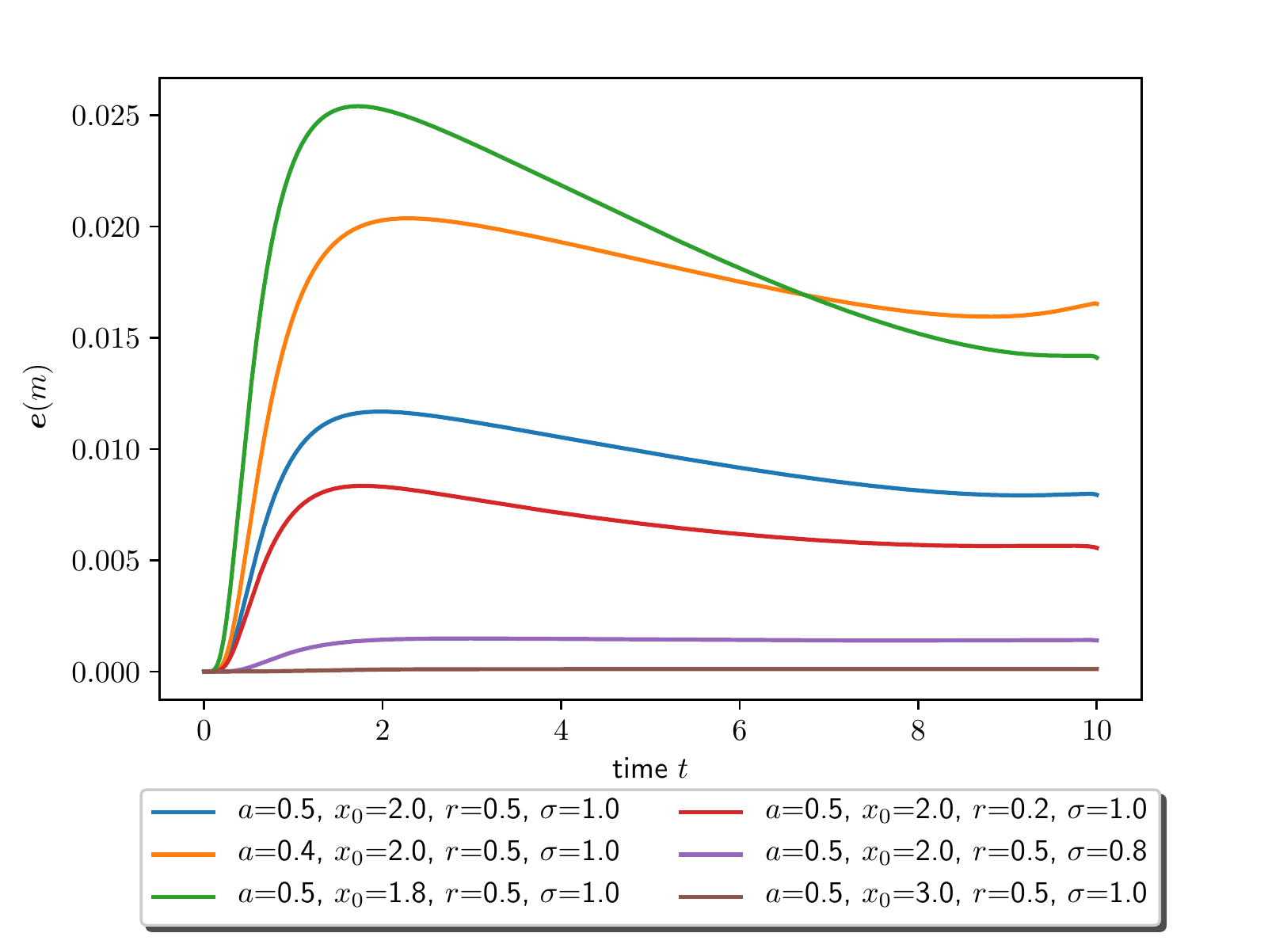}
  \caption{Default rate}
\end{subfigure}
\caption{Comparison for various settings. Here, we used $L=10$, $T=10$, $q=0.1$, $\epsilon = q^2$. The values of the parameters $a, x_0, r$ and $\sigma$ are specified in the legends. }
\label{fig:mfg-dyn-statics}
\end{figure}

\begin{figure}[H]
\centering
\begin{subfigure}{.45\textwidth}
  \centering
  \includegraphics[width=\linewidth]
  {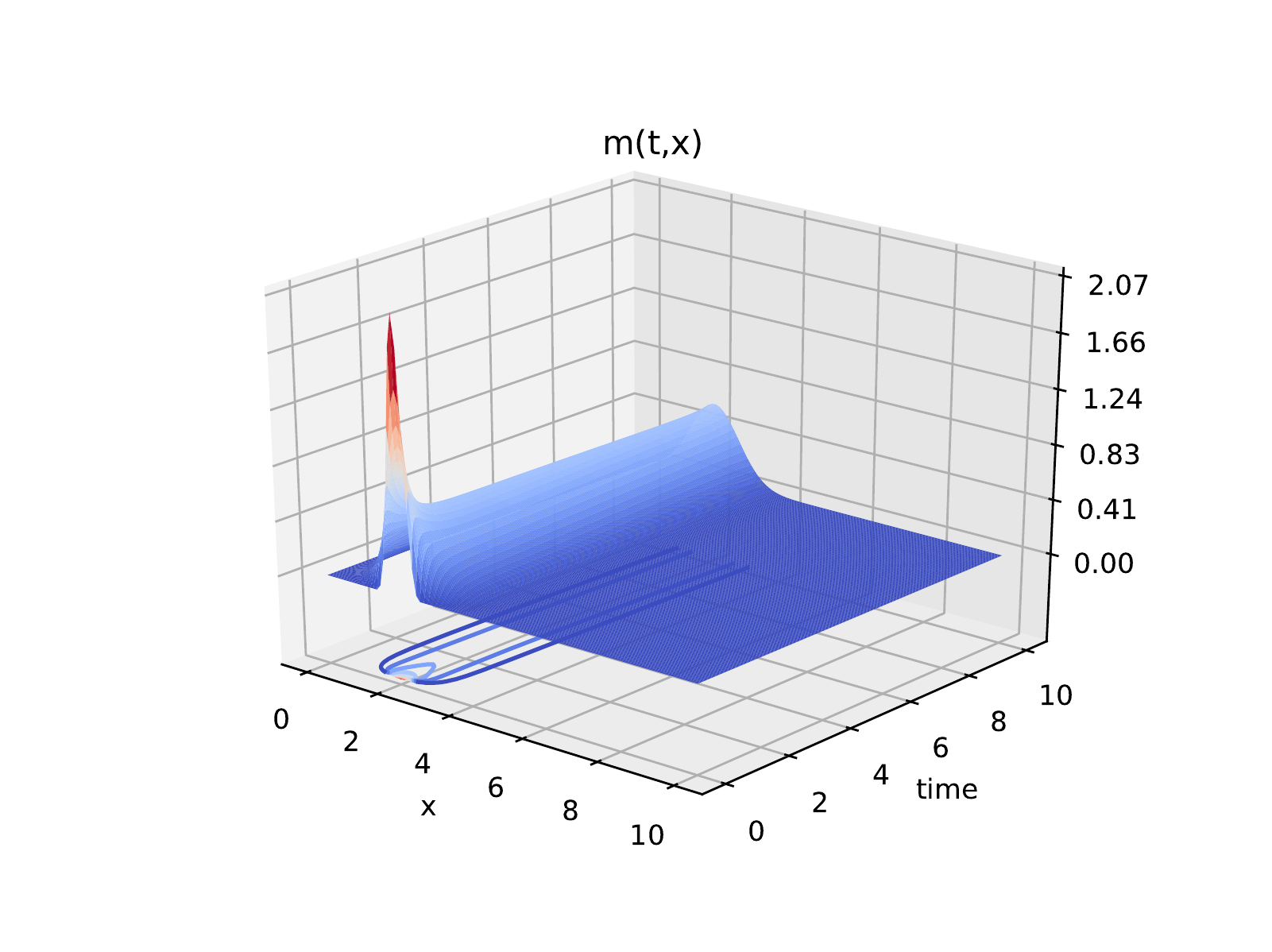}
  \caption{Evolution of $m$}
\end{subfigure}%
\begin{subfigure}{.45\textwidth}
  \centering
  \includegraphics[width=\linewidth]
  {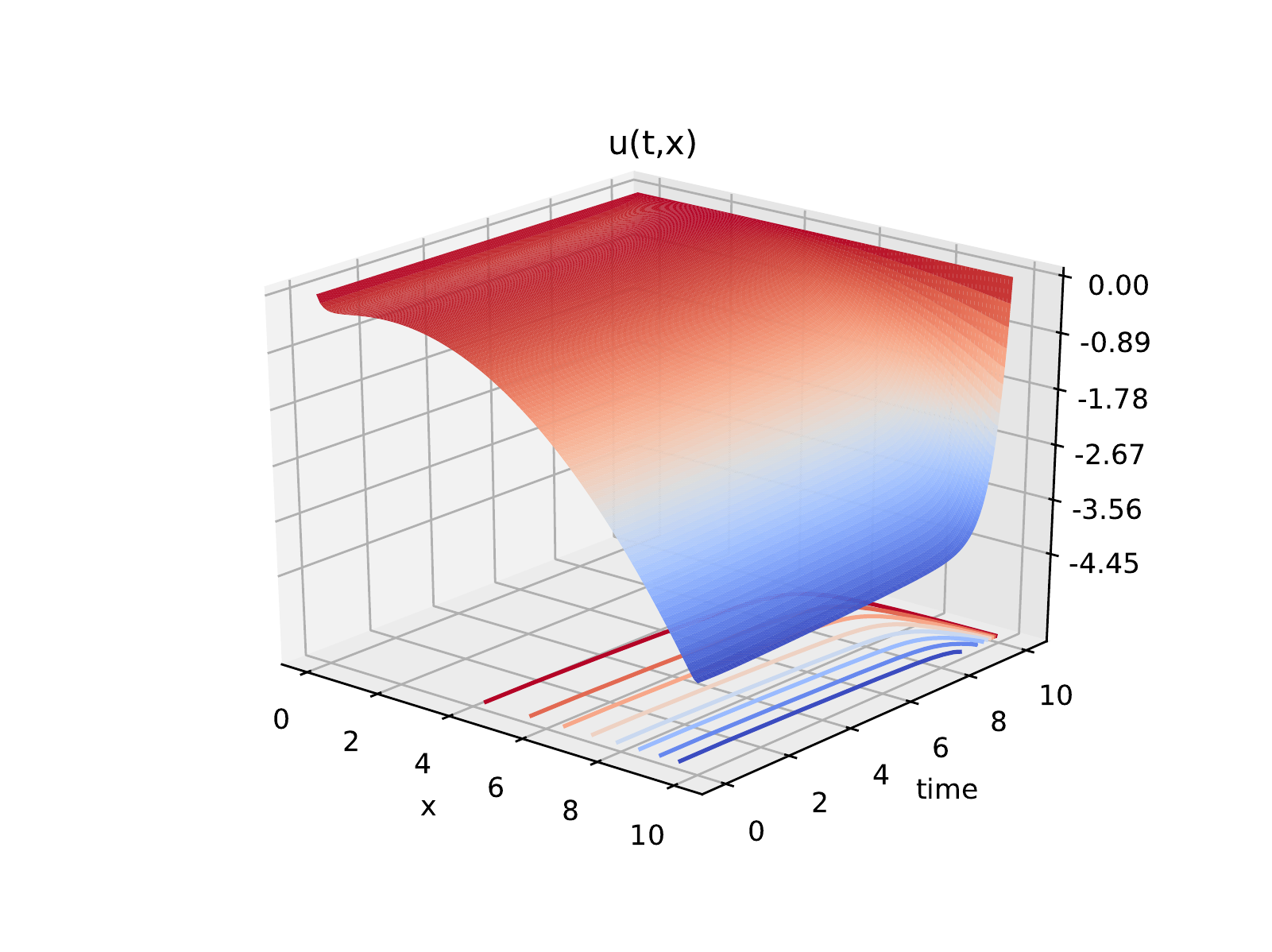}
  \caption{Evolution of $u$}
\end{subfigure}

\caption{Evolution of $m$ and $u$ as functions of $(t,x)$ in the baseline setting with $a= 0.5, x_0 = 2.0, r=0.5, \sigma=1.0$, and $L=10$, $T=10$, $q=0.1$, $\epsilon = q^2$. }
\label{fig:mfg-dyn-evol-m-u}
\end{figure}

{\small
\bibliographystyle{abbrv}
\bibliography{paper-neuronal-sysrisk-bib}

\begin{thebibliography}{10}

\bibitem{MR2679575}
Y.~Achdou and I.~Capuzzo-Dolcetta.
\newblock Mean field games: numerical methods.
\newblock {\em SIAM J. Numer. Anal.}, 48(3):1136--1162, 2010.

\bibitem{MR3134900}
A.~Bensoussan, J.~Frehse, and P.~Yam.
\newblock {\em Mean field games and mean field type control theory}.
\newblock Springer Briefs in Mathematics. Springer, New York, 2013.

\bibitem{MR2853216}
M.~J. C\'{a}ceres, J.~A. Carrillo, and B.~Perthame.
\newblock Analysis of nonlinear noisy integrate \& fire neuron models: blow-up
  and steady states.
\newblock {\em J. Math. Neurosci.}, 1:Art. 7, 33, 2011.

\bibitem{MR3752669}
R.~Carmona and F.~Delarue.
\newblock {\em Probabilistic theory of mean field games with applications.
  {I}}, volume~83 of {\em Probability Theory and Stochastic Modelling}.
\newblock Springer, Cham, 2018.
\newblock Mean field FBSDEs, control, and games.

\bibitem{MR3753660}
R.~Carmona and F.~Delarue.
\newblock {\em Probabilistic theory of mean field games with applications.
  {II}}, volume~84 of {\em Probability Theory and Stochastic Modelling}.
\newblock Springer, Cham, 2018.
\newblock Mean field games with common noise and master equations.

\bibitem{MR3325083}
R.~Carmona, J.-P. Fouque, and L.-H. Sun.
\newblock Mean field games and systemic risk.
\newblock {\em Commun. Math. Sci.}, 13(4):911--933, 2015.

\bibitem{talkDelarue2015}
F.~Delarue.
\newblock Mean-field analysis of an excitatory neuronal network: application to
  systemic risk modeling?
\newblock
  "\url{https://library.cirm-math.fr/Record.htm?idlist=1&record=19276511124910947939}",
  2015.

\bibitem{MR3349003}
F.~Delarue, J.~Inglis, S.~Rubenthaler, and E.~Tanr\'{e}.
\newblock Global solvability of a networked integrate-and-fire model of
  {M}c{K}ean-{V}lasov type.
\newblock {\em Ann. Appl. Probab.}, 25(4):2096--2133, 2015.

\bibitem{MR3322871}
F.~Delarue, J.~Inglis, S.~Rubenthaler, and E.~Tanr\'{e}.
\newblock Particle systems with a singular mean-field self-excitation.
  {A}pplication to neuronal networks.
\newblock {\em Stochastic Process. Appl.}, 125(6):2451--2492, 2015.

\bibitem{delarue2019stefan}
F.~Delarue, S.~Nadtochiy, and M.~Shkolnikov.
\newblock Global solutions to the supercooled stefan problem with blow-ups:
  regularity and uniqueness.
\newblock {\em arXiv preprint arXiv:1902.05174}, 2019.

\bibitem{MR0228020}
W.~Feller.
\newblock {\em An introduction to probability theory and its applications.
  {V}ol. {I}}.
\newblock Third edition. John Wiley \& Sons, Inc., New York-London-Sydney,
  1968.

\bibitem{hambly2018mckean}
B.~Hambly, S.~Ledger, and A.~Sojmark.
\newblock A mckean--vlasov equation with positive feedback and blow-ups.
\newblock {\em arXiv preprint arXiv:1801.07703}, 2018.

\bibitem{hambly2018spde}
B.~Hambly and A.~Sojmark.
\newblock An spde model for systemic risk with endogenous contagion.
\newblock {\em arXiv preprint arXiv:1801.10088}, 2018.

\bibitem{HuangCainesMalhame-2003-individual-mass-wireless}
M.~Huang, P.~E. Caines, and R.~P. Malham{\'e}.
\newblock Individual and mass behaviour in large population stochastic wireless
  power control problems: centralized and nash equilibrium solutions.
\newblock In {\em Decision and Control, 2003. Proceedings. 42nd IEEE Conference
  on}, volume~1, pages 98--103. IEEE, 2003.

\bibitem{MR2352434-HuangCainesMalhame-2007-LQG}
M.~Huang, P.~E. Caines, and R.~P. Malham{\'e}.
\newblock Large-population cost-coupled {LQG} problems with nonuniform agents:
  individual-mass behavior and decentralized {$\epsilon$}-{N}ash equilibria.
\newblock {\em IEEE Trans. Automat. Control}, 52(9):1560--1571, 2007.

\bibitem{MR2346927-HuangCainesMalhame-2006-closedLoop}
M.~Huang, R.~P. Malham{\'e}, and P.~E. Caines.
\newblock Large population stochastic dynamic games: closed-loop
  {M}c{K}ean-{V}lasov systems and the {N}ash certainty equivalence principle.
\newblock {\em Commun. Inf. Syst.}, 6(3):221--251, 2006.

\bibitem{MR3411723}
J.~Inglis and D.~Talay.
\newblock Mean-field limit of a stochastic particle system smoothly interacting
  through threshold hitting-times and applications to neural networks with
  dendritic component.
\newblock {\em SIAM J. Math. Anal.}, 47(5):3884--3916, 2015.

\bibitem{MR1121940}
I.~Karatzas and S.~E. Shreve.
\newblock {\em Brownian motion and stochastic calculus}, volume 113 of {\em
  Graduate Texts in Mathematics}.
\newblock Springer-Verlag, New York, second edition, 1991.

\bibitem{kaushansky2018semi}
V.~Kaushansky, A.~Lipton, and C.~Reisinger.
\newblock Semi-analytical solution of a mckean-vlasov equation with feedback
  through hitting a boundary.
\newblock {\em arXiv preprint arXiv:1808.05311}, 2018.

\bibitem{MR2269875}
J.-M. Lasry and P.-L. Lions.
\newblock Jeux \`a champ moyen. {I}. {L}e cas stationnaire.
\newblock {\em C. R. Math. Acad. Sci. Paris}, 343(9):619--625, 2006.

\bibitem{MR2271747}
J.-M. Lasry and P.-L. Lions.
\newblock Jeux \`a champ moyen. {II}. {H}orizon fini et contr\^ole optimal.
\newblock {\em C. R. Math. Acad. Sci. Paris}, 343(10):679--684, 2006.

\bibitem{MR2295621}
J.-M. Lasry and P.-L. Lions.
\newblock Mean field games.
\newblock {\em Jpn. J. Math.}, 2(1):229--260, 2007.

\bibitem{ledger2018mercy}
S.~Ledger and A.~Sojmark.
\newblock At the mercy of the common noise: Blow-ups in a conditional
  mckean--vlasov problem.
\newblock {\em arXiv preprint arXiv:1807.05126}, 2018.

\bibitem{PLLCDF}
P.-L. Lions.
\newblock Cours du {C}oll{\`e}ge de {F}rance.
\newblock "\url{http://www.college-de-france.fr/default/EN/all/equ$_-$der/}",
  2007-2011.

\bibitem{nadtochiy2018mean}
S.~Nadtochiy and M.~Shkolnikov.
\newblock Mean field systems on networks, with singular interaction through
  hitting times.
\newblock {\em arXiv preprint arXiv:1807.02015}, 2018.

\bibitem{MR3910001}
S.~Nadtochiy and M.~Shkolnikov.
\newblock Particle systems with singular interaction through hitting times:
  application in systemic risk modeling.
\newblock {\em Ann. Appl. Probab.}, 29(1):89--129, 2019.

\bibitem{MR2978140}
R.~L. Schilling, R.~Song, and Z.~Vondra\v{c}ek.
\newblock {\em Bernstein functions}, volume~37 of {\em De Gruyter Studies in
  Mathematics}.
\newblock Walter de Gruyter \& Co., Berlin, second edition, 2012.
\newblock Theory and applications.

\end{thebibliography}
}

\appendix

\section{Derivation of the KFP equation}
\label{sec:deriv-KFP}

We shall derive the KFP for the transition
 probability $\, p(t, x) \,$ of stochastic process  $\, \mathcal X_{t}\, $, $\, t \ge 0 \,$ represented by 
\begin{equation} \label{eq: A1} 
\mathcal X_{t} \, =\,  \mathcal X_{0} + \int^{t}_{0} (-a) ( \mathcal X_{s} - \mathbb E [ \mathcal X_{s}] ) {\mathrm d} s + W_{t} + \int^{t}_{0} \mathbb E [ \mathcal X_{s}] {\mathrm d} \mathcal M_{s} - \alpha \int^{t}_{0} \mathbb E [ \mathcal X_{s} ] {\mathrm d}_{s} \mathbb E [ \mathcal M_{s}]  \, , 
\end{equation}
where $\, \alpha \ge 0 \,$. Taking the expectations of both sides, we obtain a linear integral equation 
\[
\overline{x}_{t} \, :=\,  \mathbb E [ \mathcal X_{t}] \, =\,  \mathbb E [ \mathcal X_{0}] + (1-\alpha ) \int^{t}_{0} \mathbb E [ X_{s}] {\mathrm d}_{s} \mathbb E [ \mathcal M_{s} ] \, =\,  \overline{x}_{0} + (1-\alpha) \int^{t}_{0} \overline{x}_{s}\dot{e}_{s} {\mathrm d} s  \, ; \quad t \ge 0 \, .  
\]
Solving this equation, we have $\, \overline{x}_{t} \, =\, \mathbb E [ \mathcal X_{t} ] \, =\,  \overline{x}_{0} \exp ( ( 1- \alpha ) e_{t} ) \,$, $\, t \ge 0 \,$. Then substituting it into (\ref{eq: A1}), we have 
\[
\mathcal X_{t} \, =\,  \mathcal X_{0} + \int^{t}_{0} [ (-a) ( \mathcal X_{s} - \overline{x}_{s}) - \alpha \overline{x}_{s} \, \dot{e}_{s} ] {\mathrm d} s + W_{t} + \int^{t}_{0 } \overline{x}_{s} {\mathrm d} \mathcal M_{s} \, ; \quad t \ge 0 \, . 
\]

For a given smooth function $\, \varphi : [0,  \infty) \to \mathbb R \,$ with a compact support so that $\, \varphi (0) \, =\,  0 \,$, we apply the change of variable formula for the c\`ad\`ag semimartingale to obtain 
\[
\varphi (   {\mathcal X}_{t+h} ) \, =\,  \varphi (   {\mathcal X}_{t} ) + \int^{h}_{0}\varphi^{\prime}(   {\mathcal X}_{t+u} ) [(-a)  ({\mathcal X}_{t+u}  - \overline{x}_{t+u})  - \alpha \overline{x}_{s} \dot{e}_{s}] {\mathrm d}  u + \int^{h}_{0} \varphi ^{\prime}(   {\mathcal X}_{t+u}) {\mathrm d}   {W}_{t+u} 
\]
\[
+ \sum_{0 \le u < h }[\varphi (  {\mathcal X}_{t+u}) - \varphi  (   {\mathcal X}_{t+u-})  ]  + \frac{\,1\,}{\,2\,} \int^{h}_{0} \varphi^{\prime\prime} (   {\mathcal X}_{t+u} ) {\mathrm d}_{u} \langle   {\mathcal X}\rangle_{t+u}
\]
for $\, t \ge  0 \,$, $\, h > 0\,$, where $\, {X}_{t+u-}\,$ is the left limit $\, \lim_{s\uparrow u} {X}_{t+s} \,$, $\, t \ge 0 \,$, $\, u \in [0, h)\,$. Note that if there exists a jump in  $\, {X}_{\cdot}\,$ at time $\, t+u\,$ for some $\, u \in [0, h)\,$, then there is a jump in $\,\varphi ( {X}_{\cdot})\,$ from  $\, \varphi ({X}_{t+u-}) \, =\,  \varphi (0) \, =\,  0 \,$ to $\, \varphi ( {X}_{t+u}) \, =\,  \varphi (\overline{x}_{t+u}) \,$. For $\, h > 0 \,$, multiplying by $\,1/h\,$, taking expectations, letting $\, h \downarrow 0\,$, we obtain 
\[
\partial_{t} \int^{\infty}_{0} \varphi (x) {p}(t, x) {\mathrm d} x \, =\,  \int^{\infty}_{0} \varphi^{\prime}(x) [ (-a) (x - \overline{x}_{t} ) - \alpha \overline{x}_{t} \dot{e}_{t} ]   {p}(t, x) {\mathrm d} x + \frac{\,1\,}{\,2 } \int^{\infty}_{0} \varphi^{\prime\prime} (x) {p}(t, x) {\mathrm d} x + \varphi (\overline{x}_{t}) \dot{e}_{t} 
\]
for $\, t \ge  0 \,$. Note that since the intensity of jumps in $\, {\mathcal X}_{\cdot}\,$ occurs at the rate $\, {\dot{\widehat{e}}}_{\cdot}\,$, the last term converges in probability, as $\,h \downarrow 0 \,$, 
\begin{equation} 
\frac{\,1\,}{\,h\,}\sum_{0 \le u < h }[\varphi ({\mathcal X}_{t+u}) - \varphi  ( {\mathcal X}_{t+u-})  ] \, =\,  \frac{\,1\,}{\,h\,} \sum_{0 \le u < h}\varphi ( \overline{x}_{t+u})  \xrightarrow[h \downarrow 0]{}  \varphi ( \overline{x}_{t})  \, {\dot {e}}_{t} \, =\,  \int^{\infty}_{0}  \varphi ( x) \dot{e}_{t} \delta_{ \overline{x}_{t}} ( {\mathrm d} x)  \, , 
\end{equation}
where $\, \delta_{a} ( {\mathrm d} x ) \,$ is the Dirac delta measure at $\,a\,$. Repeating the integration by parts for the other terms, we conclude the Fokker-Planck equation of the density function 
\begin{equation} \label{eq: FP0} 
\partial_{t} {p}(t, x) + \partial_{x} [ (-a(x- \overline{x}_{t}) - \alpha \overline{x}_{t}\dot{  e}_{t} ) {p}(t, x)] - \frac{\,1\,}{\,2 \,} \partial^{2}_{xx} {p}(t, x) \, =\,  \dot{ e}_{t} \, \delta_{ \overline{x}_{t}} ( {\mathrm d} x)  \, 
\end{equation}
for $\, t \ge 0\, $, $ \, x \in (0, \infty) \, $ with $\, \overline{x}_{t} \, =\,  \int^{\infty}_{0} x p(t, x) {\mathrm d} x \,$, $\, t \ge 0 \,$. 

For the boundary condition it is natural to assume that for every $\, t \ge 0 \,$ 
\begin{equation} \label{eq: boundary0} 
\lim_{x \downarrow 0} {p}(t, x) \, =\,  0 , \quad  \lim_{x\to + \infty}  \, {p}(t, x) \, =\, 0 \,  , \quad\lim_{x\to +\infty} \partial_{x} {p}(t, x) \, =\,  0 \;.  
\end{equation}
Then since $\, \, \mathbb P ( {X}_{t} \in [0, \infty) ) \, =\,  \int^{\infty}_{0} {p}(t, x) {\mathrm d} x \, =\,  1\,$, interchanging the order of differentiation and integration, substituting (\ref{eq: FP0}) and integrating once, we obtain 
\begin{equation*}
\begin{split} 
0 \, & =\,  \frac{\,\partial \,}{\,\partial t\,} \int^{\infty}_{0} {p}(t, x) {\mathrm d} x \, =\,  \int^{\infty}_{0} \partial_{t}  {p}(t, x) {\mathrm d} x \, = \frac{\,1\,}{\,2\,} \partial_{x} {p}(t, 0) + \dot{{e}}_{t} \, , 
\end{split}
\end{equation*}
where we used (\ref{eq: boundary0}) to compute the definite integrals. Thus 
\begin{equation} \label{eq: edott}
\dot{ {e}}_{t} \, =\,  \frac{\, {\mathrm d}\,}{\, {\mathrm d}t \,} \mathbb E [ { \mathcal M}_{t} ] \, =\,  - \frac{\,1\,}{\,2 \,} \partial_{x} p(t, 0) \, , 
\end{equation}
and we have the nonlinearity term in (\ref{eq: FP0}).

\section{Proof of Theorem \ref{thm:characterize-statio-p-e0}}
\label{sec:proofThmCharac}

\begin{proof}
{\bf (i) } Let us assume that $a>0$.
 Without the Dirac term, the ODE \eqref{eq:statioFP-eq}--\eqref{eq:statioFP-dot-et} has the following set of solutions~:
 \be\label{general_solution}
 x \mapsto e^{-a(x-x_0)^2-2x_0e_0(x-x_0)} \left( A + B \int_0^x  e^{a(y-x_0)^2+2x_0e_0(y-x_0)} dy \right) ,
 \ee
 for some constants $A$ and $B$.
 
Hence the solution to the ODE \eqref{eq:statioFP-eq}--\eqref{eq:statioFP-dot-et} must have this form before and after $x_0$ with a change of regime at the point $x_0$. We solve separately on each interval $[0,x_0)$ and $(x_0,\infty)$, and look for a continuous density function $p$.

\noindent {\bf On the interval $[0,x_0)$}~:
Take a solution of the form of \reff{general_solution}. The boundary condition $p(0)=0$ implies $A=0$. The boundary condition $p'(0)=2e_0$ implies $B=2 e_0$. Hence the solution is~:
\be\label{sol_before_x0}
	p(x) 
	&=& 2 e_0 e^{-a(x-x_0)^2-2x_0e_0(x-x_0)}  \int_0^x  e^{a(y-x_0)^2+2x_0e_0(y-x_0)} dy\\
	&=&   2 e_0 \int_0^x  e^{ay^2 + 2 x_0 (e_0-a) (y-x) -ax^2} dy.
\ee

\noindent {\bf On the interval $(x_0,\infty)$}~:
 Take a solution of the form of \reff{general_solution}. If $p$ is continuous, then we must have~:
 \b*
	 p(x) = e^{-a(x-x_0)^2-2x_0e_0(x-x_0)} \left( p(x_0) + B   \int_{x_0}^x  e^{a(y-x_0)^2+2x_0e_0(y-x_0)} dy \right),
 \e*
 where $p(x_0)$ is given by \reff{sol_before_x0} at point $x_0$. Now the dynamics \eqref{eq:statioFP-eq}--\eqref{eq:statioFP-dot-et} indicates formally that~:
 \b*
	 -\frac{p'(x_0+)-p'(x_0-)}{2} &=& e_0.
  \e*
  Hence, according to \reff{sol_before_x0}, we have~:
  \b*
	  - 2 x_0 e_0 p(x_0) + B &=& p'(x_0+) = - 2 e_0 + p'(x_0-) =  - 2 e_0 -2 x_0 e_0 p(x_0) + 2e_0.
  \e*
 Therefore $B=0$ and $\mathcal X_{\cdot}$ has a gaussian distribution on the left side of $x_0$, i.e.~:
 \b*
	 p(x) = p(x_0) e^{-a(x-x_0)^2-2x_0e_0(x-x_0)}. 
 \e*

\noindent {\bf Global solution}~:
 Combining the results on the two intervals, we obtain~\eqref{eq:FPstationary-solutionp}.
 
\noindent {\bf Characterization of $e_0$}~: 
 It remains to derive $e_0$ which satisfies the two relations~:
 \begin{equation}
 \label{eq:relations-e0}
	 \int_0^\infty x p(x) dx \;=\; x_0, \qquad \mbox{and} \qquad   \int_0^\infty  p(x) dx \;=\; 1.
 \end{equation}
 
 Let us assume that $\int_0^\infty p(x) dx=1$. Then, one can check by direct computation that~:
 \b*
 \int_{0}^{\infty} x p(x) dx 
 \,=\, - \frac{x_0}{a} \left(e_0- a\right)\int_{0}^{\infty} p(x) dx  + \frac{2 e_0}{\tau  ^2} x_0 \,=\, x_0.
 \e*
  Hence $e_0$ is determined by the relation $\int_0^\infty p(y)dy=1$. Using the notation $\tau   = \sqrt{2a}$, this relation rewrites~:
 \b*
 \frac{1}{2e_0} &=&  \int_{x=0}^{x_0}  \int_{y=0}^x e^{-a x^2 - 2 x_0 (e_0-a) (x-y) + a y^2} dy dx  +  \int_{x=x_0}^{\infty}  \int_{y=0}^{x_0} e^{-a x^2 - 2 x_0 (e_0-a) (x-y) + a y^2} dy dx\\
&=& \int_{y=0}^{x_0}  \int_{x=y}^{x_0} e^{-a x^2 - 2 x_0 (e_0-a) (x-y) + a y^2} dx dy  + \int_{y=0}^{x_0} \int_{x=x_0}^{\infty}   e^{-a x^2 - 2 x_0 (e_0-a) (x-y) + a y^2} dx dy\\ 
&=& \int_{y=0}^{x_0} e^{a y^2 + 2 x_0 (e_0-a) y } \int_{x=y}^{\infty} e^{-a x^2 - 2 x_0 (e_0-a) x} dx dy\\
&=& \int_{y=0}^{x_0} e^{\frac{1}{2} \left( \tau   y + \frac{2 x_0}{\tau  } (e_0-a) \right)^2 } \int_{x=y}^{\infty} e^{-\frac{1}{2} \left( \tau   x + \frac{2 x_0}{\tau  } (e_0-a) \right)^2 }  dx dy\\
&=& \frac{1}{\tau  } \int_{y=\frac{2 x_0}{\tau  } (e_0-a)}^{\tau   x_0 + \frac{2 x_0}{\tau  } (e_0-a)} e^{\frac{y^2}{2} } \int_{x=\frac{1}{\tau  } (y- \frac{2 x_0}{\tau  } (e_0-a))}^{\infty} e^{-\frac{1}{2} \left( \tau   x + \frac{2 x_0}{\tau  } (e_0-a) \right)^2 }  dx dy\\
&=& \frac{1}{\tau  ^2} \int_{y=\frac{2 x_0e_0}{\tau  } - \tau   x_0}^{\frac{2 x_0e_0}{\tau  } }  e^{\frac{y^2}{2}} \int_{x=y}^{\infty} e^{-\frac{x^2}{2}}  dx dy.
 \e*
 Hence $e_0$ is given by~\eqref{eq:FPstationary-solutione0}. 

\noindent {\bf Uniqueness of $e_0$}~: To conclude the proof, let us show that there is a unique $e_0$ satisfying~\eqref{eq:FPstationary-solutione0}.
Let us define the continuous functions~:
\[
F(u) \, =\,   \frac{1}{\tau  ^2} \int_{y=\frac{2 x_0 u }{\tau  } - \tau   x_0}^{\frac{2 x_0 u }{\tau  } }  e^{\frac{y^2}{2}} \int_{x=y}^{\infty} e^{-\frac{x^2}{2}}  dx dy \,  , \quad G(u) \, =\, \frac{1}{\, 2 u\, } ,  \quad 0 < u < \infty \,. 
\] 
and show that there exists $\, e_{0} \in (0,+\infty)\,$ such that $\, F(e_{0}) \, =\, G(e_{0})\,$.

One can check that for every $x > 0$,
\[
\frac{x}{ \, 1 + x^{2} \, } e^{-x^{2} / 2} \le \int^{\infty}_{x} e^{-y^{2}/ 2} {\mathrm d} y \le \frac{1}{ \, x \, } e^{-x^{2} / 2} \,.
\]
Using this inequality, we obtain the following lower and upper bounds of $\, F(u) \,$, for every $u> \tau  ^2/2$~: 
\begin{equation} \label{eq: lubound}
\frac{1}{\, 2 \tau  ^{2}\, } \log \Big( \frac{ \tau  ^{2} + 4 x_{0}^{2} u^{2}}{\, \tau  ^{2} + ( 2 x_{0} u - \tau  ^{2} x_{0} )^{2}\, }\Big) \, =\, \frac{1}{\tau  ^{2}} \int^{ \frac{2x_{0} u }{\tau  }}_{ \frac{2x_{0} u }{\tau  } - \tau   x_{0}} \frac{x}{\, 1+x^{2}\, } {\mathrm d} x 
\end{equation}
\[
\le F(u) \le  \frac{1}{\tau  ^{2}} \int^{ \frac{2x_{0} u }{\tau  }}_{ \frac{2x_{0} u }{\tau  } - \tau   x_{0}} \frac{1}{\, x\, } {\mathrm d} x \, =\,  \frac{1}{\, \tau  ^{2}\, } \log \Big( \frac{ 2 u}{\, 2 u - \tau  ^{2}} \Big).
\]
Thus $\, \lim_{u\to \infty}F (u) \, =\,  0\,$, and hence $\, \lim_{u\to \infty} (F(u) - G(u)) \, =\, 0 \,$.  

Moreover, given $\, \tau   > 0$ and $x_{0} > 0 \,$, let us choose $\,t_{1} > 1\, / \, 2 \,$ such that~:
\begin{equation}\label{eq:property-t1}
\, 2 a x_{0}^{2} \, =\,  \tau  ^{2} x_{0}^{2} > \frac{ e^{1/t_{1}} - 1}{\,  4 t_{1}^{2} - (2 t_{1} - 1)^{2} e^{1/t_{1}}\, } > 0\,.  
\end{equation}

Then $\, F(u_{1}) - G(u_{1}) > 0\,$ at $\,u_{1} \, =\, t_{1} \tau  ^{2} > \tau  ^{2} \, / \, 2\,$, because  \eqref{eq: lubound} and \eqref{eq:property-t1} imply~:
\[
F(t_{1} \tau  ^{2}) \ge \frac{1}{\, 2 \tau  ^{2}\, } \log \Big( \frac{\, 1 + 4 t_{1}^{2} \tau  ^{2} x_{0}^{2}\, }{ 1+ (2 t_{1} - 1)^{2} \tau  ^{2} x_{0}^{2}} \Big) > \frac{1}{\, 2 t_{1} \tau  ^{2}\, } \, =\, G(t_{1} \tau  ^{2}) \, .   
\]
Combining this observation with the continuity of $\, F(\cdot) - G(\cdot) \,$ and the fact that $\, \lim_{u\to 0+} (F(u) - G(u)) \, =\, -\infty\,$, we claim the existence of $\, e_{0} \in (0, u_{1}) \,$ such that $\, F(e_{0}) \, =\,  G(e_{0}) \, =\,  2 \, / \, e_{0}\,$.  Indeed,
by differentiation of $\, F(\cdot)\,$ and by change of variables we obtain~:
\[
F^{\prime} (u) \, =\,  \frac{\, 2 x_{0}\, }{\tau  ^{3}} \Big[ e^{ \frac{c^{2}_{1}}{2}} \int^{\infty}_{ c_{1}} e^{-x^{2}/2} {\mathrm d} x - e^{ \frac{(c_{1}-c_{2})^{2}}{2}} \int^{\infty}_{ c_{1} -c_{2} } e^{-x^{2}/2} {\mathrm d} x \Big] \Big \vert _{\{c_{1} = 2 x_{0} u / \tau   \, ,\, \,  c_{2} = \tau   x_{0}\}} 
\]
\[
\, =\,  \frac{\, 2 x_{0}\, }{\tau  ^{3}} \Big[ \int^{\infty}_{0} \Big( \frac{v}{\,(c_{1}^{2} + v^{2})^{1/2}\,}  -  \frac{v}{\, [(c_{1}-c_{2})^{2} + v^{2}]^{1/2}\,}  \Big )  e^{-v^{2}/2} {\mathrm d} v \Big] \Big \vert _{\{ c_{1}= 2 x_{0} u / \tau   \, ,\, \,  c_{2} = \tau   x_{0}\}}
\]
and hence, $\, F^{\prime}(u) >  0 \,$ for $\, 0 < u < \tau  ^{2}\, /\,  4 \,$ and $\, F^{\prime}(u) \le 0 \,$ for $\, u \ge \tau  ^{2}\, /\,  4\,$. Thus the function $\, F(u) \,$ is unimodal and takes the unique maximum at $\, u \, =\, \tau  ^{2} \, / \, 4\,$. Since $\,u \mapsto G(u) \,$ is monotonically decreasing to zero, $\, F(\cdot) - G(\cdot) \,$ is also unimodal with $\, \lim_{u\to \infty} (F(u) - G(u)) \, =\, 0 \,$. Therefore, the solution $\, e_{0}\,$ to $\, F(u) \, =\, G(u) \,$ is at most one in $\, (0, 2 a t_{1})\,$. 

\medskip 

\noindent {\bf Upper bound for $\, e_{0}\,$ in (\ref{eq:FPstationary-solution-e0est})}. Let us evaluate an upper bound $\, 2 a t_{1} \,$ for $\, e_{0}\,$. The condition (\ref{eq:property-t1}) is equivalent to 
\[
4 t ^{2}_{1} ( 2 a x_{0}^{2} - e^{1/t_{1}} ) + 4 (2 a x_{0}^{2}) t_{1} + 1 - e^{1/t_{1}}( 1 + 2 a x_{0}^{2}) > 0 \, . 
\]
This holds if $\, 2 a x_{0}^{2} \ge e^{1/t_{1}}\,$ and if $\, 4 ( 2 a x_{0}^{2})t_{1}  + 1 - e^{2} ( 1 + 2 a x_{0}^{2} ) \ge 0 \,$, because $\, t_{1} > 1\, /\,  2\,$. Thus it suffices to have $\, t_{1} \ge 1 / (\log (2 a x_{0}^{2}) \,$ and $\, t _{1} \ge ( e^{2} ( 1 + 2 a x_{0}^{2}) - 1 ) / ( 4 a x_{0}^{2}) \,$. This implies (\ref{eq:FPstationary-solution-e0est}). 

\bigskip 

\noindent {\bf (ii) } Let us assume that $a=0$. The proof of~\eqref{eq:FPstationary-solutionp-a0} is very similar to the proof of~\eqref{eq:FPstationary-solutionp} so we do not reproduce it. As for the characterization of $e_0$, note that the second equality in~\eqref{eq:relations-e0} is always satisfied in this case. The first equality leads to~\eqref{eq:FPstationary-solutionp-a0}.
\end{proof}

\section{Proof of (\ref{eq: RENEW})} \label{Appendix: (2.31)}
With the renewal theory, we shall show the last inequality in (\ref{eq: RENEW}), namely, 
\begin{equation} \label{eq: RENEWa}
\sum_{k=1}^{\infty} \mathbb P ( \sup_{0 \le s \le t } (W_{s} + s)^{+} \ge k x_{0} ) \le \frac{\, t\, }{x_{0}} \, ; \quad t \ge 0 \, . 
\end{equation}
under the assumption $\, x_{0} \ge 1\,$. Let us denote by $\, \xi\,$ the first passage time of Brownian motion with constant drift for the level $\, x_{0}  (> 0)\,$, i.e, $\, \xi \, :=\, \inf \{ s > 0 : W_{s} + s \ge x_{0}\} \,$, and consider the sequence $\, \xi, \xi_{1}, \xi_{2}, \ldots \,$ of independent copies of $\, \xi\,$, the cumulative sum $\, S_{n} \, :=\, \xi_{1} + \cdots + \xi_{n}\,$ and the renewal process $\, N(t) \, :=\, \sum_{k=1}^{\infty} {\bf 1}_{\{S_{k} \le t\}}\,$, $\, t \ge 0\,$. The density function and the Laplace transform of $\, \xi\,$ is well known, e.g., section 3.5.C of \textsc{Karatzas \& Shreve}~\cite{MR1121940}. Then the right hand of (\ref{eq: RENEWa}) becomes 
\[
\sum_{k=1}^{\infty} \mathbb P ( \sup_{0 \le s \le t } (W_{s} + s)^{+} \ge k x_{0} ) \, =\,  \sum_{k=1}^{\infty} \mathbb E[ {\bf 1}_{\{S_{k} \le t\}} ] \, =\,  \mathbb E[ N(t) ] \, . 
\]
which satisfies the renewal equation 
\begin{equation} \label{eq: RENEWc}
m(t) \, :=\, \mathbb E[ N(t) ] \, =\,  \mathbb P ( \xi \le t )  + \int^{t}_{0} \mathbb E[ N(s)] \mathbb P ( \xi \in {\mathrm d}  s) \, ; \quad t \ge 0 \, .  
\end{equation}

Hence applying the Laplace transforms  
\[
\widehat{ m }(\theta) \, :=\,   \int^{\infty}_{0} e^{- \theta t} m(t) {\mathrm d} t \,  \, , \quad \widehat{f}(\theta) \, :=\,  \mathbb E[ e^{-\theta \xi} ] \, =\,  \exp \big( x_{0} (1 - \sqrt{ 2 \theta + 1 \, }) \big) \, ;  \quad \theta > 0 \, 
\]
to the renewal equation (\ref{eq: RENEWc}), we solve it in terms of  Laplace transforms : 
\[
\widehat{m}(\theta) \, =\, \frac{ \widehat{f}(\theta) }{\, \theta ( 1 - \widehat{f}(\theta) ) \, } \, ; \quad \theta > 0 \, . 
\]
Note that the Laplace transform of $\, \ell(t) \, :=\, t \, / \, x_{0}\,$, $\, t \ge 0 \,$ is 
\[
\widehat{\ell}(\theta) \, =\,  \int^{\infty}_{0} e^{-\theta t } \ell (t) {\mathrm d} t \, =\,  \frac{1}{\, x_{0}\, \theta^{2}\, } \, ; \quad \theta > 0 \, . 
\]

If $\, x_{0} \ge 1 \,$, by direct calculations we verify that $\, \widehat{\ell} (\cdot) - \widehat{m}(\cdot) \, =\, \widehat{ \ell - m }(\cdot) \,$ is completely monotone: 
\begin{equation} \label{eq: RENEWb}
(-1)^{k} \frac{ {\mathrm d}^{k} }{\, {\mathrm d} \theta^{k}\, }  [ \widehat{\ell - m}] (\theta ) \ge 0 \, ; \quad k \in \mathbb N_{0} \, , \, \, \theta > 0 \, . 
\end{equation}
Note that if $\, x_{0} < 1\,$, then (\ref{eq: RENEWb}) does not hold for some $\, \theta\,$ and some $\,k \in \mathbb N_{0}\,$. Thus by \textsc{Post}'s inversion formula (Theorem XIII.10.3 of \textsc{Feller}~\cite{MR0228020} and also see Theorem 1.4 of \textsc{Schilling, Song \& Vondracek}~\cite{MR2978140}) of Laplace transforms with (\ref{eq: RENEWb}) we conclude that if $\,x_{0} \ge 1\,$, then 
\[
\sum_{k=1}^{\infty} \mathbb P ( \sup_{0 \le s \le t } (W_{s} + s)^{+} \ge k x_{0} )  \, =\, m(t) \, =\,  \lim_{k\to \infty} \frac{\, (-1)^{k}\, }{\, k!\, }  \Big( \frac{k}{\, t \, } \Big)^{k+1} \frac{ {\mathrm d}^{k} \widehat{m} }{ {\mathrm d} \theta^{k} } (\theta) \Big \vert_{\theta \, =\, k\, /\, t}  
\]
\[
\le 
 \lim_{k\to \infty} \frac{\, (-1)^{k}\, }{\, k!\, }  \Big( \frac{k}{\, t \, } \Big)^{k+1} \frac{ {\mathrm d}^{k}  \widehat{\ell}}{ {\mathrm d} \theta^{k} } (\theta) \Big \vert_{\theta \, =\, k\, /\, t} \, =\,  \ell (t) \, =\,  \frac{t}{\, x_{0}\, } \, ; \quad t \ge 0 \, . 
\]
This is (\ref{eq: RENEWa}), which completes the proof of (\ref{eq: RENEW}).

\end{document}